\documentclass[11pt,oneside]{amsart}
\usepackage{bbm}
\usepackage{mathrsfs}

\usepackage[english]{babel}
\usepackage[utf8]{inputenc}
\usepackage[T1]{fontenc}

\usepackage[top=3cm,bottom=2cm,left=3cm,right=3cm,marginparwidth=1.75cm]{geometry}
\usepackage{physics}
\usepackage{amsmath,amssymb}
\usepackage{esint}
\usepackage{graphicx}
\usepackage[dvipsnames]{xcolor}
\usepackage[colorlinks=true,allcolors=BlueViolet, pagebackref=true]{hyperref}
\renewcommand*\backref[1]{\ifx#1\relax \else (Cited on p.#1) \fi}
\hypersetup{hypertexnames=false}
\usepackage{cleveref}

\usepackage[shortlabels]{enumitem}
\usepackage[all,cmtip]{xy}
\usepackage{relsize}
\usepackage{subcaption}
\usepackage{dsfont}

\usepackage{import}
\graphicspath{{Figures/}}

\usepackage{xspace} 
\usepackage{cancel}
\usepackage{fancyvrb}

\DeclareFontFamily{U}{BOONDOX-calo}{\skewchar\font=45 }
\DeclareFontShape{U}{BOONDOX-calo}{m}{n}{
  <-> s*[1.05] BOONDOX-r-calo}{}
\DeclareFontShape{U}{BOONDOX-calo}{b}{n}{
  <-> s*[1.05] BOONDOX-b-calo}{}
\DeclareMathAlphabet{\mathbdx}{U}{BOONDOX-calo}{m}{n}
\SetMathAlphabet{\mathbdx}{bold}{U}{BOONDOX-calo}{b}{n}
\DeclareMathAlphabet{\mathbbdx}{U}{BOONDOX-calo}{b}{n}

\newtheorem{theorem}{Theorem}[subsection]
\newtheorem{proposition}[theorem]{Proposition}
\newtheorem{corollary}[theorem]{Corollary}
\newtheorem{lemma}[theorem]{Lemma}

\newtheorem{yauconjecture}{Yau's conjecture}
\newtheorem{probyauconjecture}{The Probabilistic Yau's conjecture}
\newtheorem{definition}[theorem]{Definition}

\newtheorem{mainthm}{Theorem}

\theoremstyle{remark} 

\newtheorem{remark}[theorem]{Remark}

\crefname{equation}{Equation}{Equations}
\crefname{gather}{Equation}{Equations}
\crefname{multline}{Equation}{Equations}
\crefname{figure}{Figure}{Figures}
\crefname{question}{Question}{Question}
\crefname{section}{Section}{Sections}
\crefname{subsection}{Subsection}{Subsections}
\crefname{appendix}{Appendix}{Appendices}
\crefname{lemma}{Lemma}{Lemmas}
\crefname{proposition}{Proposition}{Propositions}
\crefname{theorem}{Theorem}{Theorems}
\crefname{innercustomthm}{Theorem}{Theorems}
\crefname{mainthm}{Theorem}{Theorems}
\crefname{corollary}{Corollary}{Corollaries}
\crefname{definition}{Definition}{Definitions}
\crefname{remark}{Remark}{Remarks}
\crefname{proposition}{Proposition}{Proposition}
\crefname{corollary}{Corollary}{Corollaries}
\crefname{example}{Example}{Examples}
\crefname{claim}{Claim}{Claim}
\crefname{conjecture}{Conjecture}{Conjecture}
\crefname{yauconjecture}{Yau's conjecture}{Yau's conjecture}

\definecolor{bluola}{RGB}{138,43,226}

\newcommand{\R}{\mathbb{R}}

\newcommand{\N}{\mathbb{N}}
\newcommand{\PP}{\mathbb{P}}

\newcommand{\E}{\mathbb{E}}

\newcommand{\de}{\partial}

\newcommand{\inter}[1]{%
  {\kern0pt#1}^{\mathrm{o}}%
}

\newcommand{\f}{\varphi}
\renewcommand{\a}{\alpha}
\newcommand{\e}{\varepsilon}

\renewcommand{\t}{\tau}

\newcommand{\vol}[1]{\mathrm{Vol}^{#1}}

\newcommand{\n}{n}
\renewcommand{\d}{\n}

\newcommand{\Var}{{\mathbb{V}\mathrm{ar}}}
\newcommand{\Cov}{\mathrm{Cov}}
\newcommand{\coeff}{\Theta}
\newcommand{\lf}{\mathcal{L}_f}

\newcommand{\lfl}{\mathcal{L}_{f_\ell}}

\newcommand{\berry}{\mathbdx{b}}
\newcommand{\horm}{\mathbdx{h}}
\newcommand{\lI}{\{i:\ \lambda_i\in I\}}

\newcommand{\B}{\mathbbm{B}}

\renewcommand{\S}{S}

\def\randin{%
  \mathchoice%
    {\raisebox{-.35ex}{$\displaystyle{^\subset}$}\mkern-11.5mu\raisebox{+.45ex}{$\displaystyle{_\subset}$}}
    {\mkern+1mu\raisebox{-.27ex}{$\textstyle{^\subset}$}\mkern-11.7mu\raisebox{+.45ex}{$\textstyle{_\subset}$}}
    {\raisebox{.35ex}{$\scriptstyle\subset$}\mkern-14mu\raisebox{-.15ex}{$\scriptstyle\subset$}}
    {\raisebox{.3ex}{$\scriptscriptstyle\subset$}\mkern-13.5mu\raisebox{-.10ex}{$\scriptscriptstyle\subset$}}
}

\makeatletter
\newcommand{\tpitchfork}{%
  \raise-0.1ex\vbox{
    \baselineskip\z@skip
    \lineskip-.52ex
    \lineskiplimit\maxdimen
    \m@th
    \ialign{##\crcr\hidewidth\smash{$-$}\hidewidth\crcr$\pitchfork$\crcr}
  }%
}
\newcommand{\Oh}{\mathrm{O}}
\newcommand{\oh}{\mathrm{o}}
\newcommand{\iid}{\emph{i.i.d.}\xspace}
\newcommand{\MRW}{\emph{MRW}\xspace}
\newcommand{\CRW}{\emph{CRW}\xspace}
\newcommand{\RW}{\emph{RW}\xspace}
\newcommand{\wMRW}{\emph{wMRW}\xspace}
\newcommand{\mC}{\mathcal{C}}
\newcommand{\m}[1]{\mathcal{#1}}
\newcommand{\be}{\begin{equation}}
\newcommand{\ee}{\end{equation}}
\numberwithin{equation}{section}

\newcommand{\bega}{\begin{equation}\begin{aligned}}
\newcommand{\eega}{\end{aligned}\end{equation}}

\newcommand{\begt}{\begin{equation}\begin{gathered}}
\newcommand{\eegt}{\end{gathered}\end{equation}}
\newcommand{\kop}{\left\{}
\newcommand{\pok}{\right\}}
\newcommand{\tyu}{\left(}
\newcommand{\uyt}{\right)}
\newcommand{\qwe}{\left[}
\newcommand{\ewq}{\right]}

\newcommand{\dist}[2]{\mathrm{dist}\!\tyu #1 , #2\uyt}

\title[Statistics on Yau's conjecture]{
Statistics on Yau's conjecture: Variance asymptotics
}
\author{Michele Stecconi, Anna Paola Todino}
\date{\today}
\begin{document}

\begin{abstract}
We investigate the probabilistic counterpart of Yau's conjecture on the
nodal volume of Laplace eigenfunctions on compact manifolds, by
studying the high-frequency variance asymptotics of Riemannian random
waves. \cref{thm:ape} gives quantitative bounds on the fluctuations of their
nodal volume for all regimes of spectral windows, including the
monochromatic one, of spectral size~$1$. For every non-monochromatic
window the variance has the exact order $1/N$, where $N$ is the
dimension of the underlying space of eigenfunctions, on any compact
manifold. In the monochromatic regime, on manifolds without conjugate
points and in particular on negatively curved ones, we prove that the
variance is $o(1/N)$: this improves the previous best bound, cf. 
\cite{CH20}, to its square divided by a logarithmic factor, and exhibits a
form of Berry's cancellation, relative to the natural scale $1/N$, for
monochromatic random waves on such chaotic manifolds. Our proofs rely
on a local and global analysis combining the Kac--Rice formula, the
Wiener--It\^o chaos decomposition recently obtained by the authors
\cite{cgv2025StecconiTodino}, and a sharp analysis of the error in the
pointwise Weyl law associated to an arbitrary spectral window
(\cref{thm:OdeltaRWintro}). We also introduce a general machinery (\cref{thm:drago}) that
yields variance decay for the nodal volume of an arbitrary Gaussian
field under broad geometric conditions, subject to correlation decay
assumptions.
\end{abstract}
\maketitle

\section*{Acknowledgments} 
M. Stecconi was supported by the  Luxembourg National Research Fund (Grant code: O24/18972745/GFRF). A.P. Todino has been partially funded by INdAM-GNAMPA project (CUP:E53C25002010001). We are grateful to Louis Gass, Boris Hanin, Giovanni Peccati, and Domenico Marinucci for useful discussions.
\setcounter{tocdepth}{2}
\tableofcontents 

\section{Conventions and notation}\label{sec:notations}

The following list contains some recurring conventions adopted in our work.
\begin{enumerate}[(i)]


\item A \emph{random element} (see \cite{Billingsley}) of the topological space $T$ (or \emph{with values} in $T$) is a measurable mapping $X\colon \Omega\to T$, defined on a probability space $\tyu \Omega,\mathscr{E},\PP \uyt$. In this case, we write
    \be\label{eq:randin}
    X\randin T
    \ee 
    and denote by $[X]=\PP X^{-1}$ the (push-forward) Borel probability measure on $T$ induced by $X$. We will use the notation
\be 
\PP\{X\in U\}:=
\PP X^{-1}(U)
\ee 
to indicate the probability that $X\in U$, for some Borel measurable subset $U\subset T$, and write (as usual)
\be 
\E\{f(X)\}:=\int_{T}f(t)d[X](t),
\ee
to denote the expectation of the random variable $f(X)$, where $f\colon T\to \R^k$ is a measurable mapping such that the above integral is well-defined.
We will say that $X$ is a \emph{random variable}, a \emph{random vector} or a \emph{random field}, respectively, when $T$ is the real line, a vector space, or a space of functions, 
respectively. 
\item The sentence: ``$X$ has the property $\mathcal{P}$ almost surely'' (abbreviated ``a.s.'') means that the set  $S=\{t\in T : t \text{ has the property }\mathcal{P}\}$ contains a Borel set of $[X]$-measure $1$. It follows, in particular, that the set $S$ is $[X]$-measurable, i.e. it belongs to the $\sigma$-algebra obtained from the completion of the measure space $(T,\mathcal{B}(T),[X])$.
\item \emph{Notation for measures:}\label{subsec:notameas}
Given a measure $\mu$ on a measurable space $\Omega$, we denote the measure of a measurable subset $A$ and, when defined, the integral of a function $\alpha\colon \Omega\to \R$ as 
\be 
\mu(A)=\int_A\mu(dx) \quad \text{and} \quad \langle \mu, \alpha\rangle=\int_\Omega\alpha(x)\mu (dx),
\ee
respectively. We will sometimes formally identify the measure with its density $\mu(dx)$, interpreted in the distributional sense. Moreover, when $\mu=\vol{n}$ is the $n$-Hausdorff measure of a Riemannian manifold $(\Omega=M,g)$ of dimension $\dim(M)=\n$, we will just write
\be 
d x\equiv \vol{\n}(d x). 
\ee
Finally, if $\mu(A)\neq 0$, then we may write 
\be 
\fint_A \a(x)\mu(d x):=\frac{1}{\mu(A)}\int_A\a(x)\mu(d x).
\ee
\item We will denote the $n-1$-volume of the round unit sphere $\S^{n-1}=\kop x \in \R^{n}\colon |x|=1\pok$ as $s_{n-1}=\frac{2\pi^{\frac{n}2}}{\Gamma\tyu\frac{n}{2}\uyt}$ and the volume of the standard unit ball $\B^n= \kop  x \in \R^{n}\colon |x|\le 1\pok$ as $b_n=\frac{s_{n-1}}{n}$.
\end{enumerate}

\section{Introduction}
\subsection{General setting}
We will take up the language and notations of \cite{cgv2025StecconiTodino}, in which we studied the nodal volume of a general Gaussian field on a Riemannian manifold. In this paper, we will focus on a special class of Gaussian fields, called \emph{Riemannian random Waves} (abbrv. \RW). Let us briefly recall some of the main objects and notations, for a more detailed description we refer to \cite{cgv2025StecconiTodino}.
\subsection{Gaussian nodal volumes}
Let $(M,g)$ be a $n$-dimensional Riemannian manifold, possibly with boundary. For all $r\in \N$ or $r=\infty$, we denote by $\mC^r(M)$ the space of real valued functions on $M$ of class $\mC^r$. 
The \emph{nodal volume measure} of $\phi\in \mC^0(M)$ is the mapping 
\be \label{eq:vol1}
A\mapsto \m L_\phi(A)=\vol{\d-1}(\phi^{-1}(0)\cap A)=\int_A\m L_\phi(dx),
\ee 
defined for $A\subset M$ Borel subsets. 
That is, the restriction of the $(n-1)$-Hausdorff measure of $(M,g)$ to the \emph{nodal set} $\phi^{-1}(0)$ of $\phi$. The \emph{nodal volume} of $\phi$ is the real number $\m L_\phi(M)$. 
We will consider \eqref{eq:vol1} for $\phi$ a smooth Gaussian random field on $M$, in the sense of \cite{AT07,AzaisWscheborbook,bogachev}, i.e., a random function that admits a representation as
\be \label{eq:karlo}
\phi=\sum_{i=0}^{N}\gamma_i\f_i,
\ee
for some collection of functions $\f_i\in \mC^\infty(M)$, $N\in \N\cup\{+\infty\}$, and a family of independent and identically distributed $\gamma_i\sim \mathcal{N}(0,1)$ (Karhunen-Lo\`{e}ve expansion \cite{bogachev}).

 Then, $\phi \randin \mC^\infty(M)$ (see \cref{sec:notations}) defines a probability space $\tyu \mC^\infty(M), \mathscr{B},\PP \uyt$, with $\mathscr{B}$ being the Borel $\sigma$-algebra. 
Here, $N \in \N$, or $N=+\infty$; where it is intended that if $N=+\infty$, the series \eqref{eq:karlo} is almost surely convergent in $\mC^\infty(M)$.
\subsection{Yau's conjecture}
On a compact Riemannian manifold $M$ without boundary, the eigenvalues of the negative Laplace-Beltrami operator $-\Delta$ are all positive and can be ordered as an increasing diverging sequence: $\lambda_0^2=0< \lambda_1^2\le \lambda_2^2\le \dots \le \lambda_i^2\to +\infty$, possibly with repetitions. The corresponding eigenfunctions can be selected so that they form a complete orthonormal system $(\f_i)_{i\in \N}$ in $L^2(M)$. 
Yau's conjecture states that the nodal volume of an eigenfunction $\f$, such that $\Delta \f =-\lambda^2\f$ is of the order of the corresponding frequency $\lambda$. 
\begin{yauconjecture}\label{yauconj} On any compact Riemannian manifold $(M,g)$ with $\de M=\emptyset$,
there exist constants $c_1(M,g),c_2(M,g)>0$ such that if $\Delta \f =-\lambda^2 \f$, then
\be\label{eq:yau} 
c_1(M,g)\le \frac{\mathcal{L}_{\f}(M)}{\lambda} \le c_2(M,g),
\ee
where we recall Equation \eqref{eq:vol1}.
\end{yauconjecture}
The conjecture was originally formulated by Yau in \cite[Problem 74]{yauconjecture} for $n=2$, who then extended the question to arbitrary dimension in \cite{Yau_1990_Open_problems_in_geometry}.
It was proven for analytic manifolds by Donnelly and Fefferman \cite{Yau_1988_DonnellyFefferman}; in the smooth case, the state of the art is represented in the results of Logunov and Malinnikova \cite{Yau_2018_LogunovMalinnikova,Yau_2018a_Logunov,Yau_2018b_Logunov}, who proved the validity of the lower bound on arbitrary manifolds, and obtained a polynomial upper bound of ${\mathcal{L}_{\f}(M)}/{\lambda}$  of order $O(\lambda^{\a})$, for some $\alpha > 0$. In particular, if $n=2$, they show \cite{Yau_2018_LogunovMalinnikova} that the bound holds for some $\a=\frac{1}2-\e$.\footnote{In comparing with the literature, one has to pay attention to the fact that in\cref{eq:yau}, $\lambda$ denotes the square root of the eigenvalue and that the volume is already divided by $\lambda$.} We refer to the excellent surveys \cite{logunovsurvey,Ingremeau_yau} and the references therein for more details.
Recently, three preprints produced with the assistance of AI systems have appeared, claiming that the conjecture
holds in dimension $2$ \cite{openAIsharp}, while it fails in dimensions
$3,4$ \cite{openAIcounterxample} and $5$ \cite{openAIpower}, in the sense
that the upper bound does not hold. At the time of writing, none of them
has been peer reviewed.

In this paper we are interested in exploring the probabilistic point of view.

\subsection{Riemannian Random Waves}
\begin{definition}\label{def:RRW}
Given a closed Riemannian manifold $(M,g)$ and a bounded interval $I\subset \R$, we call a \emph{Riemannian Random Wave} any Gaussian field on $M$ of the form
\be \label{eq:RW}
\phi_I^M(x)=\sum_{\lI}\gamma_i \f_i(x),\qquad \Delta \f_i=-\lambda_i^2\f_i
\ee
with $\gamma_i\sim \m N(0,1)$ \iid and where $(\f_i)_{i\in\N}$ is an $L^2(M)$-orthonormal basis of eigenfunctions of $-\Delta$, ordered so that the eigenvalues are non-decreasing. We will denote
the covariance function of $\phi^M_{I}$, as:
\be\label{eq:E}
E^M_{I}(x,y):=\E\kop \phi_{I}^M(x)\phi_{I}^M(y)\pok =\sum_{\{i :\lambda_i \in I\} }\f_i(x)\f_i(y).
\ee
In case $M$ has boundary, we will discuss the boundary conditions case-by-case. In particular, for $M\subset \R^n$, we define $\phi_I^{M}$ as the restriction of the stationary Gaussian field $\phi_I^{\R^n}$ on $\R^n$ with covariance function
\be \label{eq:defRWberry}
E^{\R^n}_{I}(x,x+u):=E^n_{I}(|u|):=\kop \begin{aligned}
&\frac{1}{(2\pi)^n}\int_{I}\int_{\lambda\S^{n-1}} e^{i\langle u, \xi \rangle} \dd \xi \dd \lambda,
\quad &\text{ if $I$ is an interval, }
\\
& \frac{1}{(2\pi)^n}\int_{\ell \S^{n-1}} e^{i\langle u , \xi \rangle} \dd \xi, \quad &\text{ if $I=\ell$ is a real number.
}
\end{aligned}\right.\footnote{In fact, it is easy to see that the imaginary part of this integral vanishes, so that 
\begin{equation*} \label{eq:Eeig}
\E\kop e^{ig_x\langle u-v,\xi\rangle }\pok=\E\kop \cos\tyu g_x\langle u-v,\xi\rangle\uyt\pok.
\end{equation*}}
\ee
In particular, $\phi_{1}^{\R^n}$ is known as Berry's random field on $\R^n$.
\end{definition}
\begin{remark}
Let us comment on  \cref{def:RRW}. It
is a special case of $\phi^{\R^n}_{\mu}$, where $\mu$ is any finite measure on $\R_+$.
\be \label{eq:defRWberrymu}
\E \kop \phi_\mu^{\R^n}(x)\phi_\mu^{\R^n}(x+u) \pok:=\frac{1}{(2\pi)^n}\int_{\R_+}\tyu \int_{\lambda\S^{n-1}} e^{i\langle u, \xi \rangle} \dd \xi \uyt \mu\tyu \dd \lambda\uyt
\ee
Clearly, the case of an interval $I$ corresponds to the Lebesgue measure on $I$, while the case of a number $\ell$ corresponds to the delta measure $\mu=\delta_\ell$. This language would be the most appropriate for a more general treatment as there is continuity with respect to $\mu$, however, we chose here to use a lighter notation. The only other measure $\mu$ that will be of some importance in the technical part of this paper (see \cref{def:RWUI}) is the uniform probability measure on $I$, denoted $\mu=\mathsf U(I)$, corresponding to the random field 
\be \label{eq:phiU}
\phi_{\mathsf U(I)}^{\R^n}=|I|^{-\frac12}\phi_{I}^{\R^n} \xrightarrow[I\to \{\ell\}]{}\phi_{\mathsf U(\{\ell\})}^{\R^n}=\phi_{\ell}^{\R^n};
\ee
while with our conventions $\phi_{\{\ell\}}^{\R^n}=\phi_{[\ell,\ell]}^{\R^n}=0$.
\end{remark}
\begin{remark}
    The constant $(2\pi)^{-n}$ in \cref{def:RRW} is carefully chosen. It is the only one for which, as $\ell\to +\infty$, one can directly compare:
\be\label{eq:const} 
\quad 
\phi^M_{[0,\ell]}\tyu x+ u\uyt \approx \phi^{\R^n}_{[0,\ell]}(u)=\ell^{\frac{n}2}\phi_{[0,1]}^{\R^n}(\ell u),
\quad
\phi^M_{[\ell-1,\ell]}\tyu x+ u\uyt \approx \phi^{\R^n}_{\ell}(u)=\ell^{\frac{n-1}2}\phi_{1}^{\R^n}(\ell u).
\ee
Such a comparison can be made precise, and we will use it in the form of \cite[Thm. 1.3]{Keeler}, a refinement of H\"ormander's pointwise Weyl's law \cref{thm:hoermplaw}, see also \cref{def:scalimRW} below.
Indeed, whatever other choice would require to correct the r.h.s. of the above asymptotics with a multiplicative constant. 
\end{remark}




Such fields and terminology were introduced by Zelditch in \cite{Zel09}, although the notation we are using is new.
Let us immediately fix two parameters associated with $\phi_I^M$, which will be of fundamental importance later on. 
\begin{definition}\label{def:lambsigm_intro} For $M$ and $I$ as in \cref{def:RRW}, we define
\begin{gather}\label{eq:sigmalambdaRW}
\sigma(\phi_I^M)^2:=\fint_M\Var\kop \phi_I^M(x)\pok dx,
\quad
\lambda(\phi_I^M)^2:=\frac{1}{\sigma(\phi_I^M)^2}\fint_M\E\kop \|d_x\phi_I\|^2 \pok dx.
\end{gather}
and call them \emph{average variance} and \emph{average frequency} (consistently with the terminology introduced in \cite[Def. 1.11]{cgv2025StecconiTodino}). If $\sigma(\phi^M_I)=0$, we set $\lambda(\phi^M_I)=0$. We extend such notation to any setting in which the integrands in \cref{eq:sigmalambdaRW} are constant, including non-compact manifolds $M$.
\end{definition}
If $M$ is closed, we can alternatively express the average variance $\sigma(\phi^M_I)$ in terms of the dimension of the eigenspace:
\begin{gather}\label{eq:phiI} 
N_I^M:=\dim \tyu \bigoplus_{ \lambda\in I}\ker(\Delta+\lambda^2)\uyt=\sum_{\lI}1,
\\
\text{then,}
\quad 
\sigma(\phi_I^M)^2
=\frac{N_I^M}{\vol{n}(M)}
\quad \text{and} \quad
\lambda(\phi_I^M)^2
=\frac1{N_I^M}\sum_{\lambda_i\in I}\lambda_i^2.
\end{gather}
From this is clear that $\lambda(\phi_I)^2$ is the average eigenvalue in $I$, indeed it corresponds to the \emph{average frequency} defined in \cite{cgv2025StecconiTodino} together with the \emph{average variance} $\sigma(\phi_I)$ and the \emph{maximal eccentricity} $\e(\phi_I)$, which we will discuss later.
For $M=r\B^n$, the average variance and the average frequency of $\phi^{\R^n}_I$ are constant values
that can be easily computed from \cref{eq:defRWberry}; in such case,
the dimension of the corresponding eigenspace is infinite, but we will adopt the convention that 
\be N^{r\B^n}_{I}:=\Var\kop \phi_I^{\R^n}(0)\pok\vol{n}(r\B^n) .
\ee
This is coherent with Weyl's law, originally proven in \cite{WeylLaw_1912}.
\begin{theorem}[Weyl's law]\label{thm:weyllaw}
\be \label{eq:Weyllaw}
N_{[0,\ell]}^M = \ell^n\tyu \frac{b_n}{(2\pi)^{n}}\vol{n}(M)+\Oh (\ell^{-1})\uyt.
\ee
\end{theorem}
The above can be seen as a consequence of \cite[Theorem 1.1]{hoerm1968} concerning the asymptotics of $E_{[0,\ell]}^M(x,y)$ for large $\ell$ and $x,y\in M$. This object is known as the \emph{spectral function} of $M$ (cf. \cite{hoerm1968}) and play a crucial role in this context, see for instance \cite{hoerm1968,canzani_hanin_2018_cinfscasymp,CH20,Keeler,Gass2020}. We extend the same definition to arbitrary intervals $I$.
\begin{theorem}[H\"ormander's pointwise Weyl's law, reported from \texorpdfstring{\cite[eq. (1.2)]{Keeler}}{}]\label{thm:hoermplaw}
For any $M$ compact Riemannian manifold, 
\be \label{eq:horlaw}
E^M_{[0,\ell]}(x,y) = \ell^n\tyu E_{[0,1]}^{n}(\ell\dist xy)+R_{[0,\ell]}^M(x,y)\uyt,
\ee
with remainder of order $R_{[0,\ell]}^M(x,y)=\Oh(\ell^{-1})$. Precisely, the estimate continues to hold for the derivatives in the following sense: for $\rho(M)$ small enough, for any $k\in \N$,
\be
\sup_{\mathrm{dist}(x,y)\le \rho(M)}\max_{|a|,|b|\le k} \frac{\| \nabla^a \nabla^b R_{[0,\ell]}^{M}(x,y)\|}{\ell^{|a|+|b|}}=\Oh(\ell^{-1}). 
\ee
In particular, it follows that: 
\be \label{eq:siglam_colorful}
\sigma(\phi_{[0,\ell]}^M)^2=\ell^{n}\tyu \frac{b_n}{(2\pi)^n}+\Oh(\ell^{-1})\uyt,\quad \lambda(\phi_{[0,\ell]}^M)^2=\ell^2\tyu\frac{n}{n+2}+\Oh(\ell^{-1})\uyt.
\ee
\end{theorem}
\begin{remark}
By construction one has the following identities, valid for any interval $I$: \be\label{eq:NEEsig} N_{I}^M=\int_M E^M_{I}(x,x) dx=\int_{M\times M}| E^M_{I}(x,y)|^2 dxdy=\sigma(\phi_{I}^M)^2\vol{n}(M)
.\ee 
Therefore, by considering $I=[0,\ell]$, one can prove \cref{thm:weyllaw} as a straightforward corollary of \cref{thm:hoermplaw}.
\end{remark}
The above theorem effectively implies that the restriction of $\phi^M_{[0,\ell]}$ to a Riemannian ball of radius $\ell^{-1}$, is asymptotically equivalent to  $\phi^{\R^n}_{[0,\ell]}$. Since $\ell$ is approximately the average frequency $\lambda(\phi^M_{[0,\ell]})$, that is, the square-root of the average eigenvalue $-\lambda(\phi^M_{[0,\ell]})^2$, the distance value $\ell^{-1}$ is commonly called the \emph{wave-length scale}. At a heuristic level, such an asymptotic comparison is the natural ansatz for all sequences of intervals high frequency intervals $I$, although this is either not true or not proven in full generality (see \cref{thm:OdeltaRWintro}). This is perfectly coherent with \cref{yauconj}, as it suggests that the nodal volume of $\phi^M_I$, should behave like $\lambda(\phi^M_I)^{-n}$ at wavelength scale $\lambda(\phi^M_I)^{-1}$, therefore by standard dimensional arguments, it should globally behave like $\lambda(\phi^M_I)$.
\subsection{Probabilistic Yau's conjecture}
The object of this paper is an investigation of the probabilistic perspective on Yau's conjecture, guided by the following---intentionally vaguely formulated---questions.
\begin{probyauconjecture}\label{probYauconj}
Choose a compact Riemannian manifold $(M,g)$ and a bounded interval $I$ with average frequency $\lambda(\phi_I^M)$. 
Is there a Law Of Large Numbers
\be\label{eq:probyau}  \frac{\mathcal{L}_{\phi_{I}^M}(M)}{\lambda(\phi_I^M)}\approx \frac{s_{n-1}}{s_n\sqrt{n}} \vol{n}(M),
\ee
as $\lambda(\phi_I^M)\to +\infty$? What is the order of the variance? Is there a Central Limit Theorem?
\end{probyauconjecture}
(Recall \cref{eq:vol1} and \cref{eq:RW} and \cref{eq:phiI} for definitions.)
The r.h.s. in \cref{eq:probyau} can be easily justified, after computing the expectation of the l.h.s., either with the Kac-Rice formula or with the chaos decomposition, see \cite[Cor. (1.3) and Eq. (1.42)]{cgv2025StecconiTodino}.

In order to investigate the above question, we give the following definition.
\begin{definition}\label{def:Y} Given a compact $n$-dimensional Riemannian manifold $(M,g)$, and any $I\subset (0,+\infty)$ bounded, we define the random variable $Y_M(I)$ and the constant $c_n$ as
    \be 
Y_M(I):=\frac{1}{c_n\vol{n}(M)}\cdot \frac{\mathcal{L}_{\phi_{I}^M}(M)}{\lambda(\phi_I^M)}, \quad c_n:= \frac{s_{n-1}}{s_n\sqrt{n}}
    \ee
where $\phi_I$ is defined as in \cref{def:RRW} and $\lambda(\phi_I^M)$ is as in \cref{eq:phiI}. 
\end{definition}
Thus, what we called the \emph{Probabilistic Yau's conjecture}, corresponds to describe the convergence of the sequence of random variables $Y_M(I)$ to $1$, for a sequence of intervals such that $\lambda(\phi_I^M)\to +\infty$. 
\begin{remark}
    One can show that $Y_{M}(cI)=Y_{cM}(I)$, for any constant $c>0$.
\end{remark}
\begin{remark}[On the constant $c_n$]
    The equality in \cref{eq:probyau} is realized for any linear function $\f_1:M=\S^n\to \R$, that is, a spherical harmonic of eigenvalue $\lambda_1^2=n$. Since the eigenvalues are: $0,n,2(n+1),\dots$, we can see $\f$ as a realization of $\phi_{\{\sqrt{n}\}}^M=\phi_I^M$ for any $\sqrt{n}\subset I\subset (0,\sqrt{2(n+1)})$. Therefore, 
$
Y_{S^n}(\kop \sqrt{n}\pok)=1.
$
Moreover, in \cite[Eq. (1.42)]{cgv2025StecconiTodino} the authors prove that 
\be 
\E\kop Y_M(I)\pok=1,
\ee 
for all manifolds on which $\phi_I^M$ is homothetic (i.e., $\e(\phi_I^M)=0$, see \cite[Def. 1.11]{cgv2025StecconiTodino}), thus including, in particular, $M\in\kop r\B^n, \S^n, (\S^1)^n\pok$ with arbitrary $I$.
\end{remark}
\begin{remark}
We will see below, that $\lambda(\phi_I)$ is always of the same order as the maximum of $I$, which later we will denote by $\ell$, and that is a more commonly used parameter for the sequence of intervals under examination. However, using the average frequency $\lambda(\phi_I)$ automatically yields the correct constant in \cref{eq:probyau}, regardless of the choice of $I$.
\end{remark}
\subsection{Models considered}\label{sec:intervals}
We want to consider sequences of intervals $I_\ell$, indexed by their maximum $\max I_\ell=\ell>0$, that is a real parameter going to $+\infty$, such that 
\be 
\frac{I_\ell}{\max I_\ell}\to [\a,1] 
\ee 
for some $\a\in [0,1]$. Equivalently, we can assume that $I_\ell$ is
of the form 
\be \label{eq:Il_intro}
I_\ell=[\a\ell-\beta_\ell,\ell]=[\ell-\eta_\ell, \ell]
\ee
where $\beta_\ell=o(\ell)$, or equivalently such that $\exists\lim_{\ell\to +\infty}\ell^{-1}\eta_\ell=1-\a\in [0,1]$, where $\eta_\ell=|I_\ell|$ denotes the amplitude of the interval. Such notation is kept consistent throughout the paper. 

In the sequel, we will mostly focus on two sequences of intervals $I_\ell$, indexed by $\ell\in \R$\footnote{A common natural choice is to use $\lambda$ in place of $\ell$. We choose $\ell$ to avoid confusion with $\lambda(\phi_\ell)$, which has the slightly different, more precise role of being the exact average eigenvalue in the interval $I_\ell$.}: the \emph{Colorful or Full Riemannian Random Wave} (\CRW), corresponding to $I=[0,\ell]$ and the \emph{Monochromatic Riemannian Random Wave} (\MRW), corresponding to $I=[\ell-1,\ell]$. 
However, to keep a general perspective, and a unifying language, we will also consider the more general intermediate regimes $I=[\a\ell-\beta_\ell,\ell]$, for $\a\in [0,1]$ and $\beta_\ell=o(\ell)$. 
They include the \CRW (with $\a=\beta_\ell=0$, i.e., $\eta_\ell=\ell$); \MRW (with $\a=\beta_\ell=\eta_\ell=1$); \emph{Random Band-Limited functions} (with $\a<1$, $\beta_\ell=0$, thus $\eta_\ell=(1-\a)\ell$) considered in \cite{igorsurvey,wigEBNRF,sarnakwigmanCPAM}; 
and \emph{weakly Monochromatic Riemannian random Waves} (\wMRW), with $\a=1$, and with $\eta_\ell=o(\ell)$ possibly diverging to $+\infty$. The latter is a terminology that we introduced in \cite{cgv2025StecconiTodino} in order to keep a clear distinction between the two models ($\eta_\ell=1$ versus $\eta_\ell\to +\infty$) which in the literature are both commonly referred to just as \emph{Monochromatic}
\footnote{From \cite{cgv2025StecconiTodino}: some references (\cite{Zel09,CH20,Keeler}) reserve the word \emph{monochromatic} only to the case $\eta_\ell=1$, while others (see \cite{canzani_MRWsurvey,canzani_sarnak}, or the monochromatic random band-limited functions of \cite{igorsurvey,sarnakwigmanCPAM,wigEBNRF}) admit that $\eta_\ell\to +\infty$.}. We will see that, in our study (\cref{sec:main results}), such distinction is crucial, as explained in \cref{sec:intrdiff} below.
The intermediate cases with $\a\in (0,1)$ are interesting in that they provide intuition by interpolation. However, one expects, as we will see in \cref{thm:ape}, that all $\a\in [0,1)$ exhibit the same qualitative behavior.

By taking the difference $N_{[\ell-\eta_\ell,\ell]}=N_{[0,\ell]}-N_{[0,\ell-\eta_\ell]}$ in \cref{thm:hoermplaw} one can deduce an analogue of Weyl's law for the sequence $\Phi^M_{I_\ell}$.
 \be\label{eq:weyllaw_interval} N_{[\ell-\eta_\ell,\ell]}^M=\Oh(\ell^{n-1}\eta_\ell)
.\ee 
and the order is sharp if $\eta_\ell\to +\infty$. Naturally, such asymptotic behavior extends to all quantities appearing in the identity \cref{eq:NEEsig}.
\subsection{The intrinsic difficulty of the monochromatic regime}\label{sec:intrdiff}

The analysis of the monochromatic case $\a=1$ presents a crucial difficulty. Considering the difference $E_{[\ell-\eta_\ell,\ell]}=E_{[0,\ell]}-E_{[0,\ell-\eta_\ell]}$, the resulting monochromatic analogue of the pointwise law \cref{thm:hoermplaw} is, in general, of this form:
\be \label{eq:horlaw2}
E^M_{[\ell-\beta_\ell,\ell]}(x,y) = \ell^{n-1}\beta_\ell\tyu E^n_1(\ell\dist xy)+\Oh\tyu \frac{
1}{\beta_\ell}\uyt\uyt,
\ee
The error size $\Oh(1)$, resulting in the strictly monochromatic regime $\beta_\ell=1$, prevents a precise comparison of $\phi^M_{[\ell-1,\ell]}$ with its Euclidean counterpart, Berry's random field $\phi^{\R^n}_1$, and this is precisely the boundary of many techniques, including our \cref{thm:drago}. Typically, e.g. in \cite{igorsurvey, canzani_MRWsurvey,sarnakwigmanCPAM}, the monochromatic random waves on general manifolds are studied in the {\wMRW} category. In order to be able to deal with {\MRW}, we will rely on \cite{Keeler}, which refines the error in  \cref{thm:hoermplaw} under additional geometric assumptions on the manifold.
\section{Main results}\label{sec:main results}
\subsection{Berry's conjecture and the monochromatic regime}\label{sec:broad}
 In 1977, Berry \cite{Berry1977} conjectured that when the geodesic
flow on $M$ is sufficiently chaotic, e.g. on negatively curved manifolds, a
high-frequency eigenfunction $\f_i$, restricted to a ball of wavelength radius $\lambda_i^{-1}$ about
a typical point, behaves like \emph{Berry's field} $\phi^{\R^n}_1$, the unique (up to a multiplicative constant) stationary,
isotropic, Gaussian field on $\R^n$ solving $\Delta f=-f$. Meant to bridge
classical chaos and quantum mechanics, the conjecture is
precisely the heuristic underlying the study of Riemannian random waves, and the monochromatic
wave $\phi^M_{[\ell-1,\ell]}$ is its natural probabilistic counterpart.

The strict monochromatic regime, characterized by spectral
windows of size $O(1)$ or smaller, is where Berry's conjecture lives, yet for random waves on general
manifolds it has remained almost untouched, see \cite{Zel09, CH20,Dierickx,Keeler}: the literature concentrates on the integrable model
cases, spheres and tori, and on the Euclidean field itself, while random waves on arbitrary
manifolds are scarcely explored (we review the state of the art concerning the nodal volume in \cref{sec:related}, and
refer to the survey \cite{canzani_MRWsurvey}), and the majority of the literature concerns the \wMRW regime, a spectral window of growing size $\oh(\ell)$, see \cite{sarnakwigmanCPAM,Gass2020,BELIAEVWigman, wigEBNRF, NazarovSodin2016}. Beyond the deterministic setting of Yau's
conjecture, essentially no results reach a spectral window smaller than $O(1)$. The obstruction is
analytic: any such study first requires the fine asymptotics of the covariance function, which
demand increasingly delicate microlocal analysis as the window shrinks (cf.\ \cref{sec:intrdiff}).
Our work is fundamentally based on two recent and independent advances, and is the first to combine them.

The first is the logarithmic improvement in the pointwise Weyl law of Keeler \cite{Keeler} (refining
\cite{canzani_hanin_2018_cinfscasymp}): on a manifold without conjugate points it improves H\"ormander's $\Oh(1/\ell)$ in \cref{thm:hoermplaw} and gives
$\Oh(1/(\ell\log\ell))$.
This gain confirms, in a strong $\mathcal C^\infty$ sense, that
Berry's field is the scaling limit of monochromatic waves on such manifolds (i.e. that \cref{eq:horlaw2} holds with an infinitesimal error), thus reducing Berry's
conjecture to the assertion that \emph{a single high-frequency eigenfunction, near a random point,
behaves like a monochromatic random wave}. Keeler's result is, to our knowledge, largely underexploited \cite{CH20,Dierickx}.

The second is the Wiener--It\^o chaos decomposition of the nodal volume obtained in  \cite{cgv2025StecconiTodino}. Earlier chaos expansions were, in practice, confined to
geometry-compatible (homothetic) fields and required the expectation of a product of $2+2n$
Hermite polynomials, a computation that becomes prohibitive in higher dimension; our formula
instead holds for arbitrary Gaussian fields on arbitrary manifolds and reduces every variance
computation to a product of just four. This allows us to reduce the variance computation to simpler pointwise correlation bounds. 
\subsection{Probabilistic Yau's conjecture: related existing results}\label{sec:related}
The behavior of the covariance function of $\phi_{[0,\ell]}^M$ has been extensively investigated in several contexts (see, e.g. \cite{Ber85, Zel09, BELIAEVWigman, Keeler}, and \cite{igorsurvey} for a recent survey). The natural scaling properties of these kernels suggest a universal asymptotic local behavior, supporting the heuristic that the field should locally resemble its Euclidean counterpart.
Specifically, this leads to the conjecture that $$\phi_I^M \approx \phi_I^{\mathbb{B}^n}, \quad \text{and consequently} \quad Y_M(I) \approx Y_{\mathbb{B}^n}(I).$$ A rigorous discussion of these scaling limits is deferred to \cref{sec:scaling}.

At the level of first moments, Zelditch \cite{Zel09} established that for Zoll or aperiodic manifolds, the total expected nodal $(n-1)$-volume satisfies the  relation
$$ \mathbb{E}[\vol{{n-1}}    (\phi_I^{-1}(0))] = C_M \ell,$$ where $C_M > 0$ is an explicit constant proportional to the Riemannian volume of $M$. This result confirms that \cref{eq:probyau} holds at the level of the expectation for this class of  manifolds.

Regarding the concentration of the nodal volume, Gass \cite[Th. 1.2 and Cor. 3.11]{Gass2020} recently proved that for $\tau \in (0,1]$,\begin{equation}Y_M([\ell-\ell^\tau, \ell]) \to   1\end{equation} almost surely and in $L^p$ for all $p \in \mathbb{N}$. Thus, in particular, the variance goes to $0$. 
This setting is (strictly) contained in that of
\wMRW (for $\tau \in (0,1)$) and \CRW (for $\tau=1$). Note that, in \cite{Gass2020}, the case $\tau\in (0,1)$ and $\tau=1$ are treated separately, and that the constant in the r.h.s. of \cite[Th. 1.2 and Cor. 3.11]{Gass2020} is different in the two cases. This is because the nodal volume is divided by the $\ell$ (called $\lambda$ in \cite{Gass2020}); in our framework, since we divide instead by $\lambda(\phi_{[\ell-\ell^\tau,\ell]})$, and $c_n$, the limit constant is always $1$. 

The fluctuations of the nodal volume have been further characterized for specific geometries. For the sphere $M=\S^n$ 
in \cite{W, marinucci2023laguerre} it was shown that the variance satisfies 
\be\label{eq:sphericalYAu} \text{Var}(Y_{\S^n}(\{\ell\})) =\begin{cases} O(\ell^{-n}) & \text{if } n \ge 3, \\
\Theta(\ell^{-2} \log \ell) & \text{if } n = 2.\end{cases} \ee
In particular, the case $n=2$ was shown in \cite{W}, and a central limit theorem was established afterwards:
the normalized fluctuations are asymptotically Gaussian \cite{MRossiWigman2020}. The $\d\ge 3$ bound was proved in \cite{marinucci2023laguerre}, where the authors conjecture it to be sharp.
An upper bound valid in any dimension had previously been obtained in \cite{W09}. Finally, \cite{Tod24} characterizes the variance of $Y_{\S^2}([\ell-o(\ell), \ell])$.

On the $n$-dimensional torus $M=(\S^1)^n$ the literature concerns a
different ensemble, the \emph{arithmetic random waves}: random
eigenfunctions with a fixed eigenvalue of order $\ell^2$, whose
eigenspace has dimension $N$ equal to the number of lattice points on
the sphere of radius $\ell/2\pi$. For these, Rudnick and Wigman \cite{RudWig08} proved that
the variance of the nodal volume, normalized by the frequency, is
$O(N^{-1/2})$ in every dimension $n\ge2$, and conjectured the order
$O(N^{-1})$, see also \cite{CL20}. The exact order of the variance is known in dimension $2$ and $3$ \cite{KKW13,BM19}, and is $\Oh(1/N^2)$; in which case, the fluctuations are also understood and non-Gaussian \cite{MPRW,Cammarota2017NodalAD}. 
In such a setting, depending on the chosen subsequence of eigenvalues, the local scaling limit may or may not coincide with Berry’s isotropic field: it recovers Berry’s model when the distribution of lattice points converges to the uniform measure on the sphere, but deviates from it when they exhibit angular clustering \cite{KurlbergWigman2017}.
Our monochromatic random waves on $\mathbb T^n$
are a different object, see \cref{rem:torus} below.

The nodal volume of Berry’s random field (see \cref{sec:broad}) has been the subject of extensive independent study.
The nodal length of planar random waves was investigated in \cite{ NourdinPeccatiRossi2019, Vidotto}, while the three-dimensional case has been considered in \cite{DalmaoEstradeLeon}.
Recently, Gass \cite[Theorem 1.10]{Gass2025} showed that by interpreting Berry's random field as the Riemannian random wave $\phi_{\{1\}}(\ell\cdot)=\phi_{\{\ell\}}$ on $\R^n$ and on any ball $\ell \B^n$ of radius $\ell$ (a choice that does not correspond to any standard boundary condition on the ball, but rather to ask no boundary conditions at all), then one has the same behavior as in \cref{eq:sphericalYAu} for the variance of $ 
 Y_{\ell \B^n}(\{1\}) = Y_{\B^n}(\{\ell\})$, proving also that the asymptotic order is exact and that the Central Limit Theorem holds.

In a more general geometric setting, Canzani and Hanin \cite{CH20} investigated manifolds with isotropic scaling \cite[Def. 1]{CH20} and short-range correlations \cite[Eq. (2)]{CH20}. Under these assumptions, which are satisfied for instance by manifolds without pairs of mutually focal points, cf. \cite[Section 1.5]{CH20}, and thus, in particular, by manifolds without conjugate points as those of negative curvature, 
the variance obeys the bound:
\begin{equation}\label{eq:varCH}\text{Var}(Y_M([\ell-1, \ell])) = O(\ell^{-\frac{n-1}{2}}).\end{equation}
This estimate ensures that the Law of Large Numbers holds in $L^2$, since the variance goes to 0. However, note that this bound is significantly larger (roughly the square root) than the one obtained for the sphere, and it has instead the order of magnitude of the generic bounds for the $L^\infty$-norm of eigenfunctions {\cite{soggezel, DONNELLY2001247}}.

In \cite{Dierickx},  a Central Limit Theorem has been established for the nodal length on general compact surfaces satisfying the Canzani–Hanin assumptions, in the monochromatic regime $[\ell-1,\ell]$. Such work is the first that make substantial use of Keeler's logarithmic improvement in Weyl's law \cite{Keeler}. An important difference with respect to our framework is that this result is local, meaning that it concerns the nodal set intersected with a ball of infinitesimal radius, whereas, in this paper, we aim at studying the whole nodal set.

In the present paper (see \cref{thm:ape} and \cref{cor:A}) we prove: on any manifold without conjugate points,
\be\label{eq:mono} \Var(Y_M([\ell-1,\ell])) = \Oh(\ell^{1-n}/\log\ell), \quad \text{(\cref{cor:A})}.
\ee 
Such a bound sharpens \eqref{eq:varCH} by a square and an additional infinitesimal logarithmic factor. Moreover, on any manifold, in the case of a non-monochromatic window
($\a<1$), we obtain the exact order 
\be \label{colo}
\Var(Y_M([\a\ell+\oh(\ell),\ell])) = \Theta(\ell^{-n}), \quad \text{(\cref{thm:ape})}
\ee 
The precise monochromatic order remains open.  When $n\ge 3$, the residual gap to
\eqref{eq:sphericalYAu}, conjectured sharp, is a factor $\ell/\log\ell$ when $n\ge 3$.

Separately, Weyl's law (see \cref{thm:weyllaw} and \cref{thm:ape}) yields that in these cases the dimension parameter is $N^M_{I_\ell}=\Theta(\ell^{n-1}|I_\ell|)$. Therefore, the asymptotic established in \cref{eq:mono} is the first, among all the above, from which we can observe a form of Berry's cancellation phenomenon: the variance is of the same order as $1/N_{I_\ell}^M$ for non-monochromatic random waves, while it has a strictly lower order in the monochromatic regime. The same cancellation can be observed when comparing the nodal set and a non-zero level set. We will discuss these in detail while stating, respectively, \cref{cor:berrycanc_intro} and \cref{cor:berrycanc2} below.

\begin{remark}
In \cite[p.\ 1946]{igorsurvey}, it is argued that the bound \eqref{eq:varCH} could be 
improved combining the Kac–Rice 2-point density technique with sharp covariance $L^2$ bounds, leading, in dimension $n=2$, to $\text{Var}(Y_M[\ell-1,\ell]) = O(\ell^{-1})$. 
Our result (\cref{eq:mono}) shows that, on surfaces without conjugate points, the order is in fact strictly smaller by a logarithmic factor. The fact that $\text{Var}(Y_M[\ell-1,\ell]) = o(\ell^{-1})$ is crucial here, as it corresponds to what we consider a Berry's cancellation phenomenon, see \cref{cor:berrycanc_intro} and \cref{cor:berrycanc2}. Our proof uses one more technical ingredient besides the 2-point Kac-Rice density, the Wiener chaos decomposition we constructed in \cite{cgv2025StecconiTodino}, thanks to which the cancellation shows up naturally. We discuss its role in \cref{subsec:why-chaos}.
\end{remark}

A related probabilistic reformulation exists for a different classical problem: Lerario and Lundberg's \emph{Statistics on Hilbert's 16th Problem}~\cite{LeLuStat16} study the expected number of components and volume of random real algebraic hypersurfaces. The question there is one of expectation and topology rather than of variance and concentration.

\subsection{Main results at a glance}
We prove three main theorems, previewed here and stated precisely in the subsections that
follow. We present them from the most concrete to the most abstract; logically, the order of the proofs is
reversed: the abstract variance machine of \cref{thm:drago} is the engine, the scaling-error
analysis of \cref{thm:OdeltaRWintro} supplies its hypotheses for random waves, and the rates of
\cref{thm:ape} are the payoff.

\smallskip
\noindent\textbf{\cref{thm:ape}: variance rates and Berry cancellation.}
The high-frequency variance of the normalized nodal volume $Y_M(I_\ell)$ of a Riemannian random
wave, for \CRW, \wMRW, and \MRW. For every non-monochromatic window ($\a<1$) it has the exact order \eqref{colo} with no assumption on the geometry, whereas in the monochromatic regime (\MRW) on a manifold without conjugate points
the order is \eqref{eq:mono}; essentially the square of the Canzani--Hanin bound \eqref{eq:varCH} (with also an infinitesimal logarithmic factor). As a consequence, Berry's cancellation (\cref{cor:berrycanc} and \cref{cor:berrycanc2}) holds on every such manifold.

\begin{remark}
    Berry Cancellation refers to a general phenomenon according to which the variance of the nodal volume of monochromatic random waves is smaller than the natural upper bound, that would hold for non-monochromatic regimes or for the nonzero level. See \cref{subsec:why-chaos}.
\end{remark}
\smallskip
\noindent\textbf{\cref{thm:OdeltaRWintro}: the scaling error in the pointwise Weyl law of random waves.}
A general estimate reducing the scaling error $\delta^M_{[\ell-\eta_\ell,\ell]}$ of the pointwise Weyl law associated to an arbitrary spectral window to two ingredients: the explicit Euclidean error $\delta^{\R^n}_{[\ell-\eta_\ell,\ell]}$ and the pointwise Weyl error
$\delta^M_{[0,\ell]}$. It gives sharp sufficient conditions for the random wave to admit a scaling
\emph{limit} in a \emph{fixed} range. This is the property on which \cref{thm:ape} rests, and the one
that separates the chaotic manifolds we treat from others (see \cref{sec:scope}).

\smallskip
\noindent\textbf{\cref{thm:drago}: general abstract machinery.}
A quantitative Law of Large Numbers for the nodal volume of an \emph{arbitrary} Gaussian field,
derived from only two local hypotheses on its covariance: a near-diagonal correlation bound and an off-diagonal short-correlation bound. This result makes no reference to random waves or to the manifold's geometry, and is the engine behind the two theorems above.

\subsection{On the method: Two-point Kac-Rice density plus Wiener chaos}\label{subsec:why-chaos}

A reader familiar with the spherical case may wonder why we do not
simply follow the strategy of \cite{W}: write the variance through
the two-point Kac--Rice density $K(f,x,y)$ of \cref{eq:KR}, and
expand it in powers of the covariance function and its derivatives, then reduce to the covariance of Berry's field via the scaling limit.
On $\S^2$ Berry's cancellation was first
observed in \cite{W} where the phenomenon received its name, confirming what Berry had found on the plane \cite{Berry2002}; this yields the sharp order
$\ell^{-2}\log\ell$ in Equation~(3.2). 

After
scaling the distance by the frequency, the density is expanded in
powers of the covariance function and its derivatives, following an elegant
method of Berry \cite{Berry2002}. The expansion reveals that the leading
non-constant term is purely oscillatory and integrates to a lower order,
while the non-oscillatory terms of leading order cancel each other
\cite[Sec.~1.6.1 and Remark~3.6]{W}; what remains is the cancellation that Berry had predicted for the random wave model
\cite{Berry2002}. 
The computation
relies on two features of the spherical model: the covariance is a
Legendre polynomial, whose asymptotics are explicit to the required
order, and isotropy reduces the two-point density to a function of one
variable. On the tori
$\mathbb{T}^2$ and $\mathbb{T}^3$, \cite{KKW13,BM19} follow a similar
route, where the role of the Legendre asymptotics is played by
arithmetic estimates on the correlations of lattice points. On a general
manifold neither feature is available: the covariance is known only up
to the scaling error of \cref{def:scalimRW}, and the two-point density
depends on the pair of points through the geometry.

The bounds available so far on general manifolds were obtained via the analogous Kac--Rice
route, and do not go below $1/N^M_{I_\ell}$: the estimate of \cite{CH20} is its square root, see \cref{eq:varCH}, and \cite[p.~1946]{igorsurvey}
argues that the order $1/N^M_{I_\ell}$ itself should be reachable by the
same method, in dimension two; both concern the monochromatic regime.
This is the natural order of the variance, in a precise sense. It is an
upper bound for every random wave admitting a scaling limit in fixed
range, monochromatic or not (\cref{prop:natural}): this confirms, in
every dimension, the prediction expressed in \cite[p.~1946]{igorsurvey},
and it is obtained from \cref{thm:drago} at an early stage of the proof of \cref{thm:ape},
where the chaotic components are only bounded in absolute value by the
covariance and no cancellation is used. It is the exact order for every
non-monochromatic window (\cref{thm:ape}), and a lower bound for the
volume of a non-zero level of the monochromatic field
(cf.~\cref{cor:berrycanc2}).
A form of Berry's cancellation is therefore
established only once the variance of the nodal volume is shown to fall
strictly below the threshold $1/N^M_{I_\ell}$, and this is what
\cref{cor:berrycanc_intro} provides on manifolds without conjugate points. In this
sense, the present paper recovers, in that generality, the cancellation
found in \cite{W,Berry2002} and subsequent works \cite{KKW13,BM19, MPRW, NourdinPeccatiRossi2019,Cammarota2017NodalAD,marinucci2023laguerre, Gass2025}.

The chaos decomposition gives us, right away, the object we have to control.
Writing
$\operatorname{Var}(Y_M(I_\ell)) = \sum_{q\ge 2}\operatorname{Var}(Y_M(I_\ell)[q])$,
the second chaos is the candidate leading term; in homothetic settings, for a random eigenfunction, however, it vanishes and what remains is of lower order. This mechanism behind Berry's cancellation was first discovered in \cite{MPRW}, and has since been employed systematically \cite{Cammarota2017NodalAD, marinucci2023laguerre, MRossiWigman2020, NourdinPeccatiRossi2019, Gass2025} on spheres, tori and Euclidean space. 

The formulas used in those works require an
isotropic (homothetic, in our terminology) setting; extending the method
to general manifolds required a different form of the decomposition,
see \cref{rem:whatweuse}.
The formula of
\cite[Corollary~2.1]{cgv2025StecconiTodino}, reported here as \cref{cor:berrycanc}, identifies the general mechanism: up to a correction of size $\e(\phi^M_{I_\ell})$, the second chaos equals
\[
\frac{1}{4(N_{I_\ell})^2}\sum_{i:\lambda_i\in I_\ell}
\Bigl(\frac{\lambda_i^2}{\lambda^2}-1\Bigr)^2 .
\]
The sum is the spectral variance of the window: the second chaos of the
nodal volume encodes the Helmholtz equation of the field, up to the
quantified error $\varepsilon$, the eccentricity. For a non-monochromatic window the sum is
of order one, which forces the exact order $\Theta(1/N^M_{I_\ell})$ of
\cref{thm:ape}; for a monochromatic window all the $\lambda_i$ coincide to
first order, the sum is $O(\eta_\ell/\ell)^2$ and the second chaos is
$o(1/N^M_{I_\ell})$. This is an identity in the eigenvalues, valid for
any window $I_\ell$, and it is what makes the
cancellation visible without an explicit covariance. In this language,
the two terms that cancel in the expansion of \cite{W} can be read as the
second chaos of the nodal volume of a field that satisfies the Helmholtz
equation as opposed to one that does not. 

What is left is to bound the higher
chaoses in terms of the covariance, which is the content of \cref{thm:drago}; the approximation of the
covariance by Berry's kernel enters only there, and only through bounds.

\begin{remark}[On the chaos decomposition of \cite{cgv2025StecconiTodino}]\label{rem:whatweuse}
The chaos expansions of the nodal volume used in
\cite{MPRW,Cammarota2017NodalAD,marinucci2023laguerre,Gass2025} are
written for isotropic fields on spheres, tori and Euclidean space, and
their coefficients are computed relying on the invariance of the model. In \cite{cgv2025StecconiTodino} we obtained a form of the
decomposition valid for any Gaussian field on any Riemannian manifold,
with coefficients expressed through the local parameters of
\cref{def:three}. From that paper we use three statements only:
the definition of the average variance, frequency and eccentricity
(\cref{def:three}, which is \cite[Def.~1.11]{cgv2025StecconiTodino});
the bound on the chaotic components in terms of the covariance jet
(\cref{cor:var_dx}, which is \cite[Thm.~1.10]{cgv2025StecconiTodino});
and the formula for the second chaos above
(\cref{cor:berrycanc}, which is \cite[Cor.~2.1]{cgv2025StecconiTodino}).
All three are restated here in the notation of the present paper.
\end{remark}

\subsection{\cref{thm:ape}: variance rates and Berry cancellation}

Let $\e(\phi_{I_\ell}^M)$ be as in \cite[Def. 1.11]{cgv2025StecconiTodino}.

\begin{proposition}
Let $(M,g)$ be a compact Riemannian manifold and let $0< \eta_\ell\le \ell$ be a real sequence such that $\exists\lim_{\ell\to +\infty} \ell^{-1}\eta_\ell =1-\a $. Assume that $\e_\ell=\e(\phi_{I_\ell}^M)<1$. Then, 
\be 
\E\kop Y_M([\ell-\eta_\ell,\ell]) \pok=1+\Oh(\e_\ell).
\ee
\end{proposition}
\begin{proof}
    See \cite[Corollary 1.3]{cgv2025StecconiTodino} and \cite[Eq. (1.42)]{cgv2025StecconiTodino}.
\end{proof}
The next theorem identifies some cases in which $\e_\ell\to 0$ and provides a new upper bound on the variance.
Define 
\be 
a(\ell):=\begin{cases}
    \log \ell, \quad \text{if $n=2$,}
    \\
1, \quad \text{if $n\ge 3$,}
\end{cases}
\ee
Recall that in the spherical case, when $M=S^n$, \cref{eq:sphericalYAu} yields $\Var\tyu Y_M([\ell-1,\ell])\uyt=\Theta(\ell^{-n}) a(\ell)$.
\begin{mainthm}\label{thm:ape}
Let $(M,g)$ be a compact Riemannian manifold of dimension $n\ge 2$. Let $0< \eta_\ell\le \ell$ be two real sequences such that $\exists\lim_{\ell\to +\infty} \ell^{-1}\eta_\ell =1-\a $. Then, if $\a<1$, 
\be 
\Var\tyu Y_M([\ell-\eta_\ell,\ell])\uyt=\Theta\tyu \ell^{-n}\uyt
\ee
If $\a=1$ and $\liminf_{\ell\to +\infty} \eta_\ell =+\infty$, 
\be \label{eq:1}
 \Var\tyu Y_M([\ell-\eta_\ell,\ell])\uyt\le
 \Oh(\ell^{-n})\tyu a(\ell)+\frac{\ell}{\eta_\ell^2}\uyt,
\ee
If $\a=1$ and $(M,g)$ has no conjugate points,
\be 
\Var\tyu Y_M([\ell-\eta_\ell,\ell])\uyt\le \Oh(\ell^{-n})\tyu a(\ell)+\frac{\ell}{\eta_\ell^2\log\ell}\uyt,
\ee
In all above cases, $\e_\ell\to 0$ and Weyl law (\cref{eq:weyllaw_interval}) is sharp:
\be 
N^M_{[\ell-\eta_\ell,\ell]}=\Theta(\ell^{n-1}\eta_\ell).
\ee
\end{mainthm}

Note that all three asymptotics include the case $[0,\ell]$, but only the third includes the purely monochromatic case $[\ell-1,\ell]$. The second and third asymptotics are significant only for $\a=1$, that is, the weakly monochromatic case.


In particular, \cref{thm:ape} covers the random band-limited functions
of \cite{igorsurvey,sarnakwigmanCPAM,BeliaevWig18,igorsurvey} ($\alpha<1$, $\beta_\ell=0$), for which
it gives the exact order $\ell^{-n}$ on any compact manifold, and the
weakly monochromatic random waves ($\alpha=1$, $\eta_\ell\to+\infty$)
considered in \cite{Gass2020,canzani_sarnak,wigEBNRF}, again on any compact
manifold.

\begin{corollary}\label{cor:A}
If $(M,g)$ is a compact Riemannian manifold of dimension $n\ge 2$, with no pairs of conjugate points, then
\be 
\Var\tyu Y_M([\ell-1,\ell])\uyt=\Oh\tyu \frac{\ell^{1-n}}{\log\ell}\uyt, \quad \text{and} \quad N^M_{[\ell-1,\ell]}=\Theta(\ell^{n-1}).
\ee
\end{corollary}
\begin{remark}
Such bound significantly improves the state of the art, currently given by \cref{eq:varCH} from \cite{CH20}. The bound given by \cref{cor:A} is indeed logarithmically smaller than the square of the former. It should be noted that our context requires a slightly more restricting assumption on the geometry, namely, no conjugate pairs instead of just non-self focal points required in \cite{CH20} (a sufficient condition for a manifold to be of isotropic scaling). 
\end{remark}

\subsubsection{Interpretation in view of Berry  cancellation}
In all the above cases we can observe a form of Berry's cancellation, 
\begin{corollary}[Berry cancellation]\label{cor:berrycanc_intro}
Let the assumptions of \cref{thm:ape} prevail and assume in addition that: if the dimension is $\d=2$, then $\eta_\ell\log\ell=\Oh(\ell)$. Then,
    \be 
\Var\tyu Y_M([\ell-\eta_\ell,\ell])\uyt= \frac{1}{N^M_{[\ell-\eta_\ell,\ell]}}\cdot\begin{cases}
    \Theta(1), \ &\text{if $\a\in [0,1)$,}
    \\
    \oh(1), \ &\text{if $\a=1$ and $(M,g)$ has no conjugate points.}
\end{cases}
    \ee
\end{corollary}
\begin{proof}
The result is a straightforward consequence of \cref{thm:ape}. In the special case, $n=2$, the additional assumption ensures that $\ell^{-n}(\log \ell+\eta_\ell^{-2}\ell/\log\ell)=\ell^{1-n}\eta_\ell^{-1} \cdot \oh(1)$.
\end{proof}
Throughout the paper, Berry's cancellation is understood relative to the
scale $1/N^M_{I_\ell}$, as discussed in \cref{subsec:why-chaos}: this is
the exact order of the variance for non-monochromatic windows, and a
lower bound for the volume of non-zero levels (\cref{cor:berrycanc2}
below), so that a cancellation is a variance of strictly smaller order.
We stress that, in absolute terms, our monochromatic bound
$O(\ell^{1-n}/\log\ell)$ is still larger than the non-monochromatic
order $\ell^{-n}$: our results do not separate the two regimes, and
doing so would require a lower bound in the monochromatic case, which
remains open (see \cref{sec:scope}).

The natural heuristic, in our context would suggest that the following inequality is sharp:
    \begt 
 \Var\tyu Y_M([\ell-\eta_\ell,\ell])[2]\uyt
\le 
\int_{M\times M}\|(\ell^{n-1}\eta_\ell)^{-1}j^{1,1}_{x,y}E_{[\ell-\eta_\ell,\ell]}\|^2_{g^{\ell}} \dd x \dd y \sim \tyu N^M_{[\ell-\eta_\ell,\ell]}\uyt^{-1} \sim \frac{1}{\ell^{n-1}\eta_\ell},
    \eegt
  where $j^{1,1}_{x,y}$ is the (1,1)-jet of $E$ defined in \cref{eq:jlkC} and the norm is defined in \eqref{eq:norm}. In such case, one would have $\Var\tyu Y_M([\ell-\eta_\ell,\ell])\uyt\sim (N^M_{[\ell-\eta_\ell,\ell]})^{-1}$, since the full variance is lower bounded by its second chaotic component. \cref{cor:berrycanc_intro} confirms that this is indeed the case if $\a<1$, but in the monochromatic regime: $\a=1$, the above inequality is asymptotically loose; see \cref{subsec:why-chaos} for the mechanism.

  Berry's cancellation is often associated also to the difference between the behavior of the zero level (nodal set) and that of other levels. While in the rest of the paper we focus solely on the nodal volume (sticking to the setting of Yau's conjecture), we include here one result on the variance of the level volume of monochromatic random wave, as a matter of comparison with \cref{cor:berrycanc}. Indeed, the next theorem confirms the aforementioned second version of Berry's cancellation. Let $\sigma_\ell(x)^2:=\Var(\phi^M_{[\ell-\eta_\ell,\ell]}(x))$, for any $x\in M$.
 \begin{corollary}[Berry cancellation - non-zero level version]\label{cor:berrycanc2}
Under the same hypotheses of \cref{cor:berrycanc_intro}, if $\a=1$ and $(M,g)$ has no conjugate points, then
    \be 
\Var\tyu \frac{\vol{n-1}(\{x:\phi_{[\ell-\eta_\ell,\ell]}^M(x)=t\sigma_\ell(x)\})}{\lambda\tyu \phi_{[\ell-\eta_\ell,\ell]}^M\uyt}\uyt
=
\frac{1}{N^M_{[\ell-\eta_\ell,\ell]}}\cdot\begin{cases}
    \Omega(1), \quad &\text{if $t\neq 0$, }
    \\
    \oh(1), \quad &\text{if $t=0$.}
\end{cases}
    \ee
\end{corollary}
The lower bound in this theorem is proven in the appendix, see \cref{thm:nonzero}, while the upper bound ($t=0$) is the same as before.  Considering the non-centered field $\phi_{[\ell-\eta_\ell,\ell]}^M(x)-t\sigma_\ell(x)$, instead than $\phi_{[\ell-\eta_\ell,\ell]}^M(x)-t$, allows for more natural computations than for a constant level. In standard well explored cases, like sphere, tori, and Euclidean space, the function $\sigma_\ell(x)$ is constant so this subtlety does not appear. 
\begin{remark}
In the two dimensional case, both \cref{cor:berrycanc_intro} and \cref{cor:berrycanc2}, additionally assume that $\eta_\ell$ does not grow too fast: it should be of order bounded by $\ell/\log \ell$. Such regime includes our main focus, the monochromatic case $\eta_\ell=1$, with a large margin. We don't know if the assumption is strictly necessary for the result to hold. 
However, the case $n=2$ is expected to be special. This can be observed in the sphere \eqref{eq:sphericalYAu} and, more crucially, in the error in the pointwise Weyl law, see \eqref{eq:RWscalerrorr_intro} in \cref{thm:OdeltaRWintro} below, see also \cref{rem:sharp_scabo2}. 
\end{remark}

Finally, Berry's cancellation is sometimes understood, more loosely, as
the variance being of smaller order than its natural one; the following
proposition identifies that natural order as $1/N^M_{I_\ell}$ for every
random wave admitting a scaling limit.
The following bound is weaker than \cref{thm:ape}, but it holds under a
single hypothesis and is obtained at an early stage of the proof of
\cref{thm:ape}, before any cancellation is exploited.
\begin{proposition}[Natural order]\label{prop:natural}
Let $I_\ell=[\ell-\eta_\ell,\ell]$ be as in \cref{thm:ape}, and assume
that $\phi^M_{I_\ell}$ admits a $C^3$ scaling limit in some fixed range
$[0,\rho]$, $\rho>0$, i.e. $\delta^M_{I_\ell}(\alpha,C^3,[0,\rho])\to0$ in
the sense of \cref{def:scalimRW}. Then
\be\label{eq:natural}
\Var\bigl(Y_M(I_\ell)\bigr)=O\Bigl(\frac1{N^M_{I_\ell}}\Bigr).
\ee
By \cref{thm:OdeltaRWintro}, the hypothesis holds on every
compact manifold if $\alpha<1$, or if $\alpha=1$ and
$\eta_\ell\to+\infty$; and on manifolds without conjugate points for
every window with $\inf_\ell\eta_\ell>0$, including the monochromatic
one.
\end{proposition}
\begin{proof}
This is \cref{prop:addendum}~(4), combined
with Weyl's law in the form \eqref{eq:dafuck}, which bounds the integral
of $\|\sigma^{-2}j^{1,1}_{x,y}E\|^2_{g^\ell}$ by $O(1/N^M_{I_\ell})$. No
cancellation is involved: the chaotic components are only bounded in
absolute value by the covariance.
\end{proof}

\subsection{\cref{thm:OdeltaRWintro}: the scaling error in the pointwise Weyl law of random waves}

\subsubsection{The scaling limits of random waves}

For $I_\ell$ as above, we can compare the covariance function $E_{I_\ell}^M$ of $\phi_{I_\ell}^M$ with that of the Euclidean analogue:
\be 
E^M_{I_\ell}(x,y)\approx E^{\R^n}_{I_\ell}(u,v), \quad \text{where $\dist xy=|u-v|$}.
\ee
We find that a most insightful way to express such relation is as follows. 
\begin{definition}\label{def:RWUI}
Let us define the Gaussian field $\phi_{\mathsf{U}(I)}^{\R^n}$  with covariance function
\be 
E^n_{\mathsf U (I)}(t):=\begin{cases} 
\frac{1}{|I|}E^n_{I}
, &\text{ if $|I|\neq 0$}
\\
E^n_{\ell}
, &\text{ if $I=\{\ell\}$}.
\end{cases}
\ee
for all $I$ bounded intervals. (The choice of notation is because $\mathsf U (I)$ also denotes the uniform probability measure on $I$.)
\end{definition}

    \begin{definition}\label{def:scalimRW}
Let $M$ be either a compact Riemannian manifold of dimension $n$, or $M=\R^n$. Let $\dist \cdot \cdot$ denote the Riemannian distance function on $M\times M$.
Let $I_\ell\subset [0,+\infty)$ be an interval such that $\ell=\max I_\ell\in \R$ and let $\a\in [0,1]$.
We will refer to the following scaling bound for $\phi^M_{I_\ell}$ as the \emph{$\mC^k$ Scaling Bound of Random Waves of type $\a$ in range $[\rho_\ell',\rho_\ell]$}:
\begin{gather}\label{eq:scalimRW}
E_{I_\ell}^M(x,y)=\ell^{n-1}{|I_\ell|} \tyu 
{E^n_{\mathsf U(\a,1)}}\tyu \ell\cdot \dist x y \uyt
+
R_{\a,I_\ell}^{M}(x,y)
\uyt, \quad \text{and define}
\\
\sup_{ \rho_\ell'\le \mathrm{dist}(x,y)\le \rho_\ell}\sup_{|a|,|b|\le k} \frac{\| \nabla^a \nabla^b R_{\a,I_\ell}^{M}(x,y)\|}{\ell^{|a|+|b|}}\sqrt{n}^{|a|+|b|}=: \delta_{I_\ell}^{M}(\a,\mC^k,[\rho_\ell',\rho_\ell])
    \end{gather} 
where $0\le\rho_\ell'\le \rho_\ell$ and $k\in \N$. We call the number $\delta_{I_\ell}^{M}(\a,\mC^k,[\rho_\ell',\rho_\ell])$ the 
\emph{$\mC^k$ Scaling error of Random Waves of type $\a$ in range $[\rho_\ell',\rho_\ell]$}. 

\textbf{Notation:} in what follows, to keep a lighter notation, we shall write $\delta_{I_\ell}^M \equiv \delta_{I_\ell}^{M}(\a,\mC^k,[\rho_\ell',\rho_\ell])$.
    \end{definition}
Intuitively, if $\delta_{I_\ell}^M\approx 0$, this is saying that for $u\le \ell \rho_\ell$, one has
\be \label{eq:approxscalimRWintro}
\tyu\ell^{n-1}{|I_\ell|}\uyt^{-\frac12}\phi_{I_\ell}^M\tyu x+ \frac{u}{\ell}\uyt\approx \phi_{\mathsf U(\a,1)}^{\R^n}(u)
\ee
and $\delta_{I_\ell}^M$ keeps track of the error in such approximation. The case of $I_\ell=[0,\ell]$ has been widely studied in the literature, and is known as \emph{H\"ormander pointwise Weyl's law}. In particular, on any manifold $M$, we have that $\delta^M_{[0,\ell]}=\Oh(\ell^{-1})$.

\subsubsection{The scaling bound in the Colorful case}

What we know about the above scaling bound derives from the colorful case $I_\ell=[0,\ell]$; the corresponding scaling limit is commonly referred to as \emph{H\"ormander's Pointwise Weyl's law}, and corresponds to \cref{thm:hoermplaw} and to point (1) of the proposition below.
The next proposition collects some results on the behavior of $\delta_{[0,\ell]}:=\delta^M_{[0,\ell]}(0,\mC^k,[0,\rho_\ell])$ on different geometries and for different choices of range $\rho_\ell$.
\begin{proposition}\label{prop:chk}
[\texorpdfstring{\cite{Keeler,hoerm1968,canzani_MRWsurvey,CH20}}{}]
We have the following pointwise Weyl's laws on $(M,g)$, a compact Riemannian manifold without boundary. The next asymptotic relations hold, as $\ell\to +\infty$, for all $k\in \N$ but not necessarily uniformly in $k$.
\begin{enumerate}
    \item Without additional assumptions (result by H\"ormander \cite{hoerm1968}, equivalent to \cref{thm:hoermplaw})
    \be 
\delta_{[0,\ell]}^M=O\tyu \frac{1}{\ell}\uyt,
    \ee
  with $\rho_\ell=\rho(M)$ fixed small enough. 
    \item On a manifold without self-focal points (result by Canzani and Hanin \cite{canzani_hanin_2018_cinfscasymp,CH20} reported in \cite[Eq. (1.3)]{Keeler})
   \be 
\delta_{[0,\ell]}^M=o\tyu \frac{1}{\ell}\uyt
    \ee 
    with $\rho_\ell=o(1)$. Such $M$ is a particular case of what Canzani and Hanin \cite{CH20} define as a \emph{manifold of isotropic scaling}, that is a manifold $M$ such that: with $\a=1$ one has $\delta_{[\ell-1,\ell]}=o(1)$ in infinitesimal range $\rho_\ell=o(1)$.
    \item On a manifold without conjugate points (result by Keeler  \cite[Theorem 1.1]{Keeler}):
   \be 
\delta_{[0,\ell]}^M=O\tyu \frac{1}{\ell\log(\ell)}\uyt
    \ee 
in constant range $\rho_\ell=\frac12 \mathrm{inj}(M,g)$, where $\mathrm{inj}(M,g)>0$ is the injectivity radius.
\item If $M=\R^n$, then (obviously, and by \cref{lem:euclidWeyl})
\be 
\delta_{[0,\ell]}^{\R^n}=0,
\ee 
in any range $\rho$.
\end{enumerate}
\end{proposition}

By evaluating at $x=y$, one obtains the classical Weyl law (\cref{thm:weyllaw}) since $N^M_{[0,\ell]}=\Theta(\ell^n)$; this explains the prefactor in \cref{eq:scalimRW} and \cref{eq:approxscalimRWintro}, since by \cref{eq:NEEsig}, we deduce that whenever the scaling limit holds, it implies the sharpness of the bound at \cref{eq:weyllaw_interval}.
\begin{proposition}\label{prop:Iweyllaw}
\be 
\delta^M_{I_\ell}=\oh(1) \implies N_{I_\ell}^M=\Theta\tyu \ell^{n-1}{|I_\ell|} \uyt.
\ee
\end{proposition}
The above asymptotic is derived by expressing $E_{[\ell-\eta_\ell,\ell]}=E_{[0,\ell]}-E_{[0,\ell-\eta_\ell]}$. This is the standard technique used in the literature to reduce to the case of the spectral function $E_{[0,\ell]}$. In general, by analyzing the Euclidean case, we prove the following general relation between $\delta_{I_\ell}^M$ and $\delta_{[0,\ell]}^M$. Combined with \cref{prop:chk}, the next theorem resumes all that we know about the error in the scaling bound of general random waves, the $\delta_{I_\ell}^M$ is in \cref{def:scalimRW}.
\begin{mainthm}\label{thm:OdeltaRWintro}
Let $k\in \N$ and $\rho_\ell$ be a bounded sequence.
When $\a\in[0,1)$, the scaling bound in \eqref{eq:scalimRW} is a scaling limit with error
\be 
\delta_{[\a\ell-\beta_\ell,\ell]}^M=\Oh\tyu\frac{\beta_\ell+1}{\ell}\uyt =o(1),
\ee
When $\a=1$, the scaling error in \eqref{eq:scalimRW} is 
\be \label{eq:RWscalerrorr_intro}
\delta_{[\ell-\beta_\ell,\ell]}^M
=\Oh\tyu \frac{\ell}{\beta_\ell} \delta_{[0,\ell]}^M\uyt +\Oh(1)\begin{cases}
    \frac{\beta_\ell}{\ell} & n\ge 3
    \\
    m_\ell \sqrt{\frac{\beta_\ell}{\ell}} 
    & n =2 
    \\
    m_\ell^2& n =1
\end{cases},
\ee
as $\ell\to +\infty$, for all $\a\in [0,1]$, where
\be 
m_\ell:=\sqrt{
    \min\kop\max\kop \frac{\beta_\ell}{\ell}, \beta_\ell\rho_\ell\pok,1 \pok.
    }
    \ee
 In particular, the $\mC^k$ Scaling Bound of Random Waves of type $\a=1$ in fixed range $[0,\rho]$, with $\rho=\rho_\ell$, is a scaling limit, that is,  $\delta^M_{I_\ell}=\oh(1)$, under either of the following (sufficient) conditions:
\begin{enumerate}[(i)]
    \item $n\ge 2$ and $\beta_\ell\to +\infty$;
    \item $n\ge 2$, $\frac 1{\beta_\ell}$ is bounded, and $\delta_{[0,\ell]}^M=\oh(\ell^{-1})$.
\end{enumerate}
\end{mainthm}
All results \cite{canzani_hanin_2018_cinfscasymp,Keeler}, on the behavior of \cref{eq:scalimRW} for $\eta_\ell=1$ are based on a study of $E_{[0,\ell]}^M$, in combination with an argument analogue to that of \cref{thm:OdeltaRWintro}. In particular, \cref{thm:OdeltaRWintro} confirms the generally known fact that, without additional assumptions on the manifold $M$, a sufficient condition to ensure a scaling limit is that $\eta_\ell\to+\infty$, since $\delta^M_{[0,\ell]}=\Oh(\ell^{-1})$. 
Instead, the case $\a\neq 1$ behaves essentially as the colorful case $[0,\ell]$. Even more, if $\beta_\ell=0$, which is the most commonly studied situation, in the existing literature, see \cite{sarnakwigmanCPAM,wigEBNRF,igorsurvey}.

The latter two results prove a scaling limit without any further assumption on the geometry of the manifold $(M,g)$ in the colorful, band-limited and weak monochromatic cases. For what concerns the (strict) \emph{monochromatic} case: $\beta_\ell=\eta_\ell=1$, in general the scaling bound \eqref{eq:scalimRW} does not need to be a scaling limit, since \cref{thm:OdeltaRWintro} and \cref{prop:chk} only give 
\be 
\delta_{[\ell-1,\ell]}^M=\Oh(1),
\ee
regardless of the (bounded) range $\rho_\ell$.  
\begin{corollary}[strict monochromatic] Let $n\ge 2$. 
When $\a=1$, if $\beta_\ell$ and $\frac{1}{\beta_\ell}$ are bounded, the scaling bound in \eqref{eq:scalimRW} becomes
\be 
\delta_{[\ell-\beta_\ell,\ell]}^M
=\Oh\tyu \ell \delta_{[0,\ell]}^M\uyt+
\Oh(1)\begin{cases}
    \frac{1}{\ell} & n\ge 3, \\
    \frac{1}{\sqrt{\ell}}  & n =2 .
\end{cases}
\ee 
It is a scaling limit if the manifold $(M,g)$ satisfies the following additional assumption:
\begin{enumerate}
    \item If the manifold has no self-focal points, the error is
\be 
\delta_{[\ell-\beta_\ell,\ell]}^M
=\oh(1),
\ee
in infinitesimal range: $\rho_\ell=\oh(1)$ (equivalent to \cite[Theorem 1]{canzani_hanin_2018_cinfscasymp}).
\item If the manifold has no conjugate points, the error is
\be 
\delta_{[\ell-\beta_\ell,\ell]}^M
= \Oh\tyu \frac{1}{\log \ell}\uyt,
\ee
in fixed range $\rho_\ell=\frac{1}{2}\mathrm{inj}(M,g)$. {(Equivalent to \cite[Theorem 1.1 and Theorem 1.3]{Keeler}.)}
\end{enumerate}
\end{corollary}
As discussed in \cref{rem:sharp_scabo2}, in the case $n=2$, the order $\frac{1}{\sqrt{\ell}}$ of the second term is sharp.
\subsection{More precise bound in terms of the scaling error}\label{sec:scaling}
We have no reason to expect \cref{thm:ape} to be optimal in the monochromatic regime. Indeed, it does not give the optimal rate known on the sphere. 
The theorem is obtained through many intermediate steps, and although we cannot obtain a better rate with the results available so far, we shall report an intermediate result, which we conjecture should yield the optimal rate, after a further investigation of the behavior of $\e_\ell$ or $\delta_\ell$. 
{\begin{proposition}\label{thm:ratefall} 
For any sequence $\eta_\ell= \oh(\ell)$, let $\e_\ell=\e(\phi_{[\ell-\eta_\ell,\ell]}^M)$. Under the assumption that the scaling limit  $\delta_\ell=\delta_{[\ell-\eta_\ell,\ell]}^M\to 0$ holds, there exists $r_0>0$ such that for all $r_0\le r_\ell\le \rho\ell$, we have
\begin{gather}\label{eq:RWvarmono_intro} 
\Var\tyu Y_M([\ell-\eta_\ell,\ell])\uyt
\le
\frac{1}{\ell^n}\tyu\frac\ell{\eta_\ell}\Oh( \e_\ell) 
+\Oh(r_\ell)\uyt
\\
+ \Oh(1)\int_{\substack{x,y\in M,
\\
\frac{r_\ell}{\ell}
\le \mathrm{dist}(x,y)\le \rho}}\|(\ell^{n-1}\eta_\ell)^{-1}j^{1,1}_{x,y}E_{[\ell-\eta_\ell,\ell]}\|^4_{g^{\ell}} \dd x \dd y 
\\
\text{ setting $r_\ell=\Oh(1)$, we get:}
\\
\Var\tyu Y_M([\ell-\eta_\ell,\ell])\uyt
\le
\frac{1}{\ell^{n}}
\tyu 
\frac\ell{\eta_\ell}\Oh( \e_\ell+\delta_\ell^2) 
+a(\ell) \uyt,
\\
\text{ using that $\e_\ell\le \delta_\ell \le \Oh(\frac{\ell}{\eta_\ell}\delta^M_{[0,\ell]})$,}
\\
\le
\frac{1}{\ell^{n}}\Oh\tyu 
\frac{\ell^2}{\eta_\ell^2}\delta_{[0,\ell]}^M
+a(\ell)
\uyt \label{bo1}
\\
\text{ using the general bound $\delta^M_{[0,\ell]}=\Oh(\ell^{-1})$,}
\\
\le
\frac{1}{\ell^{n}}\Oh\tyu 
\frac{\ell}{\eta_\ell^2}+a(\ell)
\uyt.
    \end{gather}
    \end{proposition}
    }
\subsection{\cref{thm:drago}: General abstract machinery}
\subsubsection{How to bound the variance using two covariance rates.}
The starting point of all our results is the next general abstract theorem, which allows one to produce a variance bound by using, in combination, the Kac-Rice formula near the diagonal of $M\times M$ and the Wiener chaos decomposition (obtained in the companion paper \cite{cgv2025StecconiTodino}) to control the off-diagonal behavior. 

We will use the following quantity to measure the covariance of a Gaussian field $\phi$ and its derivatives, collected in the form of a jet, see \cref{sec:preliminaries} for precise definitions. Given $k\in \N$, $E\in \mC^{k,k}(M\times M)$ and $\lambda>0$, we define
\be 
\|j^{k,k}_{x,y}E\|_{g^\lambda}:=\max_{a,b \le k} \frac{\|\nabla^a_x \nabla^b_y E\|}{\lambda_{\ell}^{a+b}}(\sqrt{n})^{a+b},
\ee
where by $g^\lambda$ we denote the Riemannian metric obtained from the rescaling: $g^\lambda=\frac{\lambda^2}{n}g$ (precise definition in \cref{eq:defnorm} below). 

\be \label{eq:deft}
\t(\phi):=
\max_{x\in M}\|j^3_x(\frac{\sqrt{\Var\{\phi(\cdot)\}}}{\sigma}-1)\|_{g^{\lambda(\phi)}},
\ee 
{where $j_x^3$ is the 3-jet as defined in \eqref{def:kjet}.}
\begin{mainthm}\label{thm:drago}
There exists real number $T_0>0$ such that the following holds for any $T\ge T_0$.
Let $(M,g)$ be a compact smooth Riemannian manifold of dimension $\n$, possibly with boundary. Let $\phi$ be a Gaussian field of class $\mC^3(M)$ with covariance function $E$ and parameters $\lambda=\lambda(\phi)>0$, $\sigma=\sigma(\phi)>0$, $\e=\e(\phi)$ defined as in \cite[Def. 1.11]{cgv2025StecconiTodino} (see also \cref{def:three}), and let $\t:=\t(\phi)$ be as in \cref{eq:deft}.
Assume that there exist positive numbers: $ \frac{1} {\lambda T} \le r\le \rho$, satisfying the following assumptions.
\begin{enumerate}
\item The eccentricity is small:
\be 
\e=\e(\phi)<\frac{1}{T}, \quad \text{and} \quad \t=\t(\phi)<\frac{1}{T}.
\ee
\item The pull-back field satisfies a $\mC^3$ boundedness assumption: 
\be\label{eq:comp_intro} 
\sup_{\substack{x,y\in M,
\\
 \mathrm{dist}(x,y)\le r}
}
\sigma^{-2}\|j^{3,3}_{x,y}E\|_{g^{\lambda}}
\le T
\ee
\item The field satisfies a $\mC^1$ short-correlation assumption:
    \be\label{eq:shortc2_intro} 
\sup_{\substack{x,y\in M,
\\
r
\le \mathrm{dist}(x,y)\le \rho
}}
\sigma^{-2}\|j^{1,1}_{x,y}E\|_{g^{\lambda}}
\le \frac{1}{4}
\ee
\end{enumerate}
Then, we have the following quantitative Law Of Large Numbers:
\begin{gather}\label{eq:drago_rdw_intro}
\Var\kop \frac{\lf(M)}{\lambda}\pok\le
\Var\kop \frac{\lf(M)[2]}{\lambda}\pok
+\mathfrak{T} \rho^{-n}\tyu r^{\n}
+
\int_{\substack{x,y\in M,
\\
r
\le \mathrm{dist}(x,y)\le \rho}}\|\sigma^{-2}j^{1,1}_{x,y}E\|^4_{g^{\lambda}} \dd x \dd y 
\uyt.
\end{gather}
where $\mathfrak{T}=\mathfrak{T}(n,M,g;T)$ is a constant depending only on $n,M,g$ and on the constant $T$. 
\end{mainthm}
The proof of \cref{thm:drago} is in \cref{thm:dragodraft}.
\begin{remark}
The constant $T_0$ in the above theorem is just a big enough number, valid for all manifolds; for instance $T_0=77$ should be fine.
\end{remark}
\subsection{Scope of the method: chaotic geometries, the sphere, and open problems}
\label{sec:scope}

Our approach reduces the variance asymptotics of $Y_M(I_\ell)$ to a scaling
\emph{limit}, i.e., the vanishing $\delta^M_{I_\ell}\to 0$ of the scaling error of
\cref{def:scalimRW}, in a \emph{fixed}, macroscopic range $\rho_\ell=\rho>0$.
We stress that this is genuinely stronger than the classical local scaling limit
of the random wave model, which asserts the same convergence only in the
\emph{infinitesimal} range $\rho_\ell=\oh(1)$, i.e. on wavelength-sized
neighborhoods of the diagonal (e.g.\ \cite{Dierickx}). The passage from
infinitesimal to fixed range is a decay-of-correlations statement, and it is
precisely here that the geometry of the geodesic flow enters.

On a manifold without conjugate points, Keeler's logarithmic refinement of
H\"ormander's law gives $\delta^M_{[0,\ell]}=\Oh\!\big(\tfrac{1}{\ell\log\ell}\big)
=\oh(\ell^{-1})$ in fixed range (\cref{prop:chk}), while \cref{thm:OdeltaRWintro}
yields, for the monochromatic window, $\delta^M_{[\ell-1,\ell]}
=\Oh\!\big(\tfrac{1}{\log\ell}\big)\to 0$.
Negatively curved manifolds are the prominent example, in line with the
random wave model being a heuristic for chaotic systems, where geodesics
diverge and correlations decay at macroscopic scale. Flat tori are
another example (\cref{rem:torus}): what the method requires is the
absence of conjugate points, which both ends of the dynamical spectrum
share, rather than chaos itself. On the torus the decay of correlations
at macroscopic scale comes from the equidistribution of lattice points
rather than from hyperbolicity;
\cref{rem:torus} explains how the toral ensemble differs from arithmetic
random waves.

The round sphere $\S^n$ sits at the opposite, integrable extreme, with conujgate points since every point is self-focal. Being a symmetric space, spherical
random waves are homothetic ($\e=0$), so the second chaos of a single eigenspace vanishes
identically by \cref{cor:berrycanc_intro}.
However $\S^n$ has conjugate points, so Keeler's refinement is unavailable, and
the plain H\"ormander bound $\delta^{S^n}_{[0,\ell]}=\Oh(\ell^{-1})$ only yields
$\delta^{S^n}_{[\ell-1,\ell]}=\Oh(1)$ through \cref{thm:OdeltaRW}. 
This reflects the sharpness
of Hörmander's remainder on the sphere, where the spectral function
jumps by a whole cluster of eigenvalues, and says nothing about the
monochromatic covariance itself, which is an explicit Gegenbauer
polynomial. Hilb's asymptotics for Jacobi polynomials
\cite[Thm.~8.21.12]{szego1975orthogonal} show that the $C^0$ scaling error of
$\phi^{\mathbb S^n}_{\{\lambda_k\}}$ in a fixed range $[0,\rho]$, $\rho<\pi$,
is $O(\ell^{-1/2})$ for $n=2$ and $O(\ell^{-1})$ for $n\ge3$.

Our technique is, for now, silent on the sphere, as it requires a $C^3$ version of Hilb's bound, which is not available in the literature in the form we need. Plausibly, it follows, with the
same rates, from the same asymptotics applied to the Jacobi polynomials
with raised parameters that represent the derivatives, but we have not
carried out the verification. Granted it, \cref{thm:ratefall} with
$\eta_\ell=1$, $\varepsilon=0$ and $\ell\delta_\ell^2=O(1)$ would return
exactly the upper bounds in \eqref{eq:sphericalYAu}. 
This is a limitation of our general method, not a gap in knowledge: the
variance of the nodal volume of spherical random waves is known to be that of \cref{eq:sphericalYAu}, proven in a series of papers completed by
\cite{marinucci2023laguerre}, and including \cite{W09,MRossiWigman2020,Tod24}.
The method is based on the Kac-Rice formula and a chaos-decomposition argument close in spirit to ours but adapted to $S^n$. 
The two pictures remain complementary (see also
\cite{ZELDITCH1997}). The sphere lies outside the natural scope of Berry's heuristics, its nodal fluctuations being governed by
symmetry rather than by universality, whereas our results aim at the
manifolds for which Berry's conjecture was designed.



\begin{remark}
The assumptions $\e_\ell<1$ ensures that the field $\phi^M_{I_\ell}$ has non-degenerate first jet. However, in settings where $\e_\ell$ is just bounded, and where we have no information ensuring the non-degeneracy, then even the expectation is unclear as nor the Kac-Rice formula nor the chaos decomposition from \cite{cgv2025StecconiTodino} is applicable; this is the case for monochromatic windows $[\ell-\Oh(1),\ell]$ on general manifolds.
\end{remark}

\begin{remark}\label{rem:torus}
    Let $M=\mathbb T^n=\mathbb R^n/\mathbb Z^n$, $n\ge2$, with its flat metric. A flat torus $M$ has no conjugate points, hence $\phi^{\mathbb T^n}_{[\ell-1,\ell]}$ admits a $C^k$ scaling limit to
Berry's random field in any fixed range, with error $O(1/\log\ell)$, and the conclusions of \cref{thm:ape}, \cref{cor:berrycanc_intro} and
\cref{cor:berrycanc2} hold on $\mathbb T^n$.
This ensemble should not be confused with the
arithmetic random waves of \cref{sec:related}. The frequencies of
$\mathbb T^n$ are $\lambda=2\pi|k|$, $k\in\mathbb Z^n$, so the window
$[\ell-1,\ell]$ collects all norms of lattice points in the annulus: the set of frequencies
$\Lambda_\ell=\{2\pi|k|\colon k\in Z^n,\ (\ell-1)/2\pi<|k|\le\ell/2\pi\}$. As $\ell\to+\infty$, such a range includes a growing number of values of $\lambda$, meaning that $$\phi^{\mathbb T^n}_{[\ell-1,\ell]}=\sum_{\lambda\in \Lambda_\ell}\phi^{\mathbb T^n}_{\{\lambda\}}$$ is not a genuine eigenfunction.
An arithmetic random wave is a
single one of these $\phi^{\mathbb T^n}_{\{\lambda\}}$.
\end{remark}

\medskip
\noindent\textbf{Open problems.} Three questions are left open.
\emph{(i)} The sharp order in the monochromatic regime on chaotic manifolds:
\cref{cor:A} improves \cref{eq:varCH} to its logarithmic square, but leaves a
factor $\ell/\log\ell$ to the order $\Theta(\ell^{-n})$ conjectured by comparison
with \eqref{eq:sphericalYAu}; \cref{thm:ratefall} isolates, conditionally, the
fourth-chaos estimate that would close this gap.
(ii) The $C^3$ scaling limit of monochromatic random waves on
$\mathbb S^n$ in fixed range; 
it would make
the spherical upper bounds \eqref{eq:sphericalYAu} a consequence of \cref{thm:ratefall}.
\emph{(iii)} A central limit theorem for $Y_M(I_\ell)$, which our second-moment
analysis does not address.
\section{Preliminaries}\label{sec:preliminaries}
\subsection{Technical definitions (jets and derivatives norms)}
In this section we fix the conventions and notations that we will adopt to work with tensors, norms and covariant derivatives. Our definitions rely on standard concepts, for which we refer to the monograph \cite{leeriemann}. 
\subsubsection{Tensor calculus}
Let $M$ be a smooth $n$-dimensional manifold. For any $l,k\in \N$, we denote the bundle of \emph{tensors of type $(l,k)$} as
\be 
T^{(l,k)}M:=TM^{\otimes l}\otimes T^*M^{\otimes k}
\ee
whose fiber $T^{(l,k)}_xM$ over $x\in M$, is the space of multilinear functions $\tau\colon T^*_xM^l\times T_xM^k\to \R$. By standard properties of tensor calculus, such $\tau$ can be canonically identified with a linear function on $T_x^{(k,l)}M$. 
Sections of the bundle $T^{(l,k)}M$ are called \emph{tensor fields of type $(l,k)$} and we denote by $\m{T}^{(l,k)}$ the space that they form.

For $\tau\in T^{(l,k)}_xM$ and $\tau'\in T^{(k,l)}_xM$, we denote the standard tensor contraction as $\langle\tau,{\tau'}\rangle \in \R$, that is, the evaluation of the $\tau$ on $\tau'$. For tensor fields $\tau,\tau'$, the same notation $\langle\tau,{\tau'}\rangle$ denotes the function $x\mapsto \langle\tau_x,{\tau'}_x\rangle$.

By convention $T^{(0,0)}_xM=\R$, so tensor fields of type $(0,0)$ are functions.
$TM=T^{(1,0)}M$ is the tangent space, $T^*M=T^{(0,1)}M$ is the cotangent space; $(1,0)$ and $(0,1)$ tensor fields are, respectively, vector fields and differential $1$-forms. $T^{(0,2)}_xM$ is the space of bilinear forms on $T_xM$, and a Riemannian metric $g$ is a tensor field of type $(0,2)$ that is, in addition, symmetric and positive definite. 

\subsubsection{Covariant derivative}
We fix, once and for all, a Riemannian metric $g$ on $M$, and denote by $\nabla$ the corresponding Levi-Civita connection, canonically extended to an operator acting on tensor fields, see \cite[Propositions 4.15-4.17]{leeriemann}
\be \label{eq:covder}
\nabla\colon \m{T}^{(l,k)}\to \m{T}^{(l,k+1)}, \quad \text{where}\quad \langle (\nabla_X \tau), \tau'\rangle \equiv \langle \nabla \tau, \tau'\otimes X\rangle, \quad \forall X\in \m T^{(1,0)},\tau'\in \m T^{(k,l)}.
\ee
We will mostly be concerned with tensors of type $(0,k)$, originating by differentiating functions.
Note that for a function $f\in \mC^\infty(M)$, we have that $\nabla_X f(x)=\langle d_xf,X\rangle$\footnote{Note that $\nabla f=df$ is the differential of $f$, and not the gradient.} and by iterating \cref{eq:covder} $k\in \N$ times, we obtain a tensor field $\nabla^kf$ of type $(0,k)$, corresponding to the standard differential of order $k$ in the Euclidean setting. For instance, $\nabla^2 f=\nabla(df)$ is the Riemannian Hessian of $f$. 

We will avoid the expression $\nabla_X^k f$, in favor of $\langle \nabla^k f, (\cdot)\otimes X\rangle$. The reason will be clear in a moment. We will denote by $\nabla_x^k f$ the $k$-multilinear form on $T_xM$ obtained by evaluating $\nabla^k f$ at a point $x\in M$.
\subsubsection{Jets}\label{sec:prelim_jets}
For $k\in \N$, we will call the \emph{jet bundle of order $k$} the vector bundle $J^kM$ with fiber over $x\in M$ given by
\be 
J^k_xM:=\R\oplus T^{(0,1)}_xM\oplus \dots \oplus T^{(0,k)}_xM
\ee
The $k$-\emph{jet} (also called $k$-\emph{jet prolongation}) of a function $f\in \mC^k(M)$ is a section $j^kf$ of the bundle $J^kM$, that is defined at $x\in M$ by
\be \label{def:kjet}
j^k_xf=\begin{pmatrix}
    f(x), & \nabla_x f, & \dots & ,\nabla^k_xf
\end{pmatrix}
\ee
(this expression gives an identification, depending on the metric $g$, with the canonical jet bundle.)
Moreover, for any pair of points $x,y\in M$ and $l,k\in \N$, we define 
\be\label{eq:directsum}
J^{l,k}_{x,y}M:=J^l_xM\otimes J^k_yM\cong\bigoplus_{a\le l,b\le k} T^{(0,a)}_xM\otimes T^{(0,b)}_yM,
\ee
and denote as $J^{l,k}M$ the resulting vector bundle over $M\times M$. Now, for any $a,b\in \N$ and $x,y\in M\times M$ we can differentiate a function $C\colon M\times M\to \R$ on each variable to obtain
\be
\nabla^{a,b}_{x,y}C:=\nabla_x^a\nabla_y^bC \in T^{(0,a)}_xM\otimes T^{(0,b)}_yM,
\ee
which defines $\nabla^{(a,b)}C$ as a \emph{double tensor field} (of type $(0,a)$-$(0,b)$) on $M \times M$. Finally, we are able to define the $(l,k)$-jet of $C$ as the section $j^{l,k}C$ of $J^{l,k}M$ defined by the collection of 
\be \label{eq:jlkC}
j^{l,k}_{x,y}C= \tyu \nabla_x^a\nabla_y^bC \uyt_{a\le l,b\le k}  
\ee
for any pair of points $x,y\in M$. Here, we are interpreting an element $j^{l,k}_{x,y}C$ of the space $J^{l,k}_{x,y}M$, defined in \cref{eq:directsum}, as a matrix of tensors in each of the direct summands in \cref{eq:directsum}\footnote{
A natural and deeper way to interpret a jet in $J^kM$ is as a polynomial function of degree at most $k$ on $T_xM$; and a jet in $J^{l,k}_{x,y}$ as a polynomial function on $T_xM\times T_yM$ of bi-degree $(l,k)$, that is a polynomial of degree at most $l$ in the variable $\dot{x}\in T_x^*M$ and of degree at most $k$ in the variable $\dot{y}\in T_y^*M$. Then, the elements of the matrix in \cref{eq:jlkC} are the bi-homogeneous (i.e., homogeneous in each variable) components of the polynomial. For us, it will be more convenient to keep the homogeneous components separated.
}. 
Note that all the terms are in different spaces. We call $\mC^{l,k}(M\times M)$ the space of functions $C$, said to be \emph{of class $\mC^{l,k}$}, for which $j^{l,k}C$ exists and is continuous.

\subsubsection{Gaussian fields}
All this language will turn out to be useful to describe the covariance of derivatives of a Gaussian field. First, let us recall that a coordinate-free way to express the covariance between two Gaussian vectors $G_i\randin V_i$, with $i\in \{1,2\}$, in two possibly different vector spaces is as a double tensor 
\be 
\Cov(G_1, G_2):=\E\kop G_1\otimes G_2\pok \in V_1\otimes V_2,
\ee
which can be equivalently (and canonically) seen as a bilinear form on $V_1^*\times V_2^*$, or as a linear map $V_2^*\to V_1$, or $V_2^*\to V_1$.
That said, we can formulate a classical fact on the derivatives of a covariance function as follows.
\begin{lemma}
    If $f\randin \mC^r(M)$ is a Gaussian field, then its covariance function $C$ is in $\mC^{r,r}(M\times M)$ and for any $l,k\le r$, we have
    \be 
\Cov\tyu j^l_xf, j^k_y f\uyt=\E\kop j^l_xf \otimes j^k_y f \pok=j^{l,k}_{x,y} C \in J^{l,k}_{x,y}M,
    \ee
    and the analogous identity holds for $\nabla^{a,b}_{x,y}C$.
\end{lemma}
A particular case of the above is that of the Adler-Taylor metric (see \cite[Eq. (12.2.1)]{AT07})
of a Gaussian field $f$ with covariance $C$, such that $C(x,x)=1$, that is the $(0,2)$ tensor field defined by
\be 
g^f_x:=\Cov\tyu \nabla_xf,\nabla_x f\uyt=\nabla^{1,1}_{x,x}C\in T^*_xM\otimes T^*_xM.
\ee
Given a Riemannian metric $g$, and $\lambda \in \R$, we shall define the new metric 
\be 
g_x^\lambda := \frac{\lambda^2}{n}g.
\ee
Note that $g^\lambda$ is just a constant multiple of $g$. In many instances in the following, we will need to compare $g^\lambda$ with $g^f$.
\subsubsection{Norms}\label{sec:norms}
Given any Riemannian metric (even discontinuous) $\rho$ on $M$, possibly different than our reference metric $g$, we denote by $\|u\|_\rho=\sqrt{\rho_x(u,u)}$ the norm of $u\in T_xM$. We extend it to a norm on $J^{l,k}_{x,y}M$ by setting 
\be \label{eq:defnorm}
\|\nabla^{a,b}_{x,y}C\|_\rho:=\max_{
\substack{
u\in T_xM\smallsetminus\{0\} \\ v\in T_yM\smallsetminus\{0\} 
}
}\frac{\left|\langle \nabla_x^a\nabla_y^bC, u^{\otimes a}\otimes v^{\otimes b}\rangle \right|}{\|u\|_\rho\|v\|_\rho},
\quad 
\|j^{l,k}_{x,y}C\|_\rho:=\max_{a\le l, b\le k}
{\left \|\nabla_x^a\nabla_y^bC \right\|_\rho}
\ee
The norm on $J^k_x(M)$ is defined in the analogous way. When $\rho=g$, we will omit the subscript; $\|\cdot \|_g\equiv \|\cdot\|$. Observe that
\be \label{eq:norm}
\|j^{l,k}_{x,y}C\|_{g^\lambda}=\|j^{l,k}_{x,y}C\|_{\frac{\lambda^2}{n}g}=
\max_{
\substack{
u\in T_xM\smallsetminus\{0\} \\ v\in T_yM\smallsetminus\{0\} 
}}
\max_{a\le l, b\le k}
\frac{\left\|\nabla_x^a\nabla_y^bC \right\|}{\lambda^{a+b}} (\sqrt{n})^{a+b}.
\ee
\subsubsection{Covariance of the first jet}
Particularly important for our study is the case of the first jet $j^1_xf=(f(x),\nabla_xf)$, that is an element of the space $J^1_xM=\R\oplus T^*_xM$. This case is described in \cite[Sec. 1.3.1]{cgv2025StecconiTodino}, with a slightly different notation. Let us briefly discuss how it translates to the more general framework developed above.

If $f$ is a $\mC^1$ Gaussian field, and $x,y\in M$ are two points, then the covariance between $j^1_xf$ and $j^1_yf$ is the following object:
\be 
j^{1,1}_{x,y}C=\Cov\tyu j^1_xf,j^1_yf\uyt=\begin{pmatrix}
C({x,y}) & \nabla^{1,0} C (x,y)\\ \nabla^{0,1}C (x,y)& \nabla^{1,1}C(x,y)
\end{pmatrix}\in (\R\oplus T_x^*M){\otimes }(\R\oplus T_y^*M)=J^{1,1}_{x,y}M,
\ee
corresponding exactly to $j''_{x,y}C$ in \cite[Eq. (1.32)]{cgv2025StecconiTodino}. Consequently, the definition of $\|\cdot\|_\rho$ given in \cref{eq:defnorm} above corresponds to \cite[Def. 1.8]{cgv2025StecconiTodino}, so that
\be \label{eq:comparjet}
\|j''_{x,y}C\|_{g^f}=\|j^{1,1}_{x,y}C\|_{g^f}.
\ee 
This will be important in the following, when applying \cite[Theorems 1.10, 1.13, 1.14]{cgv2025StecconiTodino}.

\section{Proof of \cref{thm:drago}}
The proof is divided in a series of eight small steps. The crucial ones will be steps (VII) and (VIII). In the next two preliminary subsections, we prove two lemmas that are needed in those steps. 
\subsection{(Step VII) Near-diagonal behavior. }\label{eq:prooflemVII}
Consider a $\mC^3$ Gaussian random field $f$ defined on a $\mC^3$ compact manifold $M$ of dimension $\n \in \N$, possibly with boundary, and endowed with a $\mC^1$ Riemannian metric $g$. 
For this section, our only assumption on $f$, other than its $\mC^3$ regularity, is that it should have unit variance and non-degenerate differential $d_xf\in T_x^*M$, that is, we assume that 
\be\label{eq:assjet} 
\E\kop |f(x)|^2\pok=1,\quad \text{ and }\quad 
\|u\|_{g^f}^2:=\E\kop |d_xf(u)|^2\pok>0, 
\ee
for every $x\in M$ and $u\in T_xM\!\smallsetminus\! \{0\}$. 

 \begin{lemma}\label{lem:unifbound}
Assume that $f_\ell$, with $\ell \in \N$ is a sequence of Gaussian fields satisfying \cref{eq:assjet}, with covariances $C_\ell$ such that
\be 
\sup_{\ell\in \N}\sup_{x,y\in M,\a,\beta \le 3} |\nabla^\a_x \nabla^\beta_y C_\ell|<+\infty,
\ee
then for any compact subset $K\subset M$,
\be 
\sup_{\ell\in \N}\E\kop |\lfl(K)|^2 \pok<+\infty.
\ee
\end{lemma}
\begin{proof}
The boundedness of the third derivatives implies that $f_\ell$ is a relatively compact family in $\mC^2(M)$, see \cite[Appendix A]{NazarovSodin2016} and \cite{dtgrf2019LerarioStecconi}. Therefore, the second moment is uniformly bounded because of \cite[Remark 1]{fom2024GassStecconi}.
An alternative proof is the explicit study of the two-point Kac-Rice formula for $\E\kop \lf(dx)\lf(dy)\pok$ carried out in \cite[Sec. 8]{nvdfg2024PeccatiStecconi}, from which one can clearly see that all the bounds are uniform with respect to (at least) the third order derivatives. 
\end{proof}
\begin{remark}
    We conjecture that this lemma is true with $2$ instead than $3$, but false with $1$.
\end{remark} 

 
 \begin{lemma}\label{lem:unifboundlambda} There exists a constant $\mathfrak{K}(n,M,g;B,\e)>0$, depending on the geometry $n,M,g$ and on an additional real parameters $B,\e>0$, decreasing in $\e$, such that the following holds.
 Assume that $f$ is a $\mC^3$ Gaussian field on $M$ satisfying \cref{eq:assjet}, with covariance $C$ such that
\be\label{eq:hypunifboundlambda} 
\sup_{\substack{x,y\in M,
\\
\mathrm{dist}(x,y)\le  r}}
\|j^{3,3}_{x,y}C\|_{g^{\lambda}}
\le B,
\ee

for some $r \le \mathrm{diam}(M,g)$, and some $\lambda>0$. Then, for any $M'\subset M$ Borel subset, we have that
\be 
\frac{\lambda^{-2}}{r^\n}
\int_{M'}\int_{B_{r}(x)} \E \kop  \lf(dx)\cdot \lf(dy)\pok < \mathfrak{K}(n,M,g;B,\lambda r)\vol{\n}(M').
\ee

 \end{lemma}
\subsubsection{Preliminary to the proof: the Kac-Rice formula}
Let $f$ be a Gaussian field with covariance function $C(x,y)$. The Kac-Rice formula for the second moment gives
\begin{gather} \label{eq:KR}
\E \kop  \lf(dx)\cdot \lf(dy)\pok=K(f,x,y)dxdy,
\\
K(f,x,y):=\E\kop \|d_xf\|\|d_yf\|\Big| f(x)=f(y)=0\pok\frac{dxdy}{(2\pi)\sqrt{1-C(x,y)^2}},
\end{gather}
for almost every pair $(x,y)$ of (necessarily distinct) points, such that $C(x,y)\neq 1$, see \cite{AzaisWscheborbook, AT07} for standard Kac-Rice formula. We will denote the above function as $K(f,x,y)$, whenever defined. This function is often called the two-point correlation function in the literature, see for instance \cite{W,benatarMW,MRossiWigman2020} among others. In \cite[Lemma 8.1]{nvdfg2024PeccatiStecconi} the authors prove that, under the assumption at \cref{eq:assjet}, one has that $C(x,y)\neq 1$ for almost every pair $(x,y)\in M\times M$; consequently, by standard arguments, we have that $\E\{|\int_M \lf(dx)|^2\}=\int_{M\times M} K(f,x,y)dxdy$ is well defined and, due to \cref{lem:unifbound}, finite.
\begin{proof}
Consider the case in which $M=\mathbb B_{ r}$ is the radius $r$ ball in $\R^\n$. Then, \cref{lem:unifbound} above implies that, for any $B>0$, there exists a constant $K(B,r)>0$, such that 
 \be \label{eq:bobbound}
\sup_{u,v\in \mathbb B_r,\a,\beta \le 3} |\nabla^\a_x \nabla^\beta_y C|<B,
 \implies 
\E\kop |\lf(\mathbb B_r)|^2 \pok<K(B,r)\in \R.\footnote{The behavior of $K(B,r)$ with respect to $r$ should be $\Theta(r^{2\n-1})$, but we don't want to rely on that.}
\ee

Moreover, let $\nu(s/r)$ denote the minimum number of balls of radius $r$ that are needed to cover a ball of radius $s>r$. Then, there is a dimensional constant $c(\n)>0$, such that $\nu(s/r)\le c(\n)(s/r)^{\n}$. Let $B_{r}(x_i)$ be one such collection of $\nu(s/r)$ elements, with $s>r$. Then,
\bega \label{eq:bobbound2}
\E\kop |\lf(\mathbb B_{s})|^2 \pok
&= 
\E\kop |\sum_{i=1}^{\nu(s/r)}\lf(B_r(x_i))|^2 \pok
\\
&\le  
\nu({s/r})\sum_{i=1}^{\nu({s/r})}\E\kop |\lf( B_r(x_i))|^2 \pok
\\
&<c(\n)^2K(B,r) (s/r)^{2\n}.
\eega
The latter inequality holds whenever the next implication holds:
\be \label{eq:imply}
u,v\in \B_s,\ |u-v|\le r \implies \max_{\a,\beta \le 3} |\nabla^\a_u \nabla^\beta_v C|<B.
\ee
We will use this later on in the proof.

Let $D\subset M$ be a finite set of points such that $\{(x,y)\in M'\times M: \mathrm{dist}(x,y)\le r\}\subset \cup_{p\in D}B_{r}(p)^2$. 
By standard arguments, we can assume that such a family has cardinality $\#(D)\le Nr^{-\n}\vol{\n}(M')$, for some geometric constant $N=N(n,M,g)>0$. 
Then, 

\be \label{eq:int-diag!}
\begin{gathered} 
\int_M\int_{B_{r}(x)} \E \kop  \lf(dx)\cdot \lf(dy)\pok
\le 
\sum_{p\in D}\int_{B_{r}(p)^2} \E \kop  \lf(dx)\cdot \lf(dy)\pok 
\\
\le 
\sum_{p\in D}\int_{B_{r}(T_pM)^2} K\tyu f,\exp_p\tyu  u\uyt,\exp_p\tyu v\uyt\uyt J^{2n} dudv
\\
\le J^{2\n+2} 
\sum_{p\in D}\int_{B_{\lambda r}(T_pM)^2} K\tyu f\circ \exp_p\tyu \frac{\cdot}{\lambda}\uyt, u, v\uyt(\lambda)^{2-2\n}dudv=\dots
\end{gathered}
\ee
Here, $J=J(n,M,g)>0$ is a geometric constant that bounds the $\mC^1$ norm of the map $\exp_p(\cdot)$. In the last step we changed variables $u,v \mapsto \lambda u ,\lambda v$ and used the fact that $K(F(\cdot),u,v)=\lambda^2K(F(\lambda^{-1} \cdot), \lambda u,\lambda v)$ for any $F$ Gaussian field. At this point, we recognize that the integral is the second moment of the nodal volume of the pull-back Gaussian field $f^p:=f\circ \exp_p\tyu \frac{\cdot}{\lambda}\uyt$ on $B_{\lambda r}(T_pM)$, then we can bound \cref{eq:int-diag!} by
\be 
\dots \le J^{2\n+2} 
\lambda^{2-2\n}\sum_{p\in D}\E\kop |\m L_{f^p}(B_{\lambda r}(T_pM))|^2\pok=\dots
\ee

Now, we check that the Implication \eqref{eq:imply} holds for the covariance $C^p$, with {$s:=\lambda r \ge \e$}. 

Let $L=L(n,M,g)$ denote the maximum Lipschitz constant of the exponential maps $\exp_x$ restricted to $u,v\in B_{\lambda r}(T_xM)$ such that $|u-v|\le \epsilon$. Notice that $L\ge 1$. For any $u,v \in B_{\lambda r}(T_xM)$ such that $|u-v| \le \epsilon$, the  distance on $M$ satisfies:\be\dist{\exp_x(u/\lambda)}{\exp_x(v/\lambda)} \le \frac{L}{\lambda} |u-v| \le \frac{L \epsilon}{\lambda}.\ee
If we choose $\varepsilon$ such that $L\varepsilon \le \lambda r$, the points stay within a distance $r$ on $M$, where the scaled jet of the covariance is bounded by $B$.
By applying the chain rule to the covariance $C^p$ of $f^p$, we obtain the following estimate for its derivatives up to order 3:
\bega 
\max_{|\a|,|\beta|\le 3}|\nabla^\a_u \nabla^\beta_v C^p|
&\le 
\max_{|\a|,|\beta|\le 3}\frac{|\nabla^\a_{\exp_x(\frac{u}{\lambda})} \nabla^\beta_{\exp_x( \frac{v}{\lambda})} C|}{\lambda^{|\a|+|\beta|}}\cdot \max_{|\a|,|\beta|\le 3}{|\nabla^\a_{\frac{u}{\lambda}} \nabla^\beta_{\frac{v}{\lambda}} \exp_x(\cdot)|}
\\
&\le 
\|j^{3,3}_{x,y}C\|_{g^{\lambda}}\cdot E \le B\cdot E
,
\eega
 where $E=E(n,M,g)$ is a geometric constant bounding the derivatives of the exponential map, since {$\frac{\|u\|}{\lambda}\le r\le \mathrm{diam}(M,g)$ }is bounded. 

By \cref{eq:bobbound2}, we deduce that 
\be 
\dots\le J^{2\n+2}\lambda^{2-2\n}\#(D)\cdot c(n)^2K(B\cdot E,\e)(\lambda r)^{2\n}\e^{-2\n}
\le r^{\n}\lambda^{2} J^{2\n+2}K(B\cdot E,\e) \e^{-2\n}N \vol{\n}(M').
\ee
Since the above inequality holds for any $\e$ such that $L\e\le \lambda r$, 
we conclude by defining $\mathfrak{K}(n,M,g;B,\lambda r):= \inf_{L\e\le \lambda r}J^{2\n+2}K(B\cdot E,\e)\e^{-2\n} N$, that is clearly decreasing in its last argument.

\end{proof}

\subsection{(Step VIII) Off-diagonal behavior: Proof of \cref{lem:offdiag}}\label{sec:prooffdiag}
Let us consider a $\mC^3$ field $f$ on $M$, satisfying \cref{eq:assjet} and with covariance $C$. 
We will use the following result from \cite{cgv2025StecconiTodino} (see also \cref{eq:comparjet}).
 \begin{theorem}[Theorem 1.10 \cite{cgv2025StecconiTodino}]\label{cor:var_dx}
Let $f:M\to \mathbb{R}$ be as above. Then,
\bega
|\E\kop \lf(dx)[q]\cdot \lf(dy)[q]\pok|
\le \
2^q\cdot\frac{\lambda(f,x)\lambda(f,y)}{\n}
\cdot
{\|j^{1,1}_{x,y}C\|_{g^f}^q}
\ dxdy.
\eega
\end{theorem}
The term $\lambda(f,x)$ in the above statement is defined as 
\be 
\lambda(f,x)^2=\E\kop \|d_xf\|^2\pok=\fint_{u\in S(T_xM)}\E\kop |\langle d_xf,u\rangle|^2\pok \n \dd u.
\ee
Be aware that such a formula holds for a field $f$ satisfying \cref{eq:assjet}; we stress this, as later we will extend it to more general fields. Under the same caveat, we have
\be 
\lambda(f)^2=\fint_M\lambda(f,x)^2 dx,\quad \text{and} \quad \e(f)=\max_{x\in M,\ u\in S(T_xM)}\left|\frac{\sqrt{\E\kop |\langle d_xf,u\rangle|^2\pok}}{\lambda(f)}\sqrt{\n}-1\right|.
\ee
Such quantities are coherent with those defined in \cref{def:three} below and \cite[Definition 1.11]{cgv2025StecconiTodino} (also reported in \cref{def:three} below). Let us also define an additional pointwise quantity:
\be 
\e_1(f,x):=(\max_{u\in S(T_xM)}\frac{\sqrt{\E\kop |\langle d_xf,u\rangle|^2\pok}}{ \lambda(f)}\sqrt{n})-1, \quad \text{as in (5.28)}.
\ee
We will use the latter theorem to prove the next lemma.
\begin{lemma}\label{lem:prooffdiag}
Assume that $f$ is a $\mC^3$ Gaussian field on $M$, with covariance $C$ such that $\e(f)<1$ and for some $\delta, r>0$, one has 
\be\label{eq:shorttt} 
\sup_{\substack{x,y\in M,
\\
\mathrm{dist}(x,y)\ge r
}}
\|j^{1,1}_{x,y}C\|_{g^{\lambda(f)}} \le \delta < \frac{(1-\e(f))^2}{2}.
\ee
Then,
\begin{gather} 
\sum_{q\ge 4}\int_M\int_{B_\rho(x)\smallsetminus B_{r}(x)} |\E \kop  \lfl(dx)[q]\cdot \lfl(dy)[q]\pok |
\\
\le 
\frac{\frac{2^6}{n}}{1-\frac{2\delta}{(1-\e(f))^2}} \frac{\lambda(f)^2 }{(1-\e(f))^8}\int_{r\le \dist xy \le \rho}\|j^{1,1}_{x,y}C\|_{g^{\lambda(f)}} ^4 dxdy.
\end{gather}
\end{lemma}
\begin{proof}
First of all, notice that $\lambda(f,x)=(1+\e_1(f,x))\lambda(f)$, for a function $\e_1(f,\cdot)$ such that $|\e_1(f,x)|\le \e(f)$. Moreover, since $\e(f)<1$, $f$ satisfies \cref{eq:assjet}. Therefore, we can apply \cref{cor:var_dx}, which implies that
\begin{gather}
\sum_{q\ge 4}\int_M\int_{B_\rho(x)\smallsetminus B_{r}(x)} \E \kop  \lfl(dx)[q]\cdot \lfl(dy)[q]\pok 
\le 
\\
\le
\sum_{q\ge 4}\tyu 2\|j^{1,1}_{x,y}C\|_{g^f}\uyt^{q-4} \int_M\int_{B_\rho(x)\smallsetminus B_{r}(x)}\
\frac{\lambda(f)^2(1+\e(f))^2}{n}
\cdot
{2^4\|j^{1,1}_{x,y}C\|_{g^f}^4}
\ dxdy=\dots
\label{eq:conti}
\end{gather}
Under the assumption in \eqref{eq:shorttt},
we have that (see \cref{sec:norms})
\be \label{eq:j11Cgfgl}
\|j^{1,1}_{x,y}C\|_{g^f}\le \|j^{1,1}_{x,y}C\|_{g^{\lambda(f)}}\frac{1}{(1-\e(f))^2}\le \frac{\delta}{(1-\e(f))^2}<\frac12.
\ee
\begin{remark}
    If the eccentricity $\e(f)$ is not smaller than one, this estimate is problematic. 
\end{remark}

(\ref{eq:j11Cgfgl}) ensures the convergence of the sum appearing in \cref{eq:conti}, which is equal to $\frac{1}{1-2\|j^{1,1}_{x,y}C\|_{g^f}}\leq \frac{1}{1-\frac{2\delta}{(1-\varepsilon(f))^2}}$. Then we have
\begin{gather}
\dots\le \frac{1}{1-\frac{2\delta}{(1-\e(f))^2}}\cdot \frac{2^4(1+\e(f))^2\lambda(f)^2}{n}\cdot \frac{1}{(1-\e(f))^8}\cdot\int_M\int_{B_\rho(x)\smallsetminus B_{r}(x)}\
{\|j^{1,1}_{x,y}C\|_{g^\lambda(f)}^4}
\ \dd y\dd x.
\end{gather}
Finally, noting $1+\varepsilon(f)\leq 2$ concludes the proof.
\end{proof}
\subsection{Auxiliary parameters}\label{sec:aux}
Consider a $\mC^3$ centered Gaussian random field $\phi$ defined on a $\mC^2$ compact manifold $M$ of dimension $\n \in \N$, possibly with boundary, and endowed with a $\mC^1$ Riemannian metric $g$. For the moment, let us assume only that $\phi\neq 0$. We recall the following definition from \cite{cgv2025StecconiTodino}. 
\begin{definition}[\texorpdfstring{\cite[Def. 1.11]{cgv2025StecconiTodino}}{}]\label{def:three}
We define three deterministic positive real numbers associated to $M,g,\phi$: the \emph{average variance $\sigma$}, the \emph{average frequency $\lambda$} as follows.
\be\label{eq:sigmalambda} 
\sigma^2=\sigma(\phi)^2:=\fint_M 
\E\kop \phi(x)^2\pok dx,\quad \lambda^2=\lambda(\phi)^2:=\frac{1}{\sigma^2}\fint_M \E\kop \|d_x\phi\|^2\pok dx,
\ee 
and the \emph{maximal eccentricity} of $\phi$:
\bega \label{eq:eps}
\e=\e(\phi):=\max_{x\in M,\ u\in S(T_xM)}\left|\frac{\sqrt{\E\kop |\langle d_x\phi,u\rangle|^2\pok}}{\sigma\lambda}\sqrt{\n}-1\right| 
\\
+\max_{x\in M}\left|\frac{\sqrt{\E\kop \phi(x)^2\pok}}{\sigma}-1\right|
+
\max_{x\in M,\ u\in S(T_xM)}\left|\langle d_x\frac{\sqrt{\E\kop \phi(\cdot)^2\pok}}{\sigma}, u\rangle \right|\frac{\sqrt{\n}}{\lambda}.
\eega
When $\e(\phi)=0$, we say that $\phi$ is a \emph{homothetic field}.
\end{definition}
Let us define auxiliary functions $\sigma,\e_0,\e_1\colon M\to \R$ as follows 
\be 
\sigma(x)^2:=\E\kop \phi(x)^2\pok,\quad \sigma^2:=\fint_M\sigma(x)^2dx \quad \text{and} \quad \e_0(x):=\frac{\sigma(x)}{\sigma}-1
\ee
let also 
\be 
\max_{u\in S(T_xM)}\frac{\|u\|_{{g_x}^{\phi}}}{\sigma}\sqrt{n}=\max_{u\in S(T_xM)}\frac{\sqrt{\E\kop |\langle d_x\phi,u\rangle|^2\pok}}{\sigma}\sqrt{n}=:(1+\e_1(x))\lambda(\phi)
\ee
Then, $\e(\phi)=\max\{\e_0,\e_0',\e_1\}$, where
\be 
\e_i=\max_{x\in M}|\e_i(x)|, \quad \e_0'=\max_{x\in M}\|d_x\e_0\|\frac{\sqrt{n}}{\lambda}.
\ee
Moreover, \cite[Def. 1.9]{cgv2025StecconiTodino} naturally extends $\lambda(\phi,x)$ to the non-unit variance field $\phi$ as:
\begin{gather} 
\lambda(\phi,x)^2:=
\frac{\E\kop \|d_x\phi\|^2\pok}{\sigma^2}
\\
\Big(=\fint_{u\in S(T_xM)}\frac{\|u\|^2_{{g_x}^{\phi}}}{\sigma^2}n \dd u=
\fint_{u\in S(T_xM)}\frac{\E\kop |\langle d_x\phi,u\rangle|^2\pok}{\sigma^2}n \dd u
\Big)
\end{gather}
In the following analysis, we shall substitute $\phi$ with the unit variance field $f$, defined by the identity
\be 
\phi(x)=f(x)\sigma(1+\e_0(x)),
\ee
without changing the random set $\{\phi=0\}=\{f=0\}$. We will use
\be \label{eq:defe03}
\tau(\phi)=\e_0^3(\phi):=\max_{x\in M}\|j^3_x\e_0\|_{g^{\lambda(\phi)}}=\max_{x\in M}\|j^3_x(\frac{\sqrt{\Var\{\phi(\cdot)\}}}{\sigma}-1)\|_{g^{\lambda(\phi)}}.
\ee 
to keep track, to some extent, of the factor $f(x)/\sigma\phi(x)$.
\subsection{Proof of the \cref{thm:drago}}\label{thm:dragodraft}
The next is a technical, overly precise, theorem of which \cref{thm:drago} is an immediate corollary.
\begin{theorem}\label{thm:dragoverkill}
There exists a smooth function $\Oh(t)>-1$, defined for $t\in [0,1)$ such that $\Oh(0)=0$, and such that the following holds.
Let $(M,g)$ be a compact smooth Riemannian manifold of dimension $\n$, possibly with boundary. Let $\phi$ be a Gaussian field of class $\mC^3(M)$ with covariance function $E$ and parameters $\lambda=\lambda(\phi)>0,\sigma=\sigma(\phi)>0,\e=\e(\phi)$ defined as in \cite[Def. 1.11]{cgv2025StecconiTodino} (see also \cref{def:three}), and let $\t:=\t(\phi)$ be as in \cref{eq:defe03}.
Assume that there exist positive numbers: $ \frac{c} {\lambda} \le r\le \rho$, $\delta$ and $B$, satisfying the following assumptions.
\begin{enumerate}
\item The eccentricity is less than one:
\be 
\e=\e(\phi)<1, \quad \text{and} \quad \t=\t(\phi)<1.
\ee
\item The pull-back field satisfies: 
\be\label{eq:hypunifboundlambdac2} 
\sup_{\substack{x,y\in M,
\\
 \mathrm{dist}(x,y)\le r}
}
\sigma^{-2}\|j^{3,3}_{x,y}E\|_{g^{\lambda}}
\le B
\ee
\item The field satisfies:
    \be\label{eq:shortc2} 
\sup_{\substack{x,y\in M,
\\
r
\le \mathrm{dist}(x,y)\le \rho
}}
\sigma^{-2}\|j^{1,1}_{x,y}E\|_{g^{\lambda}}
\le \delta  <\frac{(1-\e)^2}{2\tyu1+\Oh(\t)\uyt} 
\ee
\end{enumerate}
Then, we have the following quantitative Law Of Large Numbers:
\begin{gather}\label{eq:drago_rdw}
\Var\kop \frac{\lf(M)}{\lambda}\pok\le
\Var\kop \frac{\lf(M)[2]}{\lambda}\pok
%
+\mathfrak{C}\rho^{-n}\cdot  \Bigg(\mathfrak{D}r^{n}+
\\
\label{eq:drago_rdw2}
+
\frac{(1+\Oh(\t))(1+\Oh(\e))}{1-\delta\frac{2 (1+\Oh(\t))}{(1-\e)^2}}
\frac{2^4}{n}\int_{\substack{x,y\in M,
\\
r
\le \mathrm{dist}(x,y)\le \rho}}\|\sigma^{-2}j^{1,1}_{x,y}E\|^4_{g^{\lambda}} \dd x \dd y 
\Bigg).
\end{gather}
where $\mathfrak{D}=\mathfrak{D}(n,M,g;B,c)$ is a constant depending only on $n,M,g$ and on the constants $B,c$; $\mathfrak{C}=\mathfrak{C}(n,M,g)$ depends only on the geometry.
\end{theorem}
\begin{proof}[Proof of \cref{thm:drago} given \cref{thm:dragoverkill}]
For any $T > 1$, it is immediate that:$$\epsilon < \frac{1}{T} < 1 \quad \text{and} \quad \tau < \frac{1}{T} < 1.$$
Thus, assumption (1) of \cref{thm:dragoverkill} is satisfied. Similarly, by setting the constant $B$ in \cref{thm:drago} to the value $T$, assumption (2) of \cref{thm:dragoverkill} is identical to assumption (2) in \cref{thm:drago}. For assumption (3) we must show that there exists a $T_0$ such that for all $T \ge T_0$, the threshold in \cref{thm:drago} implies the threshold in \cref{thm:dragoverkill}:$$\frac{1}{4}<\delta < \frac{(1-\varepsilon)^2}{2(1+\mathcal{O}(\tau))}.$$ 
This is true simply because the right hand side goes to $\frac12$ as $\e,\tau\to 0$. 
Now, for the same reason, we note that the term $$P(\epsilon, \tau, \delta) := \frac{(1+\mathcal{O}(\tau))(1+\mathcal{O}(\epsilon))}{1-\delta\frac{2 (1+\mathcal{O}(\tau))}{(1-\epsilon)^2}},$$ 
is bounded by $2$, for $T_0$ large enough. 
 Hence \cref{thm:dragoverkill} implies $$\Var\left[ \frac{\mathcal{L}_f(M)}{\lambda} \right] \le \Var\left[ \frac{\mathcal{L}_f(M)[2]}{\lambda} \right] + \mathfrak{C}\rho^{-n} \left( \mathfrak{D}r^n +  2\frac{2^4}{n} \int_{{\substack{x,y\in M,
\\
r
\le \mathrm{dist}(x,y)\le \rho}}} \|\sigma^{-2} j_{x,y}^{1,1} E\|^4_{g^\lambda} \right).$$
Finally, we define the constant $\mathfrak{T}(n, M, g; T) := \max\left( \mathfrak{C}\mathfrak{D}, \mathfrak{C} 2\frac{2^4}{n} \right)$, to absorb all geometric constants. 
Note that $\mathfrak{D}$ depends on $c$ and $B$, and in our case $c = 1/T$ and $B=T$. This yields the desired bound: $$\Var \kop\frac{\lf(M)}{\lambda} \pok \le \Var\kop \frac{\lf(M)[2]}{\lambda}\pok  + \mathfrak{T} \left[ \left( \frac{r}{\rho} \right)^n + \int_{{\substack{x,y\in M,
\\
r
\le \mathrm{dist}(x,y)\le \rho}}} \|\sigma^{-2} j_{x,y}^{1,1} E\|^4_{g^\lambda}\right].$$
\end{proof}
\begin{proof}[Proof of \cref{thm:dragoverkill}]
In the following, we are going to use the expression $\Oh(t)$ for whatever function is as in the statement of the theorem; it is left intended that it is not always the same function, but that it is redefined anytime needed as one that is larger than all the previous.
Moreover, we will denote $\Oh_1(\e)$ all quantities that are bounded by $\e$, that is, such that $|\Oh_1(\e)|\le \e$. 

We stress that in throughout the whole proof, we denote $\lambda=\lambda(\phi)$ and $\e=\e(\phi)$ and $\e_0^3=\e_0^3(\phi)=\tau(\phi)$.
\begin{enumerate}[(I).]
\item \textbf{Reduction to unit-variance field ($f$ s.t. $\e_0^3(f)=0$).} Let us take up the notation of \cref{sec:aux}. First, we want to substitute $E$ with
\be 
C(x,y)=\frac{E(x,y)}{\sqrt{E(x,x)}\sqrt{E(y,y)}}=\frac{\sigma^{-2}E(x,y)}{(1+\e_0(x))(1+\e_0(y))},
\ee
that is the covariance of a unit variance field $f$, so $\sigma(f)=1$. 
\item \textbf{Comparison of the frequencies:}
We have that 
\bega
\lambda(\phi,x)^2 &=\|d_x\e_0\|^2+(1+\e_0(x))^{2}\lambda(f,x)^2
\\
&=
\frac{\lambda^2}{\n}\Oh_1(\e_0')+(1+\Oh_1(\e_0))^2\lambda(f,x)^2
\eega
Moreover, $\lambda(\phi,x)=\lambda(1+\Oh(\e_1))$, therefore:
\be \label{eq:lambdafix}
\lambda(f,x)=(1+\Oh(\e))(1+\Oh(\e_0^3))\lambda.
\ee
Secondly, we have
\bega \label{eq:inutile1}
\lambda(\phi)^2&=\frac{\lambda(\phi)^2}{n}\fint_M \tyu\frac{|\nabla_x\sigma|}{\sigma\lambda(\phi)}\sqrt{n}\uyt^2 \dd x+ \fint_M\lambda(f,x)^2\tyu\frac{\sigma(x)}{\sigma} \uyt^2 \dd x
\\
&=\lambda(\phi)^2\frac1n\Oh_1(\e'_0)^2+\lambda(f)^2(1+\Oh_1(\e_0))^2
\eega
so 
\be\label{eq:prelambdafix} 
\frac{\lambda(\phi)}{\lambda(f)}\le \frac{1+\e_0}{\sqrt{1-\frac{(\e_0')^2}{n}}}=:1+\e_\lambda\le \frac{1+\e_0^3}{\sqrt{1-(\e_0^3)^2}}=(1+\Oh(\e_0^3))
\ee
\item \textbf{Comparison of the jet norms:}
It follows that (we cannot use $\e$ to control second derivatives: that is why we need $\e_0^3$), for points at distance at most $r$,
\bega \label{eq:EtoC2}
\|j^{3,3}_{x,y}C\|_{g^{\lambda(f)}}&\le \sigma^{-2}\|j^{3,3}_{x,y}E\|_{g^{\lambda(\phi)}}\tyu 1+\Oh(\e_0^3)\uyt
\\
&\le B\tyu 1+\Oh(\e_0^3)\uyt 
\eega
where, in the last line, we assume $k\le 3$.
While for $k=1$, using $\e_0$, we can show that\footnote{We are using the following relations: \be 
\frac{(1+\e_0'-\e_0)^2}{(1-\e_0)^4}=\frac{(1+\frac{\e_0'}{1-\e_0})^2}{(1-\e_0)^2}
\le \frac{(1+\frac{\e_0^3}{1-\e_0^3})^2}{(1-\e_0^3)^2}= \frac{1}{(1-\e_0^3)^4}.
\ee} for points at distance greater than $r$,
\bega \label{eq:EtoC3}
\|j^{1,1}_{x,y}C\|_{g^{\lambda(f)}}
&\le \sigma^{-2}\|j^{1,1}_{x,y}E\|_{g^{\lambda(\phi)}}\frac{(1+\e_0'-\e_0)^2}{(1-\e_0)^4}(1+\e_\lambda)^2
\\
&\le \sigma^{-2}\|j^{1,1}_{x,y}E\|_{g^{\lambda(\phi)}}\tyu 1+\Oh(\e_0^3)\uyt
\\
& \le \delta \tyu 1+\Oh(\e_0^3)\uyt
\eega
Here, the first two inequalities hold for any $x,y\in M$, while the third and last one is true for $r\le \dist xy \le \rho$, by hypothesis (3) of the statement.
\item \textbf{Assumptions in terms of $f$:}
From now on, we will replace $E$ with $C$ and $\phi$ with $f$. We have (1),(2),(3) replaced as follows.
\be 
\e(f)\le \e(\phi)<1, \quad \text{and}\quad \e_0^3(f)=0.
\ee
Taking (a smoothened version of) the maximum of the two terms $\tyu 1+\Oh(\e_0^3)\uyt$ in \cref{eq:EtoC2} and \cref{eq:EtoC3} and using (3), we have that
\be 
\|j^{3,3}_{x,y}C\|_{g^{\lambda(f)}}< B (1+\Oh(\e_0^3))=:B'
\ee
for points at distance at most $r$ and 
\be 
\|j^{1,1}_{x,y}C\|_{g^{\lambda(f)}}< \delta (1+\Oh(\e_0^3))= :\delta' < \frac{(1-\e(f))^2}{2},
\ee
for points at distance comprised between $r$ and $\rho$.
\item \textbf{Semi-localization ($\dist xy\le \rho$).} We can focus on a fixed $\rho$-neighborhood of the diagonal, by observing the following. Let $M_1,\dots, M_\#$ be a partition of $M$ with $\#$ subsets of diameter less than $\rho$, whose union covers $M$ and such that $M_i\cap M_j$ has measure zero. It well known that this can be realized with $\#\le \mathfrak C\rho^{-n}$ for some constant $\mathfrak C= C(n,M,g)$ depending on the geometry of $(M,g)$.\footnote{$\#\le C\rho^{-n}$ where $C$ is Bishop-Gromov's \emph{volume comparison constant}, depending on the dimension, lower Ricci curvature bound, and injectivity radius of $(M,g)$. } Then,
\bega\label{eq:semilocalize}
\Var\kop \frac{\lf(M)}{\lambda}\pok&-
\Var\kop \frac{\lf(M)[2]}{\lambda}\pok
=\Var\kop \sum_{i=1}^{\#}\frac{\lf(M_i)}{\lambda}
-
\frac{\lf(M_i)[2]}{\lambda}\pok
\\
&\le \#\sum_{i=1}^{\#}\Var\kop \frac{\lf(M_i)}{\lambda}
-
\frac{\lf(M_i)[2]}{\lambda}\pok
\\
&\le 
\#\sum_{i=1}^{\#}\sum_{q\ge 4}\int_{M_i\times M_i}\E\kop \frac{\lf(dx)[q]}{\lambda}
\frac{\lf(dy)[q]}{\lambda}\pok
\\
&\le \mathfrak C\rho^{-n}\int_{M}\int_{  B_\rho(x)}\left|\sum_{q\ge 4}\E\kop \frac{\lf(dx)[q]}{\lambda}
\frac{\lf(dy)[q]}{\lambda}\pok\right|
\eega
The last inequality is due to the fact that, by construction, we have an inclusion 
\be 
\bigcup_{i=1}^{\#}M_i\times M_i\subset \kop x,y\in M: \mathrm{dist}(x,y)\le \rho\pok,
\ee
with $(M_i\times M_i)\cap (M_j\times M_j)$ having measure zero.
\item \textbf{The 2 pieces to control.}
It is very clear that, to conclude the proof, we have to estimate the final term in \cref{eq:semilocalize}, which we divide in two pieces:
\be \label{eq:two_pieces}
\mathfrak C\rho^{-n}\int_{M}\int_{  B_\rho(x)}\left|\sum_{q\ge 4}\E\kop \frac{\lf(dx)[q]}{\lambda}\frac{\lf(dy)[q]}{\lambda}\pok\right| \le
\mathfrak C\rho^{-n}\tyu \Sigma_{\text{Near-diagonal}}+\Sigma_{\text{Off-diagonal}}\uyt,
\ee 
where
\bega 
\Sigma_{\text{Near-diagonal}}:&=\int_{M}\int_{  B_r(x)}\left|\sum_{q\ge 4}\E\kop \frac{\lf(dx)[q]}{\lambda}
\frac{\lf(dy)[q]}{\lambda}\pok\right|
\\
\Sigma_{\text{Off-diagonal}}:&=\sum_{q\ge 4}\int_{M}\int_{  B_\rho(x)\setminus B_r(x) }\left|\E\kop \frac{\lf(dx)[q]}{\lambda}
\frac{\lf(dy)[q]}{\lambda}\pok\right|.
\eega
Note that in the second term, we used the triangular inequality.
\item \textbf{Near-Diagonal behavior ($\dist xy \le r$)}.

In the study of the near-diagonal behavior, it is convenient to revert the chaos decomposition, and reduce the study to Kac-Rice formula, which can be controlled as in \cref{lem:prooffdiag}. 
\bega 
\Sigma_{\text{Near-diagonal}}:&=\int_{M}\int_{  B_r(x)}\left|\E\kop \frac{\lf(dx)}{\lambda}
\frac{\lf(dy)}{\lambda} - \frac{\lf(dx)[2]}{\lambda}
\frac{\lf(dy)[2]}{\lambda}\pok\right|
\\
&\le 
\int_{M}\int_{  B_r(x)}\left|\E\kop \frac{\lf(dx)}{\lambda}
\frac{\lf(dy)}{\lambda}\pok\right| +\left|\E\kop\frac{\lf(dx)[2]}{\lambda}
\frac{\lf(dy)[2]}{\lambda}\pok\right|
\\
&\le 
\int_{M}\int_{  B_r(x)}\left|\E\kop \frac{\lf(dx)}{\lambda}
\frac{\lf(dy)}{\lambda}\pok\right| +\frac{4}{\n} \frac{\lambda(f,x)\lambda(f,y)}{\lambda^2}\|j^{1,1}_{x,y}C\|^2_{g^f}
\eega
In the last step, we used \cref{cor:var_dx}. Recalling \cref{eq:lambdafix} and \cref{eq:EtoC3}, we deduce that
\bega 
\Sigma_{\text{Near-diagonal}}:
&\le 
\int_{M}\int_{  B_r(x)}\left|\E\kop \frac{\lf(dx)}{\lambda}
\frac{\lf(dy)}{\lambda}\pok\right| +\frac{4}{\n}(1+\Oh(\e))(1+\Oh(\e_0^3))B^2 d y d x
\eega
\begin{lemma}[see \cref{lem:unifboundlambda}]
Assume that $f$ satisfies (2) for a constant $\lambda >0$. Then, for any $M'\subset M$ Borel subset, we have
\be 
\int_{M'}\int_{  B_r(x)}\left|\E\kop \frac{\lf(dx)}{\lambda}
\frac{\lf(dy)}{\lambda}\pok\right|\dd x \dd y<\mathfrak{K}(n,M,g;B,\lambda r)\vol{n}(M')r^\n.
\ee
\end{lemma}
Thus, recalling that $B'=B(1+\Oh(\e_0^3))$, and that $\lambda r\ge q$ set
\be 
\mathfrak{K}:=\mathfrak{K}(n,M,g;B',c).
\ee
We conclude that there exists a constant $\mathfrak{D}=\mathfrak{D}(n,M,g;B,c)>0$ such that
\bega 
\Sigma_{\text{Near-diagonal}}
&\le 
\mathfrak{K}\vol{n}(M)r^\n+\frac{4}{\n}(1+\Oh(\e))(1+\Oh(\e_0^3))B^2\int_{M}\int_{  B_r(x)} d y d x
\le 
\mathfrak{D}r^\n.
\eega
\item \textbf{Off-diagonal behavior ($r\le \dist xy \le \rho$): fourth chaos.}


Since $\delta' <\frac{(1-\e)^2}2$ and $\e(f)\le \e(\phi)<1$ we have that
\begin{lemma}[See \cref{lem:prooffdiag}]\label{lem:offdiag}
\begin{gather} 
\sum_{q\ge 4}\int_{M}\int_{  B_\rho(x)\smallsetminus B_r(x)}\left|\E\kop \frac{\lf(dx)[q]}{\lambda}
\frac{\lf(dy)[q]}{\lambda}\pok\right|
\\
\le \frac{\lambda(f)^2}{\lambda^2}\frac{\frac{2^4}{n}}{1-\frac{2\delta'}{(1-\e(f))^2}} \frac{(1+\e(f))^2}{(1-\e(f))^8}\int_{\kop\substack{x,y\in M,
\\
r
\le \mathrm{dist}(x,y)\le \rho}\pok}
\|j^{1,1}_{x,y}C\|^4_{g^{\lambda(f)}} \dd x \dd y .
\end{gather}
\end{lemma}

so that
\begin{gather}
\Sigma_{\text{Off-diagonal}}\le
\\
\le 
\frac{(1+\Oh(\e_0^3))(1+\Oh(\e))2^4n^{-1}}{1-\frac{2\delta(1+\Oh(\e_0^3))}{(1-\e)^2}}\int_{B_\rho(x)\smallsetminus B_{r}(x)}\
{\|\sigma^{-2}j^{1,1}_{x,y}E\|_{g^{\lambda(\phi)}}^4}
\ \dd y\dd x.
\end{gather}
\end{enumerate}
\end{proof}
\subsection{Additional remarks on \cref{thm:dragoverkill}}
Notice that the final rate in \cref{eq:drago_rdw}, of \cref{thm:dragoverkill} does not depend on $\delta$ at first order, however it will turn out to be convenient to keep track of it to have a control on the integral in the right-hand-side, given that such quantity is not always accessible. For instance in the case of random waves $E_\ell=E_{[\ell-1,\ell]}$, despite its behavior on the sphere and plane is well known, see \cite{GMT}, we are not aware of any result that provides a sharp control on the asymptotic behavior of $\int_{M\times M}E_\ell^4$, on general manifolds. However, Weyl's law (\cref{thm:weyllaw}) provides a sharp estimate of the form
\be \label{eq:pseudoweyl}
\int_{M\times M}
\tyu \sigma_\ell^{-2}\|j^{1,1}_{x,y}E_\ell\|_{g^{\lambda_\ell}}\uyt^2 \dd x \dd y
\le w_\ell,
\ee 
for a known sequence $w_\ell>0$, for instance in the case of general colorful random waves, this is $w_\ell=\Oh(\ell^{-n})$, see \cref{thm:hoermplaw} and \cref{eq:NEEsig}. This can be used in combination with \cref{thm:dragoverkill} to give an upper bound to \eqref{eq:drago_rdw2}.
\begin{proposition}
Let the hypotheses of \cref{thm:dragoverkill} hold for each $\ell$, with respect to $r=r_\ell, \lambda=\lambda_\ell,\delta=\delta_\ell,\sigma=\sigma_\ell$. If moreover, the field satisfies an analogous of the Weyl law, meaning that \cref{eq:pseudoweyl} holds for some sequence $w_\ell$,
then
\be
 \int_{M\times M}\|\sigma_\ell^{-2}j^{1,1}_{x,y}E_\ell\|^4_{\lambda_\ell^2 g} \dd x \dd y\le O(1)r_\ell^{\n}
 +w_\ell \delta_\ell^2.
\ee
\end{proposition}
\begin{proof}
We split the integral in two parts: $$\int_{M\times M} \|\sigma_\ell^{-2}j^{1,1}_{x,y}E_\ell\|^4_{\lambda_\ell^2 g} = \underbrace{\iint_{\text{dist}(x,y) < r} \|\sigma_\ell^{-2}j^{1,1}_{x,y}E_\ell\|^4_{\lambda_\ell^2 g}}_{\mathcal{I}_{\text{near}}} + \underbrace{\iint_{\text{dist}(x,y) \ge r} \|\sigma_\ell^{-2}j^{1,1}_{x,y}E_\ell\|^4_{\lambda_\ell^2 g}}_{\mathcal{I}_{\text{off}}}.$$ For the part near the diagonal $\mathcal{I}_{\text{near}}$ we use assumption (2) to get $$\sigma_\ell^{-2} \|j^{1,1}_{x,y}E\|_{g^\lambda} \le \sigma_\ell^{-2} \|j^{3,3}_{x,y}E\|_{g^{\lambda_\ell}} \le B.$$ Then $$\mathcal{I}_{\text{near}} \le B^4 \cdot \text{Vol}(\{(x,y) : \text{dist}(x,y) < r\})$$ which scales as $O(r^n)$. 
For the part off-diagonal $\mathcal{I}_{\text{off}}$ we write the four power as the product of two squares and we exploit assumption (3) to one of the two factors and \cref{eq:pseudoweyl} to the remaining integral to get $$\mathcal{I}_{\text{off}} \le \delta^2 \cdot w_\ell,$$ concluding the proof.
\end{proof}
\begin{remark}
Given $f$ non-degenerate and an arbitrary positive smooth function $\sigma$ of class $\mC^\infty(M,(0,+\infty))$, one can always define a field $\phi(x)=\sigma(x)f(x)$ that satisfies the identity at \cref{eq:inutile1} and that will be non-degenerate. In particular, this implies that for any $\phi$ we have
\be 
\lambda(\phi)^2\tyu \fint_M \sigma(x)^2 \dd x\uyt >\fint_M |\nabla_x\sigma|^2 \dd x. 
\ee
\end{remark}

\begin{remark}
About \cref{eq:prelambdafix} in the proof. There is no way to control uniformly the ratio $\frac{\lambda(\phi)}{\lambda(f)}$ using only $\tau$, without assuming that $\tau<1$. The denominator in \cref{eq:prelambdafix} \emph{can} be arbitrary large. It is possible to construct a sequence of Gaussian fields $\phi_\ell$ such that 
\be 
\sup_\ell \frac{\lambda(\phi_\ell)}{\lambda(f_\ell)}=+\infty, \quad \text{with } \quad \sup_\ell \t(\phi_\ell)\le \sqrt{n}.
\ee
\end{remark}

\section{General scheme for scaling bounds and scaling limits}
\subsection{General definition of scaling limit}
In this section, we shall consider a sequence $\phi_\ell$ of $\mC^k$, indexed by $\ell\in \N$, Gaussian random fields on a given Riemannian manifold $(M,g)$ of dimension $\n$. For this section, we are not going to restrict to the case or Riemannian random waves, but we will assume that a scaling limit property holds. In fact, we are now going to define in a very precise way two notions of \emph{scaling limit} and \emph{scaling bound}. While we might pay the price of risking to be overly pedantic, we will gain a whole language to distinguish the subtleties among the existing different instances of, vaguely speaking, scaling limits. The reason being that some of these aspects make all the difference in our analysis, towards a full understanding of \cref{probYauconj}.


 We can represent the covariance function $E_\ell(x,y)=\E\kop \phi_\ell(x)\phi_\ell(y)\pok$, that is a function of class $\mC^{k,k}$ (see \cref{sec:prelim_jets}), as follows:
    \begin{gather}\label{eq:scalimsetting}
E_\ell(x,y)=\sigma_{\ell}^2 \tyu 
K\tyu \lambda_\ell\cdot \dist x y \uyt
+
R_\ell(x,y)
\uyt, \quad \text{and define}
\\
\sup_{\rho'_\ell\le \mathrm{dist}(x,y)\le \rho_\ell}\sup_{|\a|,|\beta|\le k} \frac{\| \nabla^\a \nabla^\beta R_{\ell}(x,y)\|}{\lambda_\ell^{|\a|+|\beta|}}\sqrt{n}^{|\a|+|\beta|}=: \delta_{\ell};
    \end{gather}

    
for some set of sequences $\rho'_\ell\le \rho_\ell,\sigma_\ell,\lambda_\ell>0$ and $\delta_\ell=\delta_\ell(E_\ell,K; \rho'_\ell,\rho_\ell,\sigma_\ell,\lambda_\ell)\ge 0$; a function $R_\ell\in \mC^{k,k}(M\times M)$; and where $K\colon \R\to \R$ is the covariance function of a stationary $\mC^k$ Gaussian field $F$ on $\R^n$, that is, $K(|u|)=\E\kop F(x)F(x+u) \pok$, for all $x,u\in \R^n$.
The variance and frequency of the field $F$ are thus 
\be 
\sigma(F)^2=\Var\kop F(x)\pok=K(0), \quad \text{and} \quad \lambda(F)^2=\sigma(F)^{-2}\E\kop \|\nabla_x F\|^2\pok=\frac{-nK''(0)}{K(0)}.
\ee
Moreover, $\E\kop F(0) \nabla_x F\pok=0$ and $K'(0)=0$.
\begin{definition}
If $K(0)=1$ and $-nK''(0)=1$, we say that $F$ is \emph{Berry-normal}. 
\end{definition}
The term \emph{normal field} is already taken and has a different meaning: it is used in \cite{ersz2024MathisStecconi,cpcf2024MarinucciStecconi} to refer to homothetic Gaussian fields $f$ with $\sigma(f)=\lambda(f)^2\n^{-1}=1$. 

One can always renormalize $K$ to be Berry-normal, by replacing it with $\sigma(F)^{-2}K(\lambda(F)^{-1}(\cdot))$. In what follows it will be convenient to work with such normalized fields, in order to get rid of annoying constants.
\begin{definition}[Scaling Bound and Scaling Limit]\label{def:scalim}
We say that $E_\ell$ (or $\phi_\ell$) satisfies a $\mC^k$ \emph{scaling bound}, with
\emph{translation invariant limit $K$ (or $F$)}, with \emph{parameters}: $\mathsf P_\ell=(K, \sigma_\ell,\lambda_\ell, \rho'_\ell,\rho_\ell)$ if \cref{eq:scalimsetting} holds. We say that the scaling bound holds with \emph{effective variance $(\sigma_\ell)_\ell$}, with \emph{effective frequency} $(\lambda_\ell)_\ell$, with \emph{remainder function} $(R_\ell)_\ell$, in the \emph{range} $(\rho'_\ell,\rho_\ell)$, with \emph{error} $\delta_\ell$.
If in addition, we have that
\be 
\delta_\ell\to_{\ell \to +\infty} 0,
\ee
then we shall say \emph{scaling limit} instead of scaling bound. 
\end{definition}
\begin{remark}\label{rem:jet11}
In particular, a $\mC^1$ scaling bound in the range $(0,0)$, with error $\delta_\ell$ means that 
\be 
j^{1,1}_{x,x}E_\ell=\sigma_\ell^2\begin{pmatrix}
    \sigma(F)^2+\Oh_1(\delta_\ell) & \lambda_\ell \Oh_1(\delta_\ell) 
    \\ \lambda_\ell \Oh_1(\delta_\ell)  & \frac{\lambda_\ell^2}{n}\tyu \lambda(F)^2 g_x +\Oh_1(\delta_\ell)\uyt,
\end{pmatrix}
\ee
where $\Oh_1(\delta_\ell)\le \delta_\ell$. If $K$ is Berry-normal, then $\sigma(F)^2=1=\lambda(F)^2$. Therefore, we have that 
\be \label{eq:compare_lambda_scalim}
\sigma(\phi_\ell)^2=\sigma_\ell^2 \tyu \sigma(F)^2 +\Oh_1(\delta_\ell)\uyt, \quad \text{and} \quad \lambda(\phi_\ell)^2=\lambda_\ell^2 \tyu \lambda(F)^2 + \Oh_1(\delta_\ell)\uyt.
\ee
\end{remark}
\subsection{Statement for general scaling limits.}
\begin{theorem}\label{thm:scasca}
    Assume that $\phi_\ell\randin \mC^\infty(M)$ is a sequences of Gaussian fields that satisfies a $\mC^3$ scaling limit as in \cref{def:scalim}, in the range $(0,\rho)$, with parameters $\mathsf P_\ell=(K, \sigma_\ell,\lambda_\ell, 0,\rho)$. Assume that
    \be\label{eq:limsupK} 
\limsup_{r \to +\infty} \max_{ r\le u }\max_{0\le i\le 6} |K^{(i)}(u)|=0
    \ee
    Then, there is $r_0=r_0(n,M,g;K)>0$ such that for any sequence $r_0\le r_\ell\le \rho \lambda_\ell$, we have 
\begin{gather}
\Var\kop \frac{\m L_{\phi_\ell}(M)}{\lambda_\ell}\pok\le 
\Var\kop \frac{\m L_{\phi_\ell}(M)[2]}{\lambda_\ell}\pok
+\Oh(1)\Bigg( r_\ell^{n}\lambda_\ell^{-n}
+
\int_{\substack{x,y\in M,
\\
\frac{r_\ell}{\lambda_\ell}
\le \mathrm{dist}(x,y)\le \rho}}\|\sigma_\ell^{-2}j^{1,1}_{x,y}E_\ell\|^4_{g^{\lambda_\ell}} \dd x \dd y\Bigg).
\end{gather} 
where $\Oh(1)$ denotes a bounded sequence.
\end{theorem}
The proof of \cref{thm:scasca} is in the next subsection, see \cref{sec:scascaproof}.


\begin{corollary}\label{cor:scaberry}
    Assume that $\phi_\ell\randin \mC^\infty(M)$ is a sequence of Gaussian fields that satisfies a $\mC^3$ scaling limit as in \cref{def:scalim}, in the range $(0,\rho)$, with \emph{parameters}: $\mathsf P_\ell=(K, \sigma_\ell,\lambda_\ell, 0,\rho)$. Assume that the limit field $F$ is $\phi_{[\a,1]}^{\R^n}$ for some $\a\in [0,1)$, or that it is $\phi_{1}^{\R^n}$ and that $n\ge 2$.
    Then, there is $r_0>0$ such that for any sequence $r_0\le r_\ell\le \rho \lambda_\ell$, we have 
\begin{gather}
\Var\kop \frac{\m L_{\phi_\ell}(M)}{\lambda_\ell}\pok\le 
\Var\kop \frac{\m L_{\phi_\ell}(M)[2]}{\lambda_\ell}\pok
+\Oh(1)\Bigg( r_\ell^{n}\lambda_\ell^{-n}
+
\int_{\substack{x,y\in M,
\\
\frac{r_\ell}{\lambda_\ell}
\le \mathrm{dist}(x,y)\le \rho}}\|\sigma_\ell^{-2}j^{1,1}_{x,y}E\|^4_{g^{\lambda_\ell}} \dd x \dd y 
\Bigg).
\end{gather} 
where $\Oh(1)$ denotes a bounded sequence.
\end{corollary}
\begin{proof}[Proof of \cref{cor:scaberry} given \cref{thm:scasca}]
By \cref{lem:bessympt}, one immediately sees that the covariance $K$ of the limit field  satisfies \cref{eq:limsupK} in all cases considered in the assumptions.
\end{proof}
\begin{proposition}[Addendum]\label{prop:addendum}
In the context of \cref{thm:scasca} or \cref{cor:scaberry} we have the following additional relations.
\begin{enumerate}
\item There exists $r_*>0$ such that
\be 
0<\liminf_\ell\inf_{\dist xy \le \frac{r_*}{\lambda_\ell}}\|\sigma_\ell^{-2}j^{1,1}_{x,y}E_\ell\|_{g^{\lambda_\ell}}, \quad  \quad 
\limsup_\ell\sup_{\dist xy \le \rho}\|\sigma_\ell^{-2}j^{1,1}_{x,y}E_\ell\|_{g^{\lambda_\ell}}<+\infty.
\ee
\item
\be 
\lambda_\ell^{-n}\le \Oh(1)\int_{\substack{x,y\in M,
\\
\mathrm{dist}(x,y)\le 
\frac{r_\ell}{\lambda_\ell}}}
\|\sigma_\ell^{-2}j^{1,1}_{x,y}E_\ell\|^4_{g^{\lambda_\ell}}
\dd x \dd y\le \Oh(1)r_\ell^n \lambda_\ell^{-n}.
\ee
\item If $r_\ell$ is constant, then the conclusion can be equivalently stated as
 \begin{gather}
\Var\kop \frac{\m L_{\phi_\ell}(M)}{\lambda_\ell}\pok\le 
\Var\kop \frac{\m L_{\phi_\ell}(M)[2]}{\lambda_\ell}\pok
+\Oh(1)\Bigg(\int_{M\times M}
\|\sigma_\ell^{-2}j^{1,1}_{x,y}E_\ell\|^4_{g^{\lambda_\ell}} \dd x \dd y 
\Bigg).
\end{gather} 
\item An implied weaker bound is
\be 
\Var\kop \frac{\m L_{\phi_\ell}(M)}{\lambda_\ell}\pok
\le 
\Oh(1)\Bigg(
\int_{M\times M}\|\sigma^{-2}j^{1,1}_{x,y}E_\ell\|^2_{g^{\lambda_\ell}} \dd x \dd y 
\Bigg).
\ee
\item The actual variance, frequency and eccentricity of the field $\phi_\ell$ satisfy
\be 
\|\frac{\sigma(\phi_\ell)}{\sigma_\ell\sigma(F)}-1\|\le \delta_\ell,\quad \|\frac{\lambda(\phi_\ell)}{\lambda_\ell\lambda(F)}-1\|\le \delta_\ell, \quad \e(\phi_\ell)\le 8\delta_\ell.
\ee
The latter holds for any scaling bound, not necessarily a scaling limit.
\end{enumerate}

\end{proposition}
\begin{proof} 
The scaling limit condition, ensures that there exists $r_*>0$, possibly $r_*\le r_0$, and $c>0$ such that if $\dist xy \le \frac{r_*}{\lambda_\ell}$ then $\|\sigma_\ell^{-2}j^{1,1}_{x,y}E_\ell\|_{g^{\lambda_\ell}}\ge c$, and $\|\sigma_\ell^{-2}j^{1,1}_{x,y}E_\ell\|_{g^{\lambda_\ell}}$ is uniformly bounded; this is (1). (1)$\implies$(2)$\implies$(3)$\implies$(4) are straightforward. 
Here, for the first term one can use \cref{cor:var_dx} and \cref{eq:EtoC3}. 
Point (5) follows from \cref{rem:jet11} and \cref{lem:USEFULREL}. A more detailed justification of the claims made in this proof, can be deduced by inspecting the proof of \cref{thm:scasca}. 
\end{proof}

\subsubsection{Decorrelation control}
\begin{lemma}\label{cor:decor}
Assume that $\phi_\ell\randin \mC^\infty(M)$ is a sequences of Gaussian fields that satisfies a $\mC^k$ scaling limit as in \cref{def:scalim}, with parameters $\mathsf P_\ell=(K, \sigma_\ell,\lambda_\ell, 0,\rho)$. Let $\|j^{2k}_rK\|=\max_{0\le i\le 2k}|\frac{d^i}{dr^{i}}K(r)|$ and let $\mathfrak{D}_\ell(r)$ denote the error in the range $(\frac{r}{\lambda_\ell},\frac{r}{\lambda_\ell})$. Then, 
\be 
\sup_{\substack{x,y\in M,
\\
\mathrm{dist}(x,y)= \frac{r}{\lambda_\ell}}
}\sigma(\phi_\ell)^{-2}\|j^{k,k}_{x,y}E_\ell\|_ {g^{\lambda(\phi_\ell)}}
\le  
\Oh(1)\max\kop\|j^{2k}_rK\|,\mathfrak{D}_\ell(r)\pok.
\ee
\end{lemma}
\begin{proof}
\bega 
&\sup_{\substack{x,y\in M,
\\
\mathrm{dist}(x,y)= \frac{r}{\lambda_\ell}}
}
\sigma(\phi_\ell)^{-2}\|j^{k,k}_{x,y}E_\ell\|_ {g^{\lambda(\phi_\ell)}}
\\
&=
\sup_{\substack{x,y\in M,
\\
 \mathrm{dist}(x,y)= \frac{r}{\lambda_\ell}}
} \frac{\sigma_{\ell}^2}{\sigma(\phi_\ell)^2} \|j^{k,k}_{x,y}\tyu 
K\tyu \lambda_\ell\cdot \dist x y \uyt
+
R_\ell(x,y)
\uyt\|_{g^{\lambda_\ell}}\max_{a\in \{0,\dots,2k\}}\frac{\lambda_\ell^a}{\lambda{(\phi_\ell)}^a}
\\
&\le 
\sup_{\substack{x,y\in M,
\\
 \mathrm{dist}(x,y)= \frac{r}{\lambda_\ell}}
} (1+\Oh(\mathfrak{D}_\ell(r))) \tyu \|j^{k,k}_{x,y}\tyu 
K\tyu \lambda_\ell\cdot \dist x y \uyt
\uyt\|_{g^{\lambda_\ell}}+ \Oh(\mathfrak{D}_\ell(r))\uyt
\\
&\le 
\sup_{\substack{x,y\in M,
\\
 \mathrm{dist}(x,y)= \frac{r}{\lambda_\ell}}
} (1+\Oh(\mathfrak{D}_\ell(r))) \tyu \|j^{k,k}_{x,y} \mathrm{dist}\|_{g^{1}}^6 \cdot\|j^{2k}_rK\|+ \Oh(\mathfrak{D}_\ell(r))\uyt
\eega
Compactness of $M$ ensures that for any $k\in \N$ there is a geometric constant $G=G(\n,M,g,k)$ such that $\|j^{k,k}_{x,y} \mathrm{dist}\|_{g^{1}}^6\le G$ for all $x\neq y \in M$, given that $\mathrm{dist}^2$ is smooth on the whole $M\times M$. 
Then,
\bega \label{eq:mah}
\sup_{\substack{x,y\in M,
\\
\mathrm{dist}(x,y)=\frac{r}{\lambda_\ell}
}}
\sigma^{-2}\|j^{1,1}_{x,y}E_\ell\|_{g^{\lambda_\ell}}
&\le 
2\tyu G \cdot \|j^{2k}_rK\|+ \Oh(\mathfrak{D}_\ell(r))\uyt.
\eega
\end{proof}
\subsection{Proof of \texorpdfstring{\cref{thm:scasca}}{}}\label{sec:scascaproof}
The assumptions imply that $\phi_\ell$ satisfies a $\mC^3$ scaling limit, in any range $[a,b]\subset [0,\rho)$,  with \emph{parameters}: $\mathsf P_\ell=(K, \sigma_\ell,\lambda_\ell, a,b)$. Let us denote the errors in such a scaling limit as $\delta_\ell(a,b)$. That means that 
\be 
\sup_{a\le \mathrm{dist}(x,y)\le b}\|j^{3,3}_{x,y}R_\ell\|_{g^{\lambda_\ell}}=\delta_\ell(a,b)\le \delta_\ell(0,\rho)\to_{\ell\to +\infty} 0.
\ee
Moreover, we have an additional sequence $r_\ell$ such that $0< r_\ell\le \rho\lambda_\ell$.
To keep the proof open to future generalizations, we will distinguish three regimes: the error in the covariance of the first jet $\delta_\ell^0=\delta_\ell(0,0)$; the near-diagonal regime $\delta_\ell^\Delta=\delta_\ell(0,\frac{r_\ell}{\lambda_\ell})$; and the off-diagonal regime $\delta_\ell^\infty=\delta_\ell(\frac{r_\ell}{\lambda_\ell},\rho)$.
The additional assumption of the theorem is:
\be 
\sup_r\max_{u\ge r}\max_{0\le i\le 6} |K^{(i)}(u)|<+\infty;
\ee
we will use it to verify condition (2) of \cref{thm:drago}. We shall assume, without loss of generality, that $K$ is Berry-normal.
\subsubsection{Variance and frequency}
Here and henceforth, we will denote $\Oh_1(a)$ any quantity that satisfies $|\Oh_1(a)|\le a$.

By standard arguments we can see that, the metric and the covariance function $K$ satisfy the Taylor approximation:
\be 
\dist{\exp_x(u)}{\exp_x(v)}^2=\|u-v\|^2+\Oh\tyu \|(u,v)\|^4\uyt; \quad K(t)=1-\frac{1}{n}\frac{t^2}{2}+\Oh(t^4),
\ee
Therefore \cref{eq:scalimsetting} implies that we have
\be 
j^{1,1}_{x,x}E_\ell=\sigma_\ell^2\tyu \begin{pmatrix}
    1 & 0 
    \\ 0  & \frac{\lambda_\ell^2}{n} g_x
\end{pmatrix}
+ j^{1,1}_{x,x}R_\ell\uyt = \sigma_\ell^2\tyu \begin{pmatrix}
    1 & 0 
    \\ 0  & \frac{\lambda_\ell^2}{n} g_x
\end{pmatrix}
+ \begin{pmatrix}
    \Oh_1(\delta_\ell^0) & \Oh_1\tyu\delta_\ell^0\frac{\lambda_\ell}{\sqrt{n}}\uyt
    \\ \Oh_1\tyu\delta_\ell^0\frac{\lambda_\ell}{\sqrt{n}}\uyt  & \Oh_1\tyu\delta_\ell^0\frac{\lambda_\ell^2}{n}\uyt
\end{pmatrix}\uyt 
\ee
In particular, then $|\tr_g(\nabla^{(1,1)}_{x,x}R_\ell)|\le \delta_\ell^0 \lambda_\ell^2$, hence we can conclude that
\be 
\sigma(\phi_\ell)=\sigma_\ell(1+\Oh_1(\delta_\ell^0)),\quad \lambda(\phi_\ell)=\lambda_\ell(1+\Oh_1(\delta_\ell^0)).
\ee
in the following, we will check that the hypotheses of \cref{thm:drago} are satisfied.
\subsubsection{(1) of \cref{thm:drago}} 
By definition,
\begin{lemma}\label{lem:USEFULREL}
\be 
2\e(\phi_\ell)\le \max_{x\in M} \left\|\sigma(\phi_\ell)^{-2} j^{1,1}_{x,x}E_\ell - \begin{pmatrix}
    1 & 0 
    \\ 0  & \frac{\lambda(\phi)^2}{n} g_x
\end{pmatrix}\right\|_{g^{\lambda(\phi)}}=:\delta(\phi_\ell)\le 16\delta_\ell^0
\ee
\end{lemma}
(There is a difference is in how the norms are measured for $\e(\phi_\ell)$ and for $\delta_\ell^0$.)
\begin{proof}
Within this proof, let us drop the subscript $\ell$ and fix $E=E_\ell$, $\phi=\phi_\ell$.
Let us take up the notations introduced in \cref{sec:aux}.
 Then, $1+\e_0(x):=\sqrt{\sigma(\phi)^{-2}E(x,x)}$ and $\e=\max\{\e_1,\e_0,\e_0'\},$ and $\delta(\phi):=\max\{ \delta_1, \delta_0,\delta_0'\}$, where
\be 
\e_0= \max_{x\in M} \e_0(x), \quad  \delta_0:=\max_{x\in M} (1+\e_0(x))^2-1=\e_0^2+2\e_0
\ee 
\be 
\e_0'= \max_{x\in M} \|\nabla_x\e_0\|\frac{\sqrt{n}}{\lambda(\phi)}, \quad  \delta_0':=\max_{x\in M} \|\nabla_x(\e_0^2+2\e_0)\|\frac{\sqrt{n}}{\lambda(\phi)}
\ge 2\e_0'
\ee 
Let us define the positive definite scalar product $b_x:=\frac{n}{\sigma^(\phi)^2\lambda(\phi)^2}\nabla_{(x,x)}^{(1,1)}E$. Then, 
\be  
\e_1=\max_{x\in M}\max_{u\in S(T_xM)} |\sqrt{b_x(u,u)}-1| , 
\quad 
\delta_1:=\max_{x\in M}\max_{u,v \in S(T_xM)} |b_x(u,v)-g_x(u,v)|=\e_1^2+2\e_1
\ee
This shows that $2\e\le \delta(\phi)$. 
Now, 
\bega 
\delta(\phi)&=
\left\|\sigma(\phi)^{-2} j^{1,1}_{x,x}E - \begin{pmatrix}
    1 & 0 
    \\ 0  & \frac{\lambda(\phi)^2}{n} g_x
\end{pmatrix}\right\|_{g^{\lambda(\phi)}}
\\
&=
\left\|
\frac{\sigma_\ell^2}{\sigma(\phi)^2} j^{1,1}_{x,x}R_\ell+\tyu\frac{\sigma_\ell^2}{\sigma(\phi)^2}-1\uyt \begin{pmatrix}
    1 & 0 
    \\ 0  & \frac{\lambda(\phi)^2}{\lambda_\ell^2}\frac{\lambda_\ell^2}{n} g_x
\end{pmatrix}\right\|_{g^{\lambda(\phi)}}
\\
&\le \delta_\ell^0(1+\delta_\ell^0)^2+\delta_\ell^0(2+\delta_\ell^0)(1+\delta_\ell^0)^2\le 16 \delta_\ell^0.
\eega
\end{proof}
Arguing similarly, we have 
\be 
\t(\phi):=\max_{x\in M}\|j^3_x(\frac{\sqrt{E(x,x)}}{\sigma}-1)\|_{g^{\lambda(\phi)}}=\Oh(\delta_\ell^0).
\ee
So, for $\ell$ big enough and $T\ge T_0$, the field has small eccentricity:
\be 
\max\{\e(\phi_\ell),\t(\phi_\ell)\}=\Oh(\delta_{\ell}^0)\le \Oh(\delta_\ell)<\frac 1T,
\ee
hence condition (1) of \cref{thm:drago} is satisfied.
\subsubsection{(2) of \cref{thm:drago}: compactness in \texorpdfstring{$(0,\frac{r_\ell}{\lambda_\ell})$}{
}}

Here, $\delta(0,\frac{r_\ell}{\lambda_\ell})=\delta_\ell^\Delta$.
By \cref{cor:decor} we have that 
\bega 
&\sup_{\substack{x,y\in M,
\\
\mathrm{dist}(x,y)\le \frac{r_\ell}{\lambda_\ell}}
}
\sigma(\phi_\ell)^{-2}\|j^{3,3}_{x,y}E_\ell\|_ {g^{\lambda(\phi_\ell)}}
\le \Oh(1)\max\kop\|j^{6}_uK\|,\mathfrak{D}_\ell(u)\pok.
\eega

Moreover, under the assumption at \cref{eq:limsupK}, we have that $\|j^6_{u}K\|=\max_{0\le i\le 6} |K^{(i)}(u)|$ is a continuous function that goes to zero as $u\to+\infty$, therefore it is bounded uniformly for every $u\in \R$ and we have:
\be 
{ \sup_{0\le u \le r_\ell}\|j^6_uK\|\le \sup_{u\in \R}\max_{0\le i\le 6} |K^{(i)}(u)| =:\|j^6K\|_\infty<+\infty};
\ee
Since $\delta_\ell^0$ and $\delta_\ell^\Delta$ are also bounded, we deduce that there exists $T>0$ big enough, such that
\be 
\sup_{\substack{x,y\in M,
\\
 \mathrm{dist}(x,y)\le \frac{r_\ell}{\lambda_\ell}}
}
\sigma(\phi_\ell)^{-2}\|j^{3,3}_{x,y}E_\ell\|_{g^{\lambda(\phi_\ell)}} \le T.
\ee
Therefore, (2) of \cref{thm:drago} is verified.
\subsubsection{(3) of \cref{thm:drago}: fourth chaos off-diagonal}
Recall that $\delta^\infty_\ell=\delta_\ell(\frac{r_\ell}{\lambda_\ell},\rho)$. Arguing exactly as in the previous section, we have 
\begt \label{eq:rzero}
\sup_{\substack{x,y\in M,
\\
\frac{r_\ell}{\lambda_\ell}
\le \mathrm{dist}(x,y)\le \rho
}}
\sigma^{-2}\|j^{1,1}_{x,y}E_\ell\|_{g^{\lambda_\ell}}
\le 
\sup_{\substack{x,y\in M,
\\
\frac{r_\ell}{\lambda_\ell}
\le \mathrm{dist}(x,y)\le \rho
}}
(1+\Oh(\delta_\ell^0)) \tyu G \cdot\sup_{r_\ell \le u \le \rho \lambda_\ell}\|j^2_uK\|+ \delta_\ell^\infty\uyt
\\
\le 
(1+\Oh(\delta_\ell^0))(\Oh(\mathfrak{D}(r_\ell))+\Oh(\delta_\ell^\infty))
\\
\le 
\Oh(\mathfrak{D}(r_\ell))+\Oh(\delta_\ell)
\eegt
The assumption at \cref{eq:limsupK} ensures that we can take $r_0>0$ big enough, depending on $G$ and on $K$, such that 
\be 
\mathfrak{D}(r_\ell)=\max_{r_\ell\le u\le \rho \lambda_\ell}\|j^2_uK\|\le \max_{r_0\le u}\max_{0\le i\le 2} |K^{(i)}(u)|=G^{-1}\frac{1}{8},
\ee
where $r_0\le r_\ell$ as assumed.
For $\ell$ big enough, we have that $\delta^\infty_\ell,\delta^0_\ell\le \frac{1}{T}$, so we deduce that for such $\ell$ large enough, we have 
\be 
\sup_{\substack{x,y\in M,
\\
\frac{r_\ell}{\lambda_\ell}
\le \mathrm{dist}(x,y)\le \rho
}}
\sigma(\phi_\ell)^{-2}\|j^{1,1}_{x,y}E_\ell\|_{g^{\lambda(\phi_\ell)}}
<\frac18+\Oh(\delta_\ell^0)+\Oh(\delta_\ell^\infty)\le \frac{1}{4}
\ee
provided that $T\ge 8$ was chosen at the beginning of the proof.
This ensures (3) of \cref{thm:drago}. 
\begin{proof}[Conclusion of the proof of \cref{thm:scasca}]
We conclude by applying \cref{thm:drago} that 
\begin{gather}
\Var\kop \frac{\m L_{\phi_\ell}(M)}{\lambda(\phi_\ell)}\pok\le 
\Var\kop \frac{\lf(M)[2]}{\lambda(\phi_\ell)}\pok
+\Oh(1) r_\ell^{n}\lambda_\ell^{-n}+
\\
+
\Oh(1)\int_{\substack{x,y\in M,
\\
\frac{r_\ell}{\lambda_\ell}
\le \mathrm{dist}(x,y)\le \rho}}\|\sigma^{-2}j^{1,1}_{x,y}E\|^4_{g^{\lambda}} \dd x \dd y 
\Bigg).
\end{gather}
\end{proof}
\section{The scaling limits of random waves}\label{sec:scalimrandomwaves}
\subsection{Remarks on the scaling bound of Random waves: \texorpdfstring{\cref{def:scalimRW}}{}}
Recall that (see \cref{sec:intervals}) we want to consider sequences of intervals $I_\ell$
of the form 
\be \label{eq:Il}
I_\ell=[\a\ell-\beta_\ell,\ell]=[\ell-\eta_\ell, \ell]
\ee
where $\a\in [0,1]$, $\beta_\ell=o(\ell)$, or equivalently such that $\exists\lim_{\ell\to +\infty}\ell^{-1}\eta_\ell=1-\a\in [0,1]$, where $\eta_\ell=|I_\ell|$ denotes the amplitude of the interval.
Then, we have that, {for $I=(\alpha,1)$,}
\be 
\lim_{\a\to 1}E^n_{\mathsf U (\a,1)} =E^n_{\mathsf U (1,1)}=E^n_{\mathsf U (\{1\})}=E^n_{1}\neq E^n_{\{1\}}.
\ee
It is thus intuitively clear that, by writing the formulas in terms of $E^n_{\mathsf U (\a,1)}$, in place of $E^n_{[\a,1]}$, we will see a more \emph{continuous} behavior with respect to different values of $\a$. Indeed, such notation allows us to write the scaling bound of random waves in the same form for all $\a\in [0,1]$, without the need to single out the case $\a=1$.
\begin{lemma} \label{lem:lambdasigma_Eucli_RRW}
If $\a\in [0,1)$, we have that
    \begt 
\sigma\tyu \phi_{\mathsf{U}(\a,1)}^{\R^n}\uyt^2=E^n_{\mathsf U (\a,1)} (0)=\frac{s_{n-1}}{(2\pi)^n}\frac{1-\a^n}{n(1-\a)} \to_{\a\to 1} \frac{s_{n-1}}{(2\pi)^n}, \quad \text{and} 
\\
\lambda \tyu \phi_{\mathsf{U}(\a,1)}^{\R^n}\uyt^2=-n\frac{(E^n_{\mathsf U (\a,1)})'' (0)}{E^n_{\mathsf U (\a,1)} (0)}=
\frac{n}{n+2} \cdot \frac{1-\a^{n+2}}{1-\a^n}
\to_{\a\to 1}
1.
    \eegt
Moreover, both the above functions are continuous with respect to $\a$, for $\a\in [0,1]$. In addition, $\lambda \tyu \phi_{\mathsf{U}(\a,1)}^{\R^n}\uyt<1$ for all $\a\in [0,1)$.
\end{lemma}

    \begin{remark}
Notice that the \cref{def:scalimRW} is actually \emph{static}: it depends only on the $\ell$-th term of the sequences. In particular, the functions: 
\be 
(I,\a,k,\rho',\rho ) \mapsto  \delta_I^M(\a,\mC^k,[\rho',\rho]), \quad \text{and} \quad (\a,I)\mapsto R^M_{\a,I}\in \mC^\infty(M\times M),
\ee 
are two well defined isometry invariants of $(M,g)$. However, the above scaling bound of \emph{type} $\a$ is meaningful only for a sequence of intervals $I_\ell$ with a limit: $[\a,1]=\lim_{\ell\to+\infty} \frac{1}{\ell}I_\ell$. 
    \end{remark}
\begin{remark}
Clearly, the scaling bound of random waves is also a scaling bound in the sense of \cref{def:scalim}, with parameters
    \be 
\mathsf{P}_\ell=\tyu E^n_{\mathsf U(\a,1)}, \sigma_\ell=\sqrt{\ell^{{n-1}}|I_\ell|}, \lambda_\ell=\ell, 0,\rho_\ell\uyt,
    \ee
so that the translation invariant limit field is $F=\phi_{\mathsf{U}(\a,1)}^{\R^n}$.
In particular, note that:
    \be 
\sigma_\ell^2=\ell^{n-1}\eta_\ell=\ell^{n-1}((1-\a)\ell-\beta_\ell).
    \ee
    \end{remark}
The following holds.
\begin{lemma}\label{lem:lambda_ell_RRW}
In the settings of \cref{def:scalimRW},
\bega 
\sigma\tyu \phi^M_{I_\ell}\uyt^2
&=\ell^{n-1}\eta_\ell\tyu \frac{s_{n-1}}{(2\pi)^n}\frac{1-\a^n}{n(1-\a)}  +\Oh_1\tyu\delta_{I_\ell}^M\uyt\uyt,
\\
\lambda\tyu \phi^M_{I_\ell}\uyt^2
&=\ell^2\tyu \frac{n}{n+2} \cdot \frac{1-\a^{n+2}}{1-\a^n} +\Oh_1\tyu\delta_{I_\ell}^M\uyt\uyt.
\\
\e\tyu \phi^M_{I_\ell} \uyt &\le 8\delta_{I_\ell}^M.
\eega  
\end{lemma}
\begin{proof}
 Combine \cref{eq:compare_lambda_scalim}, \cref{lem:USEFULREL} and \cref{lem:lambdasigma_Eucli_RRW}.
\end{proof}
    {
\subsubsection{Meaning of the scaling bound}
We recall that, by definition, if $I_\ell=[\ell-\eta_\ell,\ell]$ then $|I_\ell|=\eta_\ell$ and we have that
\be 
\frac{1}{\ell^{n-1}\eta_\ell}E_{[\ell-\eta_\ell,\ell]}^M(x,y):=\frac{1}{\ell^{n-1}\eta_\ell}\sum_{i:\lambda_i\in [\ell-\eta_\ell,\ell]}\f_i(x)\f_i(y),\quad \forall x,y\in M,
\ee
is the covariance of a renormalized field, having approximately constant average variance:
\be \label{eq:constavgvar}
\frac{\#\{i:\lambda_i\in [\ell-\eta_\ell,\ell]\}}{\ell^{n-1}\eta_\ell}=\fint_M \frac{1}{\ell^{n-1}\eta_\ell}E_{[\ell-\eta_\ell,\ell]}^M(x,x) \dd x \sim E_{\mathsf{U}(\a,1)}^n(0),
\ee
due to Weyl's law, indeed $\ell^n-(\ell-\eta_\ell)^n\sim n \ell^{n-1}\eta_\ell$. The scaling bound compares the above covariance with that of the Euclidean analogue in the case of a fixed interval, that is, the constant variance field, having covariance
\begin{gather}
E_{\mathsf{U}(a,b)}^n(s)=(b-a)^{-1}E_{[a,b]}^n(s){=(b-a)^{-1}(E^n_{[0,b]}(s)-E^n_{[0,a]}(s))}
\\
= \frac{1}{(2\pi)^n}\fint_{a}^b \int_{\lambda\S^{n-1}}e^{is\xi_1}\dd \xi \dd\lambda
\\
=
\fint_{a}^b \frac{\lambda^{n-1}}{(2\pi)^{\frac{\n}2}}\frac{J_{\frac{\n-2}{2}}(\lambda s)}{(\lambda s)^{\frac{\n -2} 2}}\dd \lambda
=
\frac{1}{b-a}\qwe \frac{1}{(2\pi)^{\frac n2}}\frac{J_{\frac{n}{2}}(b s)}{(b s)^{\frac{n} 2}}-\frac{1}{(2\pi)^{\frac n2}}\frac{J_{\frac{n}{2}}(a s)}{(a s)^{\frac{n} 2}}
\ewq
\end{gather}
where $\fint_a^b=(b-a)^{-1}\int_a^b$ and, in case $b=a$, denotes the evaluation at $\lambda=a$ (i.e., $\fint_a^a=\delta_a$). 
The \emph{remainder} in such comparison is the difference between the above two functions: 
\bega \label{eq:remainder}
R^M_{\a,[\ell-\eta_\ell,\ell]}(x,y)
&=\frac{1}{\ell^{n-1}\eta_\ell}E_{[\ell-\eta_\ell,\ell]}^M(x,y)-E_{\mathsf{U}(\a,1)}^n(\ell \theta )
\eega 
where $\theta:=\dist xy$.
The error in the scaling bound $\delta^M_{[\ell-\eta_\ell,\ell]}$ measures the size of the remainder, uniformly with respect to all derivatives up to order $k$ (where each differentiation is divided by $\ell$) and to all pairs of points $x,y\in M$ at distance at most $\theta \le \rho_\ell$, for different choices of sequences $\rho_\ell$ (our main focus is $\rho_\ell=const$, while most common versions of the scaling limit consider the near-diagonal behavior: $\rho_\ell=\Oh(\ell^{-1})$).
}
\subsubsection{One way to interpret the rescaling?}
Let us denote by $\ell M$ the manifold with Riemannian metric {$g_{\ell M}=\ell^2 g_M$}. Then, one can easily show that $\mathrm{dist}_{\ell M}(x,y)=\ell \mathrm{dist}_{M}(x,y)$. Moreover, $\Delta_{\ell M}=\ell^2\Delta_M$, $\lambda_i^{\ell M}=\ell \lambda_i^M$
 and $\f_i^{\ell M}=\ell^{-\frac n2}\f_i^M$. Therefore, 
\begin{proposition}
    \be 
\frac{1}{\ell^n}E^M_I(x,y)=E^{\ell M}_{\ell I}(x,y).
    \ee
\end{proposition}
It follows that we can rewrite \cref{eq:remainder} as 
\bega \label{eq:remainder2}
R^M_{\a,[\ell-\eta_\ell,\ell]}(x,y)
&=\frac{\ell}{\eta_\ell}E_{[1-\frac{\eta_\ell}{\ell},1]}^{\ell M}(x,y)-E_{\mathsf{U}(\a,1)}^n\tyu \mathrm{dist}_{\ell M}(x,y)\uyt \quad \tyu = R^{\ell M}_{\a,[1-\frac{\eta_\ell}{\ell},1]}(x,y)\uyt,
\eega 
where $\frac{\ell}{\eta_\ell}E_{[1-\frac{\eta_\ell}{\ell},1]}^{\ell M}(x,y)$ is the covariance of a field having approximately constant average variance computed exactly as in \cref{eq:constavgvar}. 

The error in the scaling bound $\delta^M_{[\ell-\eta_\ell,\ell]}$ measures the size of the remainder, as a function on $\ell M$, uniformly with respect to all derivatives up to order $k$ and to all pairs of points $x,y\in \ell M$ at distance at most $\ell \rho_\ell$. In other words it measures the difference in the comparison:
\be 
\tyu \ell M, \sqrt{\frac{\ell}{\eta_\ell}}\phi^{\ell M}_{[1-\frac{\eta_\ell}{\ell},1]} \uyt \sim \tyu \R^n, \phi^{\R^n}_{[\a,1]}\uyt,
\ee
in a $\mC^k$ sense, at \emph{large} scale $\ell \rho_\ell$. We shall think of this as global comparison, not restricted to a compact ball, because in the following we will require $\rho_\ell=const$, hence $\ell \rho_\ell\to +\infty$. 

\subsubsection{Connection with the usual scaling limit.}
We note that setting $\rho_\ell = \ell^{-1}$ recovers the standard local scaling limit traditionally employed in the study of random waves, see \cite{igorsurvey}. In that classical framework, the analysis is confined to a small neighborhood of order $\mathcal{O}(\ell^{-1})$ on the original manifold $M$, which translates to a fixed, compact domain of radius $\ell \rho_\ell = 1$ on the manifold $\ell M$. By contrast, our setting requires $\rho_\ell = \text{const}$, allowing $\ell \rho_\ell \to +\infty$. This choice represents a significant generalization, as it forces the stochastic comparison between $(\ell M, \phi^{\ell M})$ and the flat Euclidean model $(\mathbb{R}^n, \phi^{\mathbb{R}^n})$ to hold globally on macroscopically large scales.
\subsection{The scaling bound in the Euclidean case}
In the Euclidean case, with $\theta=|x-y|$ the remainder term in the scaling bound is: if $\a=0$ (colorful case)
\be  \label{eq:rema_0}
R^{\R^n}_{0,[\ell-\eta_\ell,\ell]}(x,y)=0;
\ee
if $\a>0$ (and $0\le \beta_\ell\le \a \ell$)
\bega 
R^{\R^n}_{\a,[\ell-\eta_\ell,\ell]}(x,y)
&=\frac{1}{\ell^{n-1}\eta_\ell}E_{[\ell-\eta_\ell,\ell]}^{\R^n}(x,y)-E_{\mathsf{U}(\a,1)}^n(\ell \theta )
\\
&=\frac{1}{\ell^{n-1}\eta_\ell}E_{[\ell-\eta_\ell,\ell]}^{n}(\theta)-E_{\mathsf{U}(\a,1)}^n(\ell \theta )
\\
&=E_{\mathsf{U}(1-\frac{\eta_\ell}{\ell},1)}^{n}(\ell \theta)-E_{\mathsf{U}(\a,1)}^n(\ell \theta )
\\
&=E_{\mathsf{U}(\a-\frac{\beta_\ell}{\ell},1)}^{n}(\ell \theta)-E_{\mathsf{U}(\a,1)}^n(\ell \theta )=
\dots
\eega 
so that for any fixed $\a\in (0,1)$, recalling that $\beta_\ell=o(\ell)$ {and exploiting that $E_{[\a-\frac{\beta_\ell}{\ell},1]}^n(\ell \theta )=E_{[\a-\frac{\beta_\ell}{\ell},\a]}^n(\ell \theta )+E_{[\a,1]}^n(\ell \theta )$}, we have 
\bega \label{eq:rema_a}
\dots &=\frac{\beta_\ell}{\ell}\cdot (1-\a+\frac{\beta_\ell}{\ell})^{-1}\cdot E^n_{\mathsf{U}(\a-\frac{\beta_\ell}{\ell},\a)}(\ell \theta) + \tyu \frac{1-\a}{1-\a+\frac{\beta_\ell}{\ell}}-1\uyt E_{\mathsf{U}(\a,1)}^n(\ell \theta )
\\
&=
\frac{\beta_\ell}{\ell} \cdot \tyu \Oh(1) E^n_{\mathsf{U}(\a-\frac{\beta_\ell}{\ell},\a)}(\ell \theta) + \Oh(1) E_{\mathsf{U}(\a,1)}^n(\ell \theta )\uyt 
\eega
\begin{remark}
The above expression would be much simpler if in the case $0<\a<1$, we took $\tilde{\sigma}_\ell^2=\ell^n(1-\a)$ instead than $\sigma_\ell^2=\ell^{n-1}((1-\a)\ell-\beta_\ell)$, as we do. However, in the end the order in the scaling error does not change. Moreover, when: $\beta_\ell=0$, or when $\a=1$ there is no difference. We make the current choice, to have an elegant general form of the scaling bound.
\end{remark}
Finally, in the most important case: when $\a=1$, and thus $\eta_\ell=\beta_\ell$, we have
\bega \label{eq:rema_1}
R^{\R^n}_{1,[\ell-\eta_\ell,\ell]}(x,y) 
&=
E_{\mathsf{U}(1-\frac{\beta_\ell}{\ell},1)}^{n}(\ell \theta)-E_{1}^n(\ell \theta )
\\
&=
\fint_{1-\frac{\beta_\ell}{\ell}}^1 E_{1}^n(\lambda \ell \theta ) \lambda^{n-1}\dd \lambda-E_{1}^n(\ell \theta ).
\eega

The sequence $m_\ell$ satisfies $\sqrt{\frac{\beta_\ell}{\ell}}\le m_\ell\le 1$, depending essentially on the behavior of $\beta_\ell \rho_\ell$. It turns that in dimension $1$ and $2$ the error in the Euclidean scaling bound depends on $m_\ell$.
\begin{lemma}\label{lem:euclidWeyl}
Let $\a\in [0,1]$ and $\beta_\ell=o(\a \ell)$ as above. 
If $M=\R^n$, for any $k\in \N$, the error $\delta^{\R^n}_{[\ell-\eta_\ell,\ell]}=\delta^{\R^n}_{[\ell-\eta_\ell,\ell]}(\a,\mC^k,[0,\rho_\ell])$ in the scaling bound \eqref{eq:scalimRW} satisfies
\begin{gather}
\delta_{[\ell-\beta_\ell,\ell]}^{\R^n}
={\Oh(1)\begin{cases}
    \frac{\beta_\ell}{\ell} & n\ge 3
    \\
    m_\ell \sqrt{\frac{\beta_\ell}{\ell}} 
    & n =2 
    \\
    m_\ell^2& n =1
\end{cases}
}
\\
\delta_{ [\a\ell-\beta_\ell,\ell]}^{\R^n}=O\tyu \frac{\beta_\ell}{\ell}
\uyt,\  \text{if $\a\in [0,1)$}.\label{eq:Eucdelta_a}
\end{gather}
in bounded range $0\le \rho_\ell \le \mathrm{diam}(M)$, as $\ell\to+\infty$. In particular, if $n\ge 2$, then $\delta_{[\ell-\beta_\ell,\ell]}^{\R^n}=\oh(1)$.
\end{lemma}
\begin{proof}
We explain the case $k=0$; since the function $\mathrm{dist}^2$ is smooth on $M\times M$ and thus have bounded derivatives, the same argument works for any $k$ at the only cost of a heavier notation. 

Let us denote $\e=\frac{\beta_\ell}{\ell}\in [0,\a]$, $s=\ell\dist xy \in [0,+\infty)$ and $K(u):=E_1^n(u)$. By \cref{lem:bessympt}, the latter satisfies: 
\be \label{eq:beryineq}
K(
s)=\frac{1}{(2\pi)^{\frac{\n}2}}\frac{J_{\frac{\n-2}{2}}(s)}{s^{\frac{\n -2} 2}},\quad  \text{and} \quad |K^{(a)}(s)|\le C(1+s)^{-\frac{\n-1}{2}},\ \forall a\in \N.
\ee
Recalling \cref{eq:genRemainder}, we have to study the behavior as $\e\to 0$, of the quantity:
\bega\label{eq:quantity}
{E^n_{\mathsf U(\a-\frac{\beta_\ell}{\ell},\a)}}\tyu \ell \dist xy\uyt 
&=
{E^n_{\mathsf U(\a-\e,\a)}}\tyu s\uyt 
=
\frac{1}{\e}\int_{\a-\e}^\a {E_\lambda^n}(s)\dd \lambda = \frac{1}{\e}\int_{\a-\e}^\a {\lambda^{n-1}} K(\lambda  s)\dd \lambda;
\eega
precisely, our goals are: if $\a=1$, to compare the above with its limit for fixed $s$:
\be 
E_1(\ell \dist xy)= K(s).
\ee
If $\a\in [0,1)$ (recall \cref{eq:rema_1}), we want to show that \eqref{eq:quantity} is bounded, uniformly for all $\a \in [0,1)$, all $0<\e<\a$ and $s>0$. Since $K$ is bounded, the latter is automatic, so \cref{eq:Eucdelta_a} is proven. Let us fix $\a=1$, from this point on.

By Lagrange Theorem, by seeing the following as a function of $\e$, we have that there exists a function $\e\mapsto \lambda_\e\in [1-\e,1]$ such that
\bega 
I(\e,s):=\frac{1}{\e}\int_{1-\e}^1 {\lambda^{n-1}} K(\lambda  s)\dd \lambda-K(s) &=\lambda_\e^{n-1}K(\lambda_\e s)-K(s)
\eega
We can bound the above quantity as follows
\bega 
|I(\e,s)|
&=
\left|\lambda_\e^{n-1}\tyu K(\lambda_\e s)-K(s)\uyt + \tyu\lambda_\e^{n-1}-1\uyt K(s)\right|
\\
&\le 
\left| K(\lambda_\e s)-K(s)\right| + \left|\tyu (1-\e)^{n-1}-1\uyt K(s)\right|
\\
&=
\left| \int_{\lambda_\e}^1 K'(\lambda s)s \dd\lambda\right| + \left|\tyu (1-\e)^{n-1}-1\uyt \right| \left|K(s)\right|
\\
&\le
\tyu 
\frac{ s+1}{(1+s/2)^{\frac{n-1}{2}}}
\uyt \Oh(\e)=\Oh(\e)(1+s)^{\frac{3-n}{2}}
\eega
where we used \cref{eq:beryineq} and {assumed that $\e\le \frac12$}. 

To prove the lemma, given $R=\ell \rho_\ell$, we need to study $\delta(\e,R):=\sup_{s\le R}|I(\e,s)|$. From the above inequality we conclude that for all $n\ge 3$ we have $\delta(\e,R)\le \Oh(\e)$, which is what we wanted to prove. For $n\in \{1,2\}$, we need to analyze the first integral more carefully. 

Observe that we have a trivial bound valid for all $\e\in [0,1]$: 
\be 
|I(\e,s)|\le \frac{\Oh( 1)
}{(1+s/2)^{\frac{n-1}{2}}}
\ee 
 If $n=1$, the combination of the two above inequalities give that
\be 
\delta(\e,R)\le \Oh(1)\min\kop \max\{\e,R\e\},1\pok,
\ee
which implies the case $n=1$. 
Now, it remains only to improve such inequality in dimension $n=2$.
We separate the cases $s\le \e^{-1}$ and $s\ge \e^{-1}$, so that we have that 
\be 
\sup_{s\le \e^{-1}}|I(\e,s)|\le \Oh(\e)(1+1/\e)^{\frac12}=\Oh(1)\sqrt{\e}, 
\ee 
and 
\be 
\sup_{\e^{-1}\le s \le R}|I(\e,s)|\le \Oh(1) \frac{\Oh(1)}{\tyu1+\e^{-1}\uyt^{\frac{1}{2}}}\le \Oh(1)\sqrt{\e}
\ee
From this we deduce that $\delta(\e,R)\le \Oh(\sqrt{\e})$, so that combining with the first above inequality, we conclude that 
\be 
\delta(\e,R)\le \Oh(1)\min\kop \max\{\e, \e \sqrt{R}\}, \sqrt{\e}\pok
\ee
This concludes the $2$-dimensional case and thus the proof of the lemma.
\end{proof}
\begin{remark}
    If $n=1$, then $K(s)=E^1_1(s)=\cos(s)$, so that 
    \be 
I(\e,s)=\frac{\sin(s)-\sin(s-s\e)}{s\e}-\cos(s)=\Oh(s\e),
    \ee
 is bounded from below along almost all curves $s\e=const $, which corresponds to the case $\beta_\ell \rho_\ell = const$. Hence, the bound of the theorem cannot be improved. In dimension $n=1$, the scaling limit holds if and only if $s\e\approx\beta_\ell\rho_\ell \to 0$.
\end{remark}
\begin{remark}\label{rem:sharp_scabo2}
    If $n=2$, then by classical Bessel functions asymptotics (see \cite[Theorem 4, Equation 6]{Krasikov_2014}, or \cite[Lemma 2.5]{GMT}), one can see that as $t\to +\infty$, $K(t)\sim c_2\frac{\cos(t-\phi_2)}{\sqrt{t}}$, for appropriately chosen constants $c_2,\phi_2$, so that 
    \be 
I(\e,s)\simeq \frac{1}{\sqrt{s}}\int_{1-\e}^1\sqrt{\lambda} \cos\tyu \lambda s-\phi_2\uyt - \cos\tyu s-\phi_2\uyt \dd \lambda. 
    \ee
    One can see that the above quantity is of order $\frac{1}{\sqrt{s}}\simeq \sqrt{\e}$, along subsequences of the form (for instance) $s\e=\frac{\pi}{2}$ with $s=\phi_2+2k\pi+\frac{\pi}{2}$.
  This, again, corresponds to the case $\beta_\ell \rho_\ell = const$. Hence, the bound of the theorem cannot be improved. Indeed, in dimension $n=2$, the theorem implies that the scaling limit error is of order $\oh(\sqrt{\e})=\oh (\sqrt{\frac{\beta_\ell}{\ell}})$, if  $s\e\approx\beta_\ell\rho_\ell \to 0$.
\end{remark}
\subsection{The scaling error of general random waves: proof of \cref{thm:OdeltaRWintro}}
We will report a general criterion to deduce the scaling bound or scaling limit for general choice of intervals $I_\ell$. As in most instances appearing in the literature, the analysis starts by taking the difference in \cref{eq:horlaw}:
\bega 
E^M_{[\ell-\eta_\ell,\ell]}{}&=E^M_{[0,\ell]}{}-E^M_{[0,\ell-\eta_\ell]}{}
\\
&=
E^n_{[\ell-\eta_\ell,\ell]}\circ \mathrm{dist} 
+
\ell^n\tyu R^M_{[0,\ell]}{}-R^M_{[0,\ell-\eta_\ell]}{}\uyt
\\
&=
\ell^{n-1}\eta_\ell\tyu  
E_{\mathsf U(\a,1)}
+R^{\R^n}_{\a,[\ell-\eta_\ell,\ell]}\circ (0,\mathrm{dist}) 
+
\frac{\ell}{\eta_\ell}\tyu R^M_{[0,\ell]}{}-R^M_{[0,\ell-\eta_\ell]}{}\uyt \uyt
 \eega
Here, $(0,\mathrm{dist})$ is the function $(x,y)\mapsto (0,\mathrm{dist}(x,y))$. Therefore, the general remainder function is:
\bega \label{eq:genRemainder}
R^M_{\a,[\ell-\eta_\ell,\ell]}
&=
R_{\a,[\ell-\eta_\ell,\ell]}^{\R^n} \circ (0,\mathrm{dist})+\frac{\ell}{\eta_\ell}\tyu R^M_{[0,\ell]}{}-R^M_{[0,\ell-\eta_\ell]}{} \uyt, 
\eega
Here, $R_{\a,[\ell-\eta_\ell,\ell]}^{\R^n} \circ (0,\mathrm{dist})$ is controlled as in \cref{lem:euclidWeyl}. 
The second part is controlled in terms of $\delta_{[0,\ell]}^M$, by triangular inequality. 

\begin{lemma}\label{thm:OdeltaRW}
The error $\delta^M_{[\ell-\eta_\ell,\ell]}$ in the scaling bound \eqref{def:scalimRW} in bounded range $\rho_\ell$, satisfies
\be 
\delta_{[\ell-\eta_\ell,\ell]}^M=
\delta_{[\ell-\eta_\ell,\ell]}^{\R^n}
+ \Oh\tyu \frac{\ell}{\eta_\ell}\delta^M_{[0,\ell]}\uyt,
\ee
as $\ell\to +\infty$, for all $\a\in [0,1]$. 
\end{lemma}
We recall that here $I_\ell=[\ell-\eta_\ell,\ell]$ is as in \cref{eq:Il}, so that $\eta_\ell=(1-\a)\ell+\beta_\ell$ and that $\beta_\ell=\oh(\ell)$.
\begin{proof}
Combining \cref{lem:euclidWeyl} and \cref{eq:genRemainder}, the statement follows immediately.
\end{proof}
\begin{proof}[Proof of \cref{thm:OdeltaRWintro}]
Combine \cref{lem:euclidWeyl} and \cref{thm:OdeltaRW} and recall that in general $\delta_{[0,\ell]}^M=\Oh(\ell^{-1})$ by \cref{thm:hoermplaw}.
\end{proof}
\subsection{Second chaos: Berry cancellation}
Let us consider a random wave sequence $\phi_\ell=\phi_{I_\ell}^M$, with $I=I_\ell=[\ell-\eta_\ell,\ell]$ defined as in \cref{eq:Il} and consider the scaling bound \cref{eq:scalimRW}, with effective variance $\sigma_\ell=\ell^{n-1}\eta_\ell$, effective frequency $\lambda_\ell=\ell$ and limit field $F_\a:=\phi_{\mathsf U(\a,1)}^{\R^n}$ as in \cref{def:RWUI}. We recall that the parameters of the field defined as in \cref{def:three} (i.e. \cite[Def. 1.11]{cgv2025StecconiTodino}), are given in \cref{eq:phiI}. 
By \cref{prop:addendum}, we can also deduce that they are related with the parameters in the scaling limit by the following relations:
\be 
\|\frac{\sigma(\phi_{I_\ell}^M)}{\ell^{n-1}\eta_\ell \sigma(F_\a)}-1\|\le \delta_{I_\ell}^M,\quad \|\frac{\lambda(\phi_{I_\ell}^M)}{\ell\lambda(F_\a)}-1\|\le \delta_{I_\ell}^M, \quad \e(\phi_{I_\ell}^M)\le 8\delta_{I_\ell}^M.
\ee
The first relation, essentially corresponds to Weyl's law: $\sigma(\phi_{[0,\ell]}^M)\sim \ell^{n}\frac{b_n}{(2\pi)^n}$.

We recall a result from \cite{cgv2025StecconiTodino} in our language (namely, in terms of \cref{def:Y}). It allows us to control the second chaos term in \cref{cor:scaberry}. We state the theorem for a manifold of arbitrary volume.
\begin{corollary}[\texorpdfstring{\cite[Corollary 2.1]{cgv2025StecconiTodino}}{}] \label{cor:berrycanc} Let $I$ a bounded interval, and let $N=N_I^M, \lambda=\lambda(\phi_I^M),\e=\e(\phi_I^M)$ as above. 
Then,
    \be\label{eq:RWvar} 
\Var\tyu Y_M(I) [2]\uyt= \frac14  \cdot 
\frac{1}{N}\qwe 
\tyu\frac{1}{N}\sum_{i:\lambda_i\in I}(\lambda_i^2-\lambda^2)^2\uyt\frac{1}{\lambda^4}+O_\n(\e)
\ewq
    \ee
Here, $O_n(\e)$ denotes a quantity that is bounded by a dimensional constant times $\e$.
\end{corollary}
\begin{proof}
The theorem coincides with \cite[Corollary 2.1]{cgv2025StecconiTodino} if $\vol{}(M)=1$. To deduce the general case, one can observe that since $c^2\Delta_{cM}=\Delta_M$, we have that $Y_{cM}(I)=Y_M(cI)$, that $N_{I}^{cM}=N_{cI}^M$ and that $\lambda_i$ and $\lambda$ change by the same constant factor, so that the expression is invariant under such rescalings.
\end{proof}
We deduce the following, also from \cite{cgv2025StecconiTodino}.
\begin{corollary}\label{cor:wMRWbercanc}
For any sequence $\eta_\ell\le \ell$, let $\e_\ell=\e(\phi_{[\ell-\eta_\ell,\ell]}^M)$. Then, we have
    \be\label{eq:RWvarmonochaos2} 
\Var\tyu Y_M([\ell-\eta_\ell,\ell])[2]\uyt
\le
\frac{1}{\ell^{n-1}\eta_\ell}\tyu \Oh\tyu \frac{\eta_\ell}{\ell}\uyt^2 +\Oh\tyu \e_\ell \uyt\uyt
    \ee
    and if $1-\a=\liminf \frac{\eta_\ell}{\ell} >0$, then
    \be\label{eq:RWvarfull_chaos2} 
\Var\tyu Y_M([\ell-\eta_\ell,\ell])[2]\uyt
=\Theta\tyu \ell^{-n}\uyt,
    \ee
    as $\ell \to +\infty$.
\end{corollary}
\begin{proof}
The first statement was proven in \cite{cgv2025StecconiTodino}. We report the proof, as it is very short.
In this case $\lambda(\phi_\ell)=\ell+o(\ell)$ and the term corresponding to the variance of the ensemble $(\lambda_i^2)_{i:\lambda_i\in I}$, is bounded by $(\eta_\ell^2+2\eta_\ell \ell)^2=o(\ell^2\eta_\ell^2)$. Finally, note that Weyl's law and \cref{eq:phiI} yield that $N=\Theta(\ell^{\n-1}\eta_\ell)$.
This implies also the upper bound in the second statement, {exploiting that $N=\Theta (\ell^n)$} in that case. 

For the lower bound in the second statement, observe that we have $\delta_\ell=\Oh(\ell^{-1})$ by \cref{prop:chk}. Thus, applying \cref{lem:lambda_ell_RRW}, we deduce that $\e_\ell=\Oh(\ell^{-1})$ and that $ \lambda^2=c_\a^2\ell^2(1+\oh(\frac 1\ell))$, where $\a<c_\a-a<c_\a<c_\a+a<1$, for some $a>0$. It follows that 
\begt 
\tyu\frac{1}{N}\sum_{i:\lambda_i\in [\ell-\eta_\ell,\ell]}(\lambda_i^2-\lambda^2)^2\uyt\frac{1}{\lambda^4}= \frac{1}{N}\sum_{i:\lambda_i\in [\a\ell-\beta_\ell,\ell]}\tyu \frac{\lambda_i^2}{\lambda^2}-1\uyt^2 
\\
\ge \frac{1}{N}\sum_{i:\lambda_i\in [(c_\a+a)\ell,\ell]}\tyu \frac{(c_\a+a)^2\ell^2}{\lambda^2}-1\uyt^2
+
\frac{1}{N}\sum_{i:\lambda_i\in [\a\ell-\beta_\ell,(c_\a-a)\ell]}\tyu 1-\frac{(c_\a-a)^2\ell^2}{\lambda^2}\uyt^2
\\
\ge \frac{1}{N}\sum_{i:\lambda_i\in [(c_\a+a)\ell,\ell]}\tyu \frac{(c_\a+\frac12 a)^2}{c_\a^2}-1\uyt^2
+
\frac{1}{N}\sum_{i:\lambda_i\in [\a\ell-\beta_\ell,(c_\a-a)\ell]}\tyu 1-\frac{(c_\a-\frac12 a)^2}{c_\a^2}\uyt^2
\\
=\Theta(1)\cdot\qwe 
\tyu \frac{(c_\a+\frac12 a)^2}{c_\a^2}-1\uyt^2(1-(c_\a+a)^n)
+
\tyu 1-\frac{(c_\a-\frac12 a)^2}{c_\a^2}\uyt^2((c_\a-a)^n-\a^n)
\ewq
\\
=\Theta(1).
\eegt
where we used Weyl Law in the last step.
We conclude, from \cref{cor:berrycanc}, that 
\be 
\Var\tyu Y_M(I) [2]\uyt= \frac14  \cdot 
\frac{1}{N}\qwe 
\Theta(1)+\Oh(\e_\ell)
\ewq
=\Theta(\ell^{-n}).
\ee
\end{proof}
\subsection{Main results on random waves: Proof of  \cref{thm:ratefall} and \cref{thm:ape}}
We use the results of the previous section to apply \cref{cor:scaberry}. 
\begin{proof}[Proof of \cref{thm:ratefall}]
As a preliminary observation, let us note that 
\be \label{eq:dafuck}
\int_{\substack{x,y\in M,
\\
\frac{r_\ell}{\ell}
\le \mathrm{dist}(x,y)\le \rho}}\|(\ell^{n-1}\eta_\ell)^{-1}j^{1,1}_{x,y}E_{[\ell-\eta_\ell,\ell]}\|^2_{g^{\ell}} \dd x \dd y =\Oh\tyu \frac{1}{\ell^{n-1}\eta_\ell}\uyt,
\ee
due to {since $\int_{M\times M} ||jE||^2=O(N_I^M)=O(\ell^{n-1}\eta_\ell)$}.
Consider the $\mC^3$ scaling limit \cref{def:scalimRW}, with $\a=1$, and $(\rho_\ell',\rho_\ell)=(0,\rho)$. The first inequality is obtained by applying \cref{thm:scasca} to the sequence $\phi_\ell=\phi^M_{[\ell-\beta_\ell,\ell]}$ and using \cref{cor:berrycanc} for the second chaos. We get the following. 
 \begin{gather}\label{eq:RWvarmonoproof} 
\Var\tyu Y_M([\ell-\eta_\ell,\ell])\uyt
\le
\frac{1}{\ell^{n-1}\eta_\ell}\tyu \Oh\tyu \frac{\eta_\ell}{\ell}\uyt^2 +\Oh\tyu \e_\ell \uyt\uyt + \Oh(\ell^{-n}r_\ell^n) \\
+ \Oh(1)\int_{\substack{x,y\in M,
\\
\frac{r_\ell}{\ell}
\le \mathrm{dist}(x,y)\le \rho}}\|(\ell^{n-1}\eta_\ell)^{-1}j^{1,1}_{x,y}E_{[\ell-\eta_\ell,\ell]}\|^4_{g^{\ell}} \dd x \dd y=\dots
\end{gather}
Using \cref{cor:decor}, with $\mathfrak{D}_\ell(r_\ell)\le \delta_\ell$ and $\|j^4_rK\|=\Oh(r_\ell^{\frac{1-n}{2}})$ by \cref{lem:bessympt}, and then \cref{eq:dafuck}, we get:
\be 
\|(\ell^{n-1}\eta_\ell)^{-1}j^{1,1}_{x,y}E_{[\ell-\eta_\ell,\ell]}\|^4_{g^{\ell}} \le 
\max\kop
\|(\ell^{n-1}\eta_\ell)^{-1}j^{1,1}_{x,y}E_{[\ell-\eta_\ell,\ell]}\|^2_{g^{\ell}} \delta_\ell^2, (1+\ell\mathrm{dist}(x,y))^{\frac{1-n}{2}\cdot 4} \pok
\ee
with $r_\ell=\Oh(1)$,
\begt 
\int_{\substack{x,y\in M,
\\
\frac{r_\ell}{\ell}
\le \mathrm{dist}(x,y)\le \rho}}
(1+\ell\mathrm{dist}(x,y))^{\frac{1-n}{2}\cdot 4} \dd x \dd y
\le \Oh(1)\int_{\frac{r_\ell}{\ell}}^{\rho} (1+\ell t)^{2-2n} t^{n-1} dt
\\
\le \Oh(\ell^{2-2n})\int_{\frac{r_\ell}{\ell}}^{\rho} t^{1-n} dt
\le \Oh(1)\begin{cases}
    \frac{\log \ell}{\ell^2}, \quad \text{if $n=2$,}
    \\
\ell^{-n}, \quad \text{if $n\ge 3$,}
\end{cases}=\ell^{-n}a(\ell)
\eegt
Then the full variance is controlled as:
\begt 
\dots\le \frac{1}{\ell^n}\tyu \Oh(r_\ell^n)+\frac{\ell}{\eta_\ell}\Oh(\e_\ell+\delta_\ell^2)+a(\ell)\uyt.
\eegt
    \end{proof}
\begin{proof}[Proof  of \cref{thm:ape}]
For $\a=1$, \cref{thm:ape} follows from \cref{thm:ratefall}, for general manifolds, while  for manifolds without conjugated pairs, the desired bound follows by substituting $\delta_{[0,\ell]}^M=O(\frac{1}{\ell\log \ell})$ in (\ref{bo1}). 

It remains only to show the case $\a<1$. \cref{thm:ratefall} is irrelevant in such setting.
The lower bound is provided by the one on the second chaos, see \cref{cor:wMRWbercanc}. The upper bound follows by \cref{prop:addendum}. Both are of order equal to
    \be 
\int_{\substack{x,y\in M,
\\
\frac{r_\ell}{\ell}
\le \mathrm{dist}(x,y)\le \rho}}\|(\ell^{n})^{-1}j^{1,1}_{x,y}E_{[\ell-\eta_\ell,\ell]}\|^2_{g^{\ell}} \dd x \dd y =\Theta (\ell^{-n}).
    \ee
\end{proof}

\begin{appendix}

\section{Some properties of the Bessel functions}
\begin{lemma}\label{lem:bessympt} For any $u\in \R^\n$, we have
\be \label{eq:Jhorm}
K_{\horm;n}(
u)=\fint_{\B^{n}} e^{i\langle u , \xi \rangle} \dd \xi=\frac{(2\pi)^{\frac{\n}2}}{b_{\n}}\frac{J_{\frac{\n}{2}}(|u|)}{|u|^{\frac \n 2}},\quad  \text{and} \quad |\nabla_u^\a K_{\horm;n}(u)|\le C(1+|u|)^{-\frac{\n+1}{2}},\ \forall \a\in \N^\n;
\ee    
\be \label{eq:Jberry}
K_{\berry;n}(
u)=\fint_{\S^{n-1}} e^{i\langle u , \xi \rangle} \dd \xi=\frac{(2\pi)^{\frac{\n}2}}{s_{\n-1}}\frac{J_{\frac{\n-2}{2}}(|u|)}{|u|^{\frac{\n -2} 2}},\quad  \text{and} \quad |\nabla_u^\a K_{\berry;n}(u)|\le C(1+|u|)^{-\frac{\n-1}{2}},\ \forall \a\in \N^\n;
\ee    
where $(J_\nu)_{\nu\in \R}$ are the Bessel functions of the first kind and $C=C(n,\a)>0$ is a constant depending only on $\a$ and $\n$.
\end{lemma}
\begin{proof}
   
When $\a=0$, it follows
 \be   J_{\frac{\nu}2}(t)\le \frac{C_\nu 
    }{\sqrt{(1+|t|)}},
    \ee
    
    for $\nu \in \N$.
    This can be seen from the asymptotic behavior of the Bessel functions:
\bega\label{eq:porcamad}
J_\nu(t)&=(\frac{2}{\pi t})^{1/2}\cos (t-\nu\frac{\pi}{2}-\frac{\pi}{4})+O(t^{-3/2}) \quad \mbox{ as } t \to \infty,\\
J_\nu(t)&=\frac{(t)^{\nu}}{2^\nu \nu!}
+O(t^{\nu+2}) \quad \mbox{ as } t \to 0,
\eega
  see for instance \cite[Eq. 9.2.1 and 9.1.7]{Abramowitz}
   . For general $\a$ (we only need $|\a|\le 2$), the recursive formula (\cite[Eq. 9.1.27]{Abramowitz})
$$J'_\nu(x)=\frac{1}{2} [ J_{\nu-1}(x)-J_{\nu+1}(x)]$$ leads to the same bound with $C=C(\nu,\alpha)$.    
\end{proof}

\section{Proof of \cref{cor:berrycanc2}}
\begin{theorem}\label{thm:nonzero} Let $(M,g)$ be a compact Riemannian manifold of dimension $n\ge 2$ without conjugate points. Let us consider the monochromatic case: $\phi_\ell:=\phi_I^M$, with $I=[\ell-\eta_\ell, \ell]$, where $\eta_\ell=\oh(\ell)$. Assume in addition that: if the dimension is $\d=2$, then $\eta_\ell\log\ell=\Oh(\ell)$. Then, we have
      \be
\Var\tyu \frac{\vol{n-1}(\{x:\phi_\ell(x)=t\sigma_\ell(x)\})[2]}{\lambda_\ell}\uyt\ge\Theta\tyu \frac{1}{N^M_{[\ell-\eta_\ell,\ell]}}\uyt,
    \ee
where $\sigma_\ell(x)^2:=\Var(\phi_I^M(x))$ and $\lambda_\ell:=\lambda\tyu \phi_{[\ell-\eta_\ell,\ell]}^M\uyt$.
\end{theorem}

\begin{proof}
Let $\e=\e(\phi_{[\ell-\eta_\ell,\ell]}^M)$ be as in \cref{def:three}. Note that, by \cref{thm:ape}, we have that $\e = \oh(1)$ as $\ell \to \infty$. Moreover, notice that \cref{cor:berrycanc_intro} applies in this setting.

Let us define 
$$ 
f(x):=\frac{\phi_I^M(x)}{\sqrt{\Var(\phi_I^M(x))}}.
$$

For any $x\in M$, denote by $S_xM :=S(T_xM)= \{u \in T_xM : \|u\|_g = 1\}$ the unit tangent sphere at $x$.
From \cite[eq. (3.34)]{cgv2025StecconiTodino}, for $x\in M$ and $u\in S_xM$, we write the projection onto the second Wiener chaos as
\begin{align}
\hat{\m L}_{f-t}(x,u)[2]&:=
\sum_{a+b=1}
e^{-\frac{t^2}{2}}\frac{H_{2a}(t)}{H_{2a}(0)}
\frac{\coeff(a,b)}{s_\n\sqrt{\n}} H_{2a}(f(x))H_{2b}\tyu \frac{\langle \nabla_xf, u\rangle}{\|\Lambda_x u\|} \sqrt{\n}\uyt \left\|\Lambda_x u\right\|\\
&=
\frac{e^{-\frac{t^2}{2}}}{2}\frac{1}{s_\n\sqrt{\n}} \qwe
 H_{2}(t)
H_{2}(f(x)) \left\| u\right\|_{g^f}+ 
 H_{2}\tyu \frac{\langle \nabla_xf , u\rangle}{\|u\|_{g^f}} \sqrt{\n}\uyt \left\|u\right \|_{g^f} \ewq \\
&=
e^{-\frac{t^2}{2}} \qwe
\frac{t^2}{2s_\n\sqrt{\n}}
H_{2}(f(x)) \|u\|_{g^f}+ \hat{\mathcal{L}}_{f}(x,u)[2] \ewq, \label{eq}
\end{align}
where in the first passage $\Lambda_x$ is defined in \cite[Sec. 3.4, def. 3.6]{cgv2025StecconiTodino} and gives $\|\Lambda_x u\|=\sqrt{n}\|u\|_{g^f}$. 

Let us focus on the first term in \cref{eq}. We have

\bega
\E\qwe\frac{t^2}{2s_\n\sqrt{\n}}
H_{2}(f(x)) \left\| u\right\|_{g^f}\frac{t^2}{2s_\n\sqrt{\n}}
H_{2}(f(y)) \left\| v\right\|_{g^f}\ewq &=
\frac{t^4}{4s_\n^2 \n } \E[
H_{2}(f(x))H_{2}(f(y))]\left\| u\right\|_{g^f}\left\| v\right\|_{g^f}\\&
=\frac{t^4}{2s_\n^2 \n } C(x,y)^2\left\| u\right\|_{g^f}\left\| v\right\|_{g^f},
\eega
       where $C(x,y):=\E[f(x)f(y)]$.
It then follows, using \cite[Definition 3.9]{cgv2025StecconiTodino} and the triangular inequality of the form $\|X+Y\|_{L^2}^2 \ge \frac{1}{2}\|X\|_{L^2}^2 - \|Y\|_{L^2}^2$, that

\bega\label{var_level}
   & \Var\tyu \frac{\vol{n-1}(\{x:\phi_\ell(x)=t\sigma_\ell(x)\})[2]}{\lambda_\ell}\uyt   
   =
  \Var\tyu \frac{1}{\lambda_\ell}\int_{M}\int_{S_xM}\hat{\m L}_{f-t}(x,u)[2] dudx\uyt
   \\
&\geq \frac{e^{-t^2}}{\lambda_\ell^2} \bigg\{  \frac{t^4}{4s_\n^2 \n}\int_{M\times M} \int_{S_xM\times S_yM}
 C(x,y)^2 \left\| u\right\|_{g^f} \left\| v\right\|_{g^f} -\E(\hat{\m L}_{f}(x,u)[2]\hat{\m L}_{f}(y,v)[2])\,du\,dv\,dx\,dy\bigg\} . 
\eega
We note that, from \eqref{eq:eps}, $$\bigg|\frac{\sqrt{\E[|d_x\phi_\ell(u)|^2]}\sqrt{n}}{\sigma_\ell \lambda_\ell} -1 \bigg| \leq \varepsilon \quad \Rightarrow \quad \bigg|\frac{\E[|d_x\phi_\ell(u)|^2]\n}{\sigma_\ell^2 \lambda_\ell^2} -1\bigg|\leq 2\varepsilon+\varepsilon^2=O(\varepsilon) $$ 
and thus 
$$\bigg| \frac{n}{\lambda_\ell^2} \bigg|\bigg| \frac{u}{\|u\|_{g}}\bigg|\bigg|^2_{g^f}-1\bigg| \leq 2 \varepsilon+\varepsilon^2 \quad \Rightarrow \quad \bigg| \|u\|_{g^f}^2- \frac{\lambda_\ell^2}{n} \|u\|_{g}^2\bigg| \leq \frac{\lambda_\ell^2}{n} \|u\|_g^2 (2 \varepsilon+\varepsilon^2). $$

Defining $R_x(u):= \|u\|_{g^f}^2-\frac{\lambda_\ell^2}{n}\|u\|_{g}^2$, we have that  $R(u):=\sup_{x\in M,\\u\ne0} \frac{|R_x(u)|}{\|u\|^2_{g}} \leq \frac{\lambda_\ell^2}{n}\varepsilon(2+\varepsilon)$. 
This implies that


\bega
\|u\|_{g^f}&= \sqrt{\frac{\lambda_\ell^2}{n}\|u\|_{g}^2+ R_x(u)}=\frac{\lambda_\ell}{\sqrt{n}} \|u\|_{g}\sqrt{1+\n\frac{R_x(u)}{\lambda_\ell^2 \|u\|_{g}^2}}=\frac{\lambda_\ell}{\sqrt{n}} \|u\|_{g}\tyu 1+O\tyu\frac{R(u)}{\lambda_\ell^2 \|u\|_{g}^2}\uyt\uyt \\ &= \frac{\lambda_\ell}{\sqrt{n}} \|u\|_{g}\tyu 1+O(\varepsilon)\uyt
\eega 
Exploiting this expression in the first term on the r.h.s. in \cref{var_level} we obtain
\bega
&\frac{1}{\lambda_\ell^2}  \frac{e^{-t^2}t^4}{4s_\n^2 \n}\int_{M\times M} \int_{S_xM\times S_yM}
 C(x,y)^2 \left\| u\right\|_{g^f} \left\| v\right\|_{g^f}\,du\,dv\,dx\,dy \\&
=  \frac{e^{-t^2}t^4}{4s_\n^2 \n^2}\int_{M\times M} \int_{S_xM\times S_yM}
 C(x,y)^2 ||u||_{g} ||v||_{g} \left(1+O(\varepsilon)\right)\left(1+O(\varepsilon)\right)\,du\,dv\,dx\,dy\\&
= \frac{e^{-t^2}t^4s_{n-1}^2}{4s_\n^2 \n^2}\int_{M\times M} 
 C(x,y)^2 \left(1+O\left(\varepsilon \right)\right)\,dxdy.
 \eega

Observe (as in \cite[Eq. (2.4)]{cgv2025StecconiTodino}) that the covariance function $C$ of the unit-variance normalization $f$ of a random wave $\phi_I$, is 
\be 
C(x,y)=\frac{(1+R(x,y))}{\sigma(\phi_I^M)^2}\sum_{\lambda_i\in I}\f_i(x)\f_i(y),
\ee where $R(x,y)$ denotes a function with $\|j''_{x,y}R\|_{g}\le \e$. Recalling that 
$$\sigma(\phi_I^M)^2
=\frac{N_I^M}{\vol{n}(M)} $$

the integral of $C(x,y)^2$ gives $\Theta (\frac{\vol{n}(M)}{N_I^M})$.
Finally, since the second term involves 
\begt 
\frac{1}{\lambda_\ell^2} \int_{M\times M}\int_{S_xM \times S_yM} \E(\hat{\m L}_{f}(x,u)[2]\hat{\m L}_{f}(y,v)[2])\,du\,dv\,dx\,dy 
\\
=\E\qwe \left|\lambda_\ell^{-1}\hat{\mathcal{L}}_{f}(x,u)[2]\right|^2\ewq\le \Var \qwe \left|\lambda_\ell^{-1}\hat{\mathcal{L}}_{f}(x,u)\right|^2\ewq =\Var\tyu Y_M([\ell-\eta_\ell,\ell])\uyt,
\eegt
which is proven to be $ \oh\tyu\frac{1}{\sigma(\phi_I^M)^2}\uyt $ by \cref{cor:berrycanc_intro}, this concludes the proof.
\end{proof}

\end{appendix}
\bibliographystyle{abbrv}
\bibliography{Shermite.bib}

\end{document}